\documentclass[11pt]{amsart}

\usepackage[text={14.06cm,22.84cm},centering]{geometry}

\usepackage{color}
\usepackage{esint,amssymb}
\usepackage{graphicx}
\usepackage{mathtools}
\usepackage[colorlinks=true, pdfstartview=FitV, linkcolor=blue, citecolor=blue, urlcolor=blue,pagebackref=false]{hyperref}
\usepackage{microtype}

\usepackage{bm}
\usepackage{scalerel} 
\usepackage{mathrsfs}
\usepackage[font={footnotesize}]{caption}

\usepackage{enumitem}

\usepackage{comment}

\usepackage{scalerel,stackengine}
\stackMath
\newcommand\reallywidehat[1]{%
\savestack{\tmpbox}{\stretchto{%
  \scaleto{%
    \scalerel*[\widthof{\ensuremath{#1}}]{\kern-.6pt\bigwedge\kern-.6pt}%
    {\rule[-\textheight/2]{1ex}{\textheight}}
  }{\textheight}%
}{0.5ex}}%
\stackon[1pt]{#1}{\tmpbox}%
}
\definecolor{darkgreen}{rgb}{0,0.5,0}
\definecolor{darkblue}{rgb}{0,0,0.7}
\definecolor{darkred}{rgb}{0.9,0.1,0.1}
\definecolor{lightblue}{rgb}{0,0.51,1}

\usepackage[english,french]{babel}

\newtheorem{theorem}{Theorem}
\newtheorem{proposition}{Proposition}
\newtheorem{lemma}[proposition]{Lemma}
\newtheorem{corollary}[proposition]{Corollary}

\theoremstyle{remark}
\newtheorem{remark}[proposition]{Remark}

\theoremstyle{definition}
\newtheorem{definition}[proposition]{Definition}

\newtheorem{assumption}[proposition]{Assumption}

\numberwithin{equation}{section}
\numberwithin{proposition}{section}

\newcommand{\Z}{\mathbb{Z}}
\newcommand{\N}{\mathbb{N}}
\newcommand{\R}{\mathbb{R}}

\newcommand{\ep}{\varepsilon}

\renewcommand{\subset}{\subseteq}

\newcommand{\cu}{{\scaleobj{1.2}{\square}}}

\DeclareMathOperator{\dist}{dist}

\DeclareMathOperator{\supp}{supp}

\DeclareMathOperator{\Rel}{Re}

\renewcommand{\bar}{\overline}
\renewcommand{\tilde}{\widetilde}

\begin{document}

\title[Local energy weak solutions in the half-space]{Local energy weak solutions for the Navier-Stokes equations in the half-space}

\author[Y. Maekawa]{Yasunori Maekawa}
\address[Y. Maekawa]{Kyoto University, Department of Mathematics, Kyoto, Japan}
\email{maekawa@math.kyoto-u.ac.jp}

\author[H. Miura]{Hideyuki Miura}
\address[H. Miura]{Tokyo Institute of Technology, Department of 
Mathematical and Computing Sciences, Tokyo, Japan}
\email{miura@is.titech.ac.jp}

\author[C. Prange]{Christophe Prange}
\address[C. Prange]{Universit\'e de Bordeaux, CNRS, UMR [5251], IMB, Bordeaux, France}
\email{christophe.prange@math.u-bordeaux.fr}

\subjclass[2010]{}
\date{\today}

\maketitle

\selectlanguage{english}

\begin{abstract}
The purpose of this paper is to prove the existence of global in time local energy weak solutions to the Navier-Stokes equations in the half-space $\R^3_+$. Such solutions are sometimes called Lemari\'e-Rieusset solutions in the whole space $\R^3$. The main tool in our work is an explicit representation formula for the pressure, which is decomposed into a Helmholtz-Leray part and a harmonic part due to the boundary. We also explain how our result enables to reprove the blow-up of the scale-critical $L^3(\R^3_+)$ norm obtained by Barker and Seregin for solutions developing a singularity in finite time. 
\end{abstract}

\section{Introduction}

This paper is devoted to the proof of existence of local energy weak solutions to the Navier-Stokes equations in the half-space
\begin{equation}
\label{e.nse}
\left\{ 
\begin{aligned}
& \partial_tu+u\cdot\nabla u-\Delta u+\nabla p = 0, \quad \nabla\cdot u=0  & \mbox{in} & \ (0,T)\times\R^3_+, \\
& u = 0  & \mbox{on} & \ (0,T)\times\partial \R^3_+. 
\end{aligned}
\right.
\end{equation}
for initial data $u_0$ locally uniformly in $L^2$ and divergence-free. 
The study of weak finite energy solutions to \eqref{e.nse} with initial data $u_0\in L^2_{\sigma}(\Omega)$, where $\Omega$ can be for instance a bounded domain, $\R^3$ or $\R^3_+$, has a long history which goes back to the seminal works of Leray \cite{L34} and Hopf \cite{H51}. The study of infinite energy solutions is much more recent. It is interesting in its own right since one can study nontrivial dynamics generated by the solutions themselves and not driven by a source term. Let us just mention the latest works of Abe and Giga \cite{AG13,AG14,Abe15,Abe16} about Stokes and Navier-Stokes equations in $L^\infty$, of Gallay and Slijep\v cevi\'c \cite{GS15} about boundedness for 2D Navier-Stokes equations and of Maremonti and Shimizu \cite{MS18}, Kwon and Tsai \cite{KT} about global weak solutions with initial data non decaying at space infinity. 

\smallskip

We are interested in a special kind of infinite energy solutions, so-called local energy weak solutions. For these solutions the energy is locally uniformly bounded. This notion of solutions has been pioneered by Lemari\'e-Rieusset \cite{lemariebook} in the whole space $\R^3$, and later slightly extended by Kikuchi and Seregin \cite{KS07}. Our goal is to extend the notion of solution to the half-space $\R^3_+$ and to prove local in time as well as global in time existence results. This answers an open problem mentioned by Barker and Seregin in \cite[Section 1]{BS15}:
\begin{quote}
Unfortunately, and analogue of Lemari\'e-Rieusset type solutions for the half-space is not known yet. In fact it is an interesting open problem.
\end{quote}

\smallskip

The class of local energy weak solutions, which will be made precise in Definition \ref{def.weakles}, is very useful, even for the study of finite energy weak solutions to \eqref{e.nse}, so-called Leray-Hopf solutions, for at least three reasons.\\
\noindent The first reason is that they satisfy a local energy inequality. In particular, the solutions are 
suitable  in the sense of Caffarelli, Kohn and Nirenberg \cite{CKN82,Lin98}, so that we can apply $\ep$-regularity to them. The half-space analogues of \cite{CKN82,Lin98}, corresponding to the $\ep$-boundary regularity, have been worked out in \cite{Ser02,SSS04,SS14}.\\
\noindent The second reason is that local energy weak solutions appear as limits of rescaled solutions of the Navier-Stokes equations. This is the case for instance when studying the local behavior of a Leray-Hopf solution near a potential singularity. The energy being supercritical in 3D with respect to the Navier-Stokes scaling $u_\lambda(y,s)=\lambda u(\lambda y,\lambda^2 s)$, the energy blows-up when zooming. The limit object is still a solution of the Navier-Stokes system, not in the finite energy class, but in the local energy class.\\
\noindent Finally, the theory of local energy solutions plays also an important role in the seminal work of Jia and {{\v S}ver{\'a}k} \cite{JS} about the construction of forward self-similar solutions with large initial data. This work and the subsequent studies \cite{JS15,GS17} represent a big progress toward understanding non-uniqueness of Leray-Hopf solutions.

\smallskip

Combining the features of the local energy weak solutions emphasized in the previous paragraph makes them powerful objects to study, for instance, blow-up of scale-critical norms near potential singularities. In this way, Seregin \cite{Ser12} was able to improve the celebrated result of Escauriaza, Seregin and {{\v S}ver{\'a}k} \cite{ISS03}. Seregin proved that: if a weak finite energy solution $u$ to \eqref{e.nse} in $\R^3$ has a first singularity at time $T$, in the sense that $u$ is smooth in the time interval $(0,T)$ and that the $L^\infty$ norm of $u$ is infinite in any parabolic cylinder $B(x_0,\rho)\times (T-\rho^2,T)$, for fixed $x_0\in\R^3$ and any $\rho>0$, then
\begin{equation*}
\|u(\cdot,t)\|_{L^3(\R^3)}\longrightarrow\infty \qquad\mbox{as} \ \ t
\rightarrow T-0.
\end{equation*}
One of our objectives is to show that the solutions we construct make it possible to prove the blow-up of the $L^3$ norm in the case of the half-space $\R^3_+$ following the scheme in \cite{Ser12}. Hence, we will recover the result of \cite[Theorem 1.1]{BS15} of the blow-up of the $L^{3,q}$ norm $3\leq q<\infty$, in the case $q=3$.

\smallskip

The content of this paper was summarized in the review article \cite{P18}. In particular, our notion of local energy weak solutions is compared to the notion of weak solutions in the half-space appearing in the work of Maremonti and Shimizu \cite{MS18}.

\subsection{Definition of local energy weak solutions}

Let us first recall the definition of loc-uniform Lebesgue spaces: for $1\leq q\leq\infty$,
\begin{equation*}
L^q_{uloc}(\R^d_+):=\left\{f\in L^1_{loc}(\R^d_+)~|~ \sup_{\eta\in\mathbb Z^{d-1}\times\mathbb Z_{\geq 0}}\|f\|_{L^q(\eta+(0,1)^d)}<\infty\right\}.
\end{equation*}
Let us define the space $L^p_{uloc,\sigma} (\R^d_+)$ of solenoidal vector fields in $L^
q_{uloc}$ as follows:
\begin{align*}
L^q_{uloc,\sigma} (\R^d_+):= \left\{ f\in L^q_{uloc}(\R^d_+)^d ~|~ \int_{\R^d_+} f \cdot \nabla \varphi \,  d x =0~~{\rm for~any}~\varphi \in C_0^\infty (\overline{\R^d_+})\right\}.
\end{align*}
For more properties of these spaces of locally uniformly $p$-integrable functions, see \cite{MMP17} and the references cited therein. We also refer to Lemma \ref{lem.characterize}, which characterizes the functions of $\mathcal L^2_{uloc,\sigma}$ as the functions $L^2_{uloc,\sigma}$ which have some (not quantified) decay at infinity.

\smallskip

Here we state the definition of local energy weak solutions to \eqref{e.nse} when the initial data belongs to 
\begin{equation*}
\mathcal L^2_{uloc,\sigma}(\R^3_+):=\overline{C^\infty_{c,\sigma}}^{L^2_{uloc}}(\R^3_+).
\end{equation*}
We will actually be able to construct local energy weak solutions for data in $L^2_{uloc,\sigma} (\R^3_+)$. Nevertheless, the introduction of the space $\mathcal L^2_{uloc,\sigma}(\R^3_+)$ is useful since the solutions in this class decay at spatial infinity, and hence, the parasitic solutions (the flows driven by the pressure with linear growth) are automatically excluded in this class.  
Then we can state the definition of solutions in a simple fashion compared with the solutions in the class of nondecaying functions, where the structure of the pressure has to be included in the definition of solutions (see Remark \ref{rem.def.weakles} below on this point). Although $\mathcal L^2_{uloc,\sigma}(\R^3_+)$ is strictly smaller than $L^2_{uloc,\sigma}(\R^3_+)$ the study in this class has an important application to the blow up criterion of  solutions to \eqref{e.nse} in $L^3$, which will be discussed in Section \ref{sec.blowup}.
  
\begin{definition}\label{def.weakles} Let $T\in(0,\infty]$ and $Q_T := (0,T)\times \R^3_+$. A pair $(u,p)$ is called a local energy weak solution to \eqref{e.nse} in $Q_T$ with the initial data $u_0\in \mathcal L^2_{uloc,\sigma} (\R^3_+)$ if $(u,p)$ satisfies the following conditions:

\noindent {\rm (i)} We have $u\in L^\infty(0,T; \mathcal L^2_{uloc,\sigma} (\R^3_+))$ if $T<\infty$, 
 $u\in L^\infty_{loc}([0,T); \mathcal L^2_{uloc,\sigma} (\R^3_+))$ if $T=\infty$ and 
$p\in L^\frac32_{loc} ((0,T)\times \overline{\R^3_+})$,  and 
\begin{align}
\sup_{x\in \R^3_+} \int_0^{T'} \| \nabla u \|_{L^2 (B(x)\cap \R^3_+)}^2 d t + \sup_{x\in \R^3_+} \left(\int_\delta^{T'} \| \nabla p \|_{L^\frac98 (B(x)\cap \R^3_+)}^\frac32 d t\right)^\frac23 <\infty
\end{align}
for all finite $T'\in (0, T]$ and $\delta \in (0,T')$. Here $B(x)$ is the ball of radius $1$ centered at $x$.

\noindent {\rm (ii)} The pair $(u,p)$ satisfies 
\begin{align}
\begin{split}
& \int_0^T - \langle u, \partial_t \varphi \rangle_{L^2(\R^3_+)}  + \langle \nabla u, \nabla \varphi \rangle_{L^2(\R^3_+)} - \langle p, {\rm div}\, \varphi \rangle_{L^2 (\R^3_+)} + \langle u\cdot \nabla u, \varphi \rangle_{L^2 (\R^3_+)}  \, d t =0 \\
&\text{ for any }~~\varphi \in C_c^\infty ((0,T)\times \overline{\R^3_+})^3 \text{ such that }~~\varphi|_{x_3=0}=0.
\end{split}
\end{align}

\noindent {\rm (iii)} The function $t\mapsto \langle u(t), w \rangle_{L^2(\R^3_+)}$ belongs to $C([0,T))$ for any compactly supported $w\in L^2(\R^3_+)^3$. Moreover, for any compact set $K\subset \overline{\R^3_+}$,
\begin{align}\label{e.cvinit}
\lim_{t\rightarrow 0} \| u(t) -u_0 \|_{L^2(K)} =0.
\end{align} 

\noindent {\rm (iv)} The pair $(u,p)$ satisfies the local energy inequality:
for any $\chi\in C^\infty_c((0,T)\times \overline{\R^3_+})$ and for $a.e.~t\in (0,T)$,
\begin{align}\label{e.locenineq} 
\begin{split}
& \| (\chi u)(t) \|_{L^2(\R^3_+)}^2 + 2\int_0^t \| \chi \nabla u \|_{L^2 (\R^3_+)}^2 d s \\
& \qquad \leq \int_0^t \langle |u|^2, \partial_s \chi^2 + \Delta \chi^2 \rangle_{L^2 (\R^3_+)}  + \langle u\cdot \nabla \chi^2, |u|^2 + 2 p \rangle_{L^2 (\R^3_+)} d s.
\end{split}
\end{align}
\end{definition}

\begin{remark}\label{rem.def.weakles}{\rm (1) Our definition is close to the one used in \cite{JS2,JS},
where the authors defined local energy weak solutions in $(0,T) \times \R^3$  which decay at spatial infinity and they do not include the representation formula for the pressure. 
For $T<\infty$ one can also define local energy weak solutions for initial data in $L^2_{uloc,\sigma} (\R^3_+)$. In this case one has to replace the condition $u\in L^\infty (0,T; \mathcal{L}^2_{uloc,\sigma} (\R^3_+))$ by 
$u\in L^\infty(0,T; L^2_{uloc,\sigma} (\R^3_+))$. 
However, since the solutions in this class do not decay at all as $|x|\rightarrow \infty$, the uniqueness is violated even for smooth solutions unless one imposes some additional condition on the pressure. This lack of the uniqueness is brought by the flows driven by  the pressure, called parasitic solutions.  
A typical way to exclude such parasitic solutions is to assume in addition that the pressure $\nabla p$ is written as $\nabla p (t) = \nabla p _{Helm}  (t) + \nabla p_{Har} (t)$, where for $a.e.\, t\in (0,T)$, $\nabla p_{Helm} (t)$ is defined as the solution to the Poisson equations $-\Delta p_{Helm} (t) = \nabla \cdot \nabla (u (t) \otimes u (t))$ in $\R^3_+$ and $\partial_3 p_{Helm} (t)=0$ on $\partial \R^3_+$ which is expressed in terms of  (some derivatives of) the Newton potential, while $\nabla p_{Har} (t)$ is the harmonic pressure which satisfies $\Delta p_{Har} (t)=0$ in $\R^3_+$ and 
\begin{align}
\lim_{R\rightarrow \infty} \| \nabla' p_{Har} (t) \|_{\{|x'|\leq 1, R<x_d<R+1\}} =0 \qquad \text{ for }~~ a.e. ~ t\in (0,T).
\end{align}
This condition for the pressure is not needed for solutions in $L^\infty (0,T; \mathcal{L}^2_{uloc,\sigma} (\R^3_+))$, since the solutions in this class decay at spatial infinity, and thus, the parasitic solutions are automatically excluded.

\noindent (2) The $\ep$-regularity theorem holds for any local energy weak solutions in Definition \ref{def.weakles}. This is not trivial since the regularity assumption for our local energy weak solutions is not strong enough and therefore one has to show that any local energy weak solution admits additional regularity so that the known $\ep$-regularity theorem is applied. To this end a detailed study of the pressure term is required, which will be done in Section \ref{sec.pressureest}, and we also need a uniqueness result (Liouville theorem) for the homogeneous Stokes system obtained in our companion paper \cite[Theorem 5]{MMP17}. This issue will be handled in Section \ref{sec.property.weak}.

\noindent (3) According to (iv) in Definition \ref{def.weakles}, only test functions $\chi$ compactly supported in space and time are allowed in the energy inequality \eqref{e.locenineq}. However, the continuity at $0$ stated in point (iii) of Definition \ref{def.weakles} allows to take test functions constant in time. Let $\chi\in C^\infty_c(\overline{\R^3_+})$. For $\delta>0$, let $\eta\in C^\infty(\R)$ is a cut-off such that $|\eta|\leq 1$, $\eta\equiv 0$ on $(-\infty,1)$ and $\eta\equiv 1$ on $(2,\infty)$. Then $\chi_\delta:=\chi(\eta(\tfrac\cdot{\delta})-\eta(\tfrac{T-\cdot}{\delta}))\in C^\infty_c((0,T) \times \overline{\R^3_+})$ is an admissible test function in \eqref{e.locenineq}. Plugging $\chi_\delta$ in \eqref{e.locenineq}, we let $\delta\rightarrow 0$. Only one term really deserves some attention. We have
\begin{align*}
&\int_0^t\int_{\R^3_+}|u|^2\partial_s(\chi^2\eta(\tfrac\cdot{\delta})^2)dxds-\int_0^t\int_{\R^3_+}|u|^2\partial_s(\chi^2)-\int_{\R^3_+}|\chi u_0|^2dx\\
=\ &\int_0^t\int_{\R^3_+}|u|^2\partial_s(\chi^2)\eta(\tfrac\cdot{\delta})^2dxds-\int_0^t\int_{\R^3_+}|u|^2\partial_s(\chi^2)\eta(\tfrac\cdot{\delta})^2dxds\\
&+\int_0^{2\delta}\int_{\R^3_+}|u|^2\chi^2\partial_s(\eta(\tfrac\cdot{\delta})^2)dxds-\int_{\R^3_+}|\chi u_0|^2dx\\
=\ &o(\delta).
\end{align*}
Indeed, 
\begin{align*}
&\int_0^{2\delta}\int_{\R^3_+}|u|^2\chi^2\partial_s(\eta(\tfrac\cdot{\delta})^2)dxds-\int_{\R^3_+}|\chi u_0|^2dx\\
=\ &2\delta^{-1}\int_0^{2\delta}\int_{\R^3_+}|u|^2\chi^2\eta'(\tfrac\cdot{\delta})\eta(\tfrac\cdot{\delta})dxds-\int_{\R^3_+}|\chi u_0|^2dx\\
=\ &2\delta^{-1}\int_0^{2\delta}\int_{\R^3_+}(|u(\cdot,s)|^2-|u_0|^2)\eta'(\tfrac\cdot{\delta})\eta(\tfrac\cdot{\delta})\chi^2dxds\\
&+\int_0^{2\delta}\partial_s(\eta(\tfrac\cdot{\delta})^2)\int_{\R^3_+}|u_0|^2\chi^2dxds-\int_{\R^3_+}|\chi u_0|^2dx,
\end{align*}
where the first term in the right hand side goes to zero by the boundedness of the Hardy-Littlewood maximal function on $L^\infty$ and the local strong convergence to initial data \eqref{e.cvinit}, and the sum of the last two terms in the right hand side is zero. We hence obtain 
\begin{align}\label{e.locenineqbis} 
\begin{split}
& \| (\chi u)(t) \|_{L^2(\R^3_+)}^2 + 2\int_0^t \| \chi \nabla u \|_{L^2 (\R^3_+)}^2 d s \\
& \qquad \leq \| \chi u_0 \|_{L^2(\R^3_+)}^2+\int_0^t \langle |u|^2, \partial_s \chi^2 + \Delta \chi^2 \rangle_{L^2 (\R^3_+)}  + \langle u\cdot \nabla \chi^2, |u|^2 + 2 p \rangle_{L^2 (\R^3_+)} d s.
\end{split}
\end{align}
This result will be used in Section \ref{sec.decay}.}
\end{remark}

\subsection{Outline of our results}

The main result of our paper is stated as follows:

\begin{theorem}\label{prop.global.weak} 
For any $u_0 \in \mathcal{L}^2_{uloc,\sigma} (\R^3_+)$ there exists a local energy weak solution $(u,p)$ to \eqref{e.nse} in $Q_\infty$ with initial data $u_0$.
\end{theorem}

This result states the global in time existence of local energy weak solutions in the sense of Definition \ref{def.weakles}. It is the analog for the half-space of the theorem of Lemari\'e-Rieusset \cite[Theorem 33.1]{lemariebook} and of Kikuchi and Seregin \cite[Theorem 1.5]{KS07} for the whole space $\R^3$. Local in time existence of local energy weak solutions for data in $L^2_{uloc,\sigma}(\R^3_+)$ is proved in Section \ref{sec.locexiles}, see Proposition \ref{prop.local}. 

\smallskip

The proof of Theorem \ref{prop.global.weak} goes roughly as follows. The evolution starts with a rough data barely locally in $L^2$, $u_0\in\mathcal L^2_{uloc}$. The local in time local energy weak solution obtained thanks to Proposition  \ref{prop.local} instantly becomes slightly more regular, $u(\cdot,t_0)\in\mathcal L^4_{uloc,\sigma}(\R^3_+)$ for almost all $t_0$ in the time existence interval. This allows to decompose the data $u(\cdot,t_0)$ into a large $C^\infty_c(\R^3_+)$ part for which we have global in time Leray-Hopf solutions, and a small part in $L^4_{uloc,\sigma}(\R^3_+)$ for which we have local in time existence of mild solutions thanks to Proposition 7.1 in \cite{MMP17}. The difficult part of this reasoning is to transfer the decay of the initial data $u_0\in\mathcal L^2_{uloc,\sigma}$ to the solution $u(\cdot,t)$, i.e. to prove that not only $u(\cdot,t_0)\in L^4_{uloc,\sigma}(\R^3_+)$ for almost all $t_0$, but that actually $u(\cdot,t_0)\in\mathcal L^4_{uloc,\sigma}(\R^3_+)$ for almost all $t_0$. This issue is already addressed in \cite[Proposition 32.2]{lemariebook} and \cite[Theorem 1.4]{KS07} in the case of the whole space. We handle this question for $\R^3_+$. Our main result in this direction is the following theorem, which holds under the assumption below. Let $T>0$ and $\delta>0$ be fixed.

\begin{assumption}\label{assu.A}
There exists $A_{T,\delta}\geq 1$ such that for all $u_0\in\mathcal L^2_{uloc,\sigma}(\R^3_+)$, for every solution $u$ to \eqref{e.nse} in $Q_T$ in the sense of Definition \ref{def.weakles} with initial data $u_0$, if $\|u_0\|_{L^2_{uloc}(\R^3_+)}\leq \delta$ 
then
\begin{equation}\label{e.apriorilenassu}
\sup_{t\in(0,T)}\sup_{\eta\in\Z^3_+}\int_{\cu(\eta)}|u(\cdot,t)|^2+\int_0^T\int_{\cu(\eta)}|\nabla u|^2+\left(\int_0^T\int_{\cu(\eta)}|u|^3\right)^\frac23\leq A_{T,\delta}.
\end{equation}
\end{assumption}

\begin{theorem}\label{prop.decayinfty}
Assume that Assumption \ref{assu.A} holds. Then for all $u_0\in\mathcal L^2_{uloc,\sigma}(\R^3_+)$, all weak local energy solution $u$ to \eqref{e.nse} on $Q_T$ in the sense of Definition \ref{def.weakles} with initial data $u_0$ satisfies
\begin{align}\label{e.decayatinfty}
&\sup_{t\in (0,T)}\sup_{\eta\in\Z^3_+}\int_{\cu(\eta)}|\vartheta_R u(\cdot,t)|^2+\int_0^T\int_{\cu(\eta)}|\vartheta_R\nabla u|^2\\
&\qquad+\left(\int_0^T\int_{\cu(\eta)}|\vartheta_R u|^3\right)^\frac23+\left(\int_\delta^T\int_{\cu(\eta)}|p|^\frac32\right)^\frac23\stackrel{R\rightarrow \infty}{\longrightarrow}0,\nonumber
\end{align}
for all $\delta\in(0,T)$, with $\vartheta$ the cut-off defined in Section \ref{sec.notations} and $\vartheta_R:=\vartheta(\cdot/R)$.
\end{theorem}

Explaining how to prove Theorem \ref{prop.decayinfty} leads us to the central results our work. The starting point to get the decay estimate \eqref{e.decayatinfty} is the local energy inequality \eqref{e.locenineq} tested against $\varphi:=\vartheta_R^2\chi_{x_0}$, where $\chi_{x_0}$ is a cut-off supported around $x_0$. Estimating the right hand side of the energy inequality requires precise estimates for the pressure. Therefore, a lot of work is devoted to studying the pressure of solutions in the sense of Definition \ref{def.weakles}. The foremost novelty of our paper is to provide a decomposition of the pressure along with estimates. In the whole space, the pressure solves 
\begin{equation*}
-\Delta p=\nabla\cdot(\nabla\cdot u\otimes u)\quad\mbox{in}\quad\R^3.
\end{equation*}
It is equal to the Helmholtz pressure of the Helmholtz-Leray decomposition. At least formally, we can represent this pressure using the fundamental solution of $-\Delta$. We then decompose the integral into a local part and a nonlocal part. To handle the nonlocal part, the point is that the pressure is defined up to a constant (possibly depending on time), so that one can gain the additional integrability needed to estimate the large scales of the data (which may not decay). For all $x_0\in\R^3$, there exists a function $c_{x_0}(t)\in L^\frac32(0,T)$ such that for all $(x,t)\in\R^3\times(0,T)$,
\begin{align}\label{e.decpressureR3}
\begin{split}
p(x,t)-c_{x_0}(t)=\ &\underbrace{-\frac13|u(x,t)|^2+\frac1{4\pi}\int_{B(x_0,2)}K(x-y)\cdot u(y,t)\otimes u(y,t)dy}_{p_{loc}}\\
&+\underbrace{\frac1{4\pi}\int_{\R^3\setminus B(x_0,2)}(K(x-y)-K(x_0-y))\cdot u(y,t)\otimes u(y,t)dy}_{p_{nonloc}},
\end{split}
\end{align}
with $K:=\nabla^2(\frac1{|x|})$, see \cite{lemariebook,KS07}. This decomposition is then used to estimate $\int_{\R^3}p\nabla(\vartheta_R^2\chi_{x_0})u$. 

\smallskip

In the present work, we generalize the representation formula \eqref{e.decpressureR3} to the case of the $\R^3_+$. Due to the boundary $\partial \R^3_+$, in addition to the Helmholtz pressure, a harmonic pressure has to be taken into account. Indeed, the pressure solves 
\begin{equation}
\label{e.pressurehp}
\left\{ 
\begin{aligned}
& -\Delta p = \nabla\cdot(\nabla\cdot u\otimes u)  & \mbox{in} &\ \R^3_+, \\
& \partial_d p = \gamma|_{x_3=0}\Delta u_{3}  & \mbox{on} &\ \partial \R^3_+. 
\end{aligned}
\right.
\end{equation}
We are able to provide an explicit representation for the Helmholtz part of the pressure, as well as for the harmonic part. Each pressure has to be splitted (as above for $\R^3$) into a local part and a nonlocal part. It is the purpose of Section \ref{sec.pressureest} to do this work. The precise decomposition of the pressure is given in \eqref{e.decomppressure}. Propositions \ref{prop.estlipressure}, \ref{prop.estlocpressure} and \ref{prop.nlpressure} are pivotal results in our work: the estimates for each pressure terms are provided there. To our knowledge, such an extensive study of the pressure in a domain with boundaries is new. We are able to provide explicit representation formulas. In this matter, we rely on the results for the linear theory in $\R^3_+$ obtained in the companion paper \cite{MMP17}. As a word of conclusion, let us mention that this level of precision in the description of the pressure can be achieved due to the special structure of $\R^3_+$, which allows to use the Fourier transform in the horizontal direction and hence to obtain explicit formulas. In more general domains (exterior domains, domains with unbounded curved boundaries), the study of solutions with locally integrable data relies on mild assumptions on the pressure which make it possible to rule out parasitic solutions. In this vein, see for instance the works \cite{AG13,Abe15} for data in $L^\infty_\sigma$.

\subsection{Notations}
\label{sec.notations}

We gather some notations which are used recurrently in this paper. We define $\cu(x_0):=x_0+(-\frac12,\frac12)^{d}\cap\R^d_+$ and for $r>0$, $r\cu(x_0):=x_0+(-\frac r2,\frac r2)^{d}\cap\R^d_+$. Moreover, $\cu'(x_0):=x'_0+(-\frac12,\frac12)^{d-1}\subset \R^{d-1}$ and $\cu_d(x_0):=x_{0,d}+(-\frac12,\frac12)\subset \R$. The function $\chi\in C^\infty_c(\R^d)$ stands for a non negative cut-off, equal to $1$ on $\cu(0)$ and $0$ on $2\cu(0)^c$, and $\chi_r\in C^\infty_c(\R^d)$ stands for a non negative cut-off, equal to $1$ on $r\cu(0)$ and $0$ on $(r+1)\cu(0)^c$. We also let $\chi_{\!_{\, x_0,r}}=\chi_r(\cdot-x_0)$. Finally, $\vartheta\in C^\infty(\R^d)$ denotes a non negative cut-off equal to $0$ on $\cu(0)$ and $1$ on $2\cu(0)^c$. The notation $P(\cdot)=e^{-(-\Delta')^\frac12\cdot}$ denotes the Poisson semigroup. The notation $m_\alpha(D')$ stands for a tangential Fourier multiplier homogeneous of order $\alpha$, $\alpha>-d+1$ which may change from line to line.

\subsection{Overview of the paper}

The structure of the paper is as follows. First we derive properties of the local energy weak solutions in the sense of Definition \ref{def.weakles} (Section \ref{sec.pressureest}, Section \ref{sec.property.weak}, Section \ref{sec.decay}). Then, we investigate local in time and global in time existence results (Section \ref{sec.locexiles} and Section \ref{sec.global.weak}). Eventually, we apply the results of the paper to investigate the blow-up of the scale-critical norm $L^3$ (Section \ref{sec.blowup}). Let us now describe each section in more details. Section \ref{sec.pressureest} is the foundation for the paper. The decomposition of the pressure, along with the representation formulas and estimates are provided there. In Section \ref{sec.property.weak} we give further properties that all solutions in the sense of Definition \ref{def.weakles} share. The goal is to show that $\ep$-regularity results apply for our class of solutions. In Section \ref{sec.decay}, we prove the crucial result, Theorem \ref{prop.decayinfty} enabling to transfer the decay of the initial data $u_0\in\mathcal L^2_{uloc,\sigma}$ to the solution. Section \ref{sec.locexiles} addresses the local in time existence of local energy weak solutions for data in $u_0\in L^2_{uloc,\sigma}$. This result, combined with Theorem \ref{prop.decayinfty} about the spatial decay of solutions makes it possible to prove Theorem \ref{prop.global.weak} in 
Section \ref{sec.global.weak}. We eventually use Theorem \ref{prop.decayinfty} in Section \ref{sec.blowup} to prove the blow-up of $\|u(\cdot,t)\|_{L^3(\R^3_+)}$ when $t\rightarrow T$, in case $T<\infty$ is the time of the first blow-up for $u$. This result is stated in Theorem \ref{theo.blowup}. In Appendix \ref{sec.lerayproj} we state results about the Helmholtz-Leray projection. Most of these results are taken from the companion paper \cite{MMP17}.

\section{Pressure estimates}
\label{sec.pressureest}

Let $x_0\in\R^3_+$ be fixed. Our goal is to get local estimates for the pressure $p(x,t)$ of the Navier-Stokes initial value problem \eqref{e.nse} 
for $x\in\cu(x_0)$, in terms of the local energy norm of the velocity $u$. 
For this purpose, let us consider the cut-off $\chi_{\!_{\, 4}}:=\chi_{\!_{\, x_0,4}}$, which is defined in Section \ref{sec.notations}. We will also need to work with the cut-off $\chi_{\!_{\, 2}}:=\chi_{\!_{\, x_0,2}}$. The first step is to decompose the solution $u$ of \eqref{e.nse} into a local part $u_{loc}^{x_0}$ with finite energy and a nonlocal part $u_{nonloc}^{x_0}$ with locally finite energy. Most of time we drop the superscript $x_0$ in this section, because it is clear that the quantities depending on the cut-off $\chi_{\!_{\, 4}}$ or $\chi_{\!_{\, 2}}$ depend on $x_0$, which is fixed here. let $(u,p)$ be any local weak solution to \eqref{e.nse} with initial data $u_0$ in the sense of Definition \ref{def.weakles}. Let us denote by $u_{li}$, $u_{loc}^{u\otimes u}$, and $u_{nonloc}^{u\otimes u}$ the solutions of the following systems:
\begin{equation}
\label{e.linearsol}
\left\{ 
\begin{aligned}
& \partial_t u_{li}-\Delta u_{li}+\nabla p_{li}  =0 & \mbox{in} & \ (0,T)\times\R^3_+, \\
& \nabla\cdot u_{li}=0, & &\\
& u_{li} = 0  & \mbox{on} & \ (0,T)\times\partial \R^3_+,\\
& u_{li}(\cdot,0)= u_0, & &
\end{aligned}
\right.
\end{equation}
\begin{equation}
\label{e.locnse}
\left\{ 
\begin{aligned}
& \partial_t u_{loc}^{u\otimes u} -\Delta u_{loc}^{u\otimes u} +\nabla p_{loc}^{u\otimes u} = -\nabla\cdot(\chi_{\!_{\, 4}}^2 u\otimes u) & \mbox{in} & \ (0,T)\times\R^3_+, \\
& \nabla\cdot u_{loc}^{u\otimes u}=0, & &\\
& u_{loc}^{u\otimes u} = 0  & \mbox{on} & \ (0,T)\times\partial \R^3_+,\\
& u_{loc}^{u\otimes u} (\cdot,0)=0, & &
\end{aligned}
\right.
\end{equation}
and the nonlocal part solves 
\begin{equation}
\label{e.nonlocnse}
\left\{ 
\begin{aligned}
& \partial_tu_{nonloc}^{u\otimes u} -\Delta u_{nonloc}^{u\otimes u} +\nabla p_{nonloc}^{u\otimes u}=-\nabla\cdot((1-\chi_{\!_{\, 4}}^2) u\otimes u), &  &\\
& \nabla\cdot u_{nonloc}^{u\otimes u} =0, & \mbox{in} & \ (0,T)\times\R^3_+, \\
& u_{nonloc}^{u\otimes u}  = 0  & \mbox{on} & \ (0,T)\times\partial \R^3_+,\\
& u_{nonloc}^{u\otimes u} (\cdot,0)=0. & &
\end{aligned}
\right.
\end{equation}
The couple $(u_{li}, p_{li})$ is the solution to the Stokes system,  $(u_{loc}^{u\otimes u}, p_{loc}^{u\otimes u})$ is the local and nonlinear part, and $(u_{nonlocal}^{u\otimes u}, p_{nonlocal}^{u\otimes u})$ is the nonlocal and nonlinear part.
These are constructed as mild solutions which satisfy the integral representation formula.
Formally we have $u=u_{li}+u_{loc}^{u\otimes u} + u_{nonloc}^{u\otimes u}$ and $p=p_{li} + p_{loc}^{u\otimes u} + p_{nonloc}^{u\otimes u}$ for the solution $(u,p)$ to \eqref{e.nse}, which will be rigorously verified when $(u,p)$ is a local weak energy solution in the sense of Definition \ref{def.weakles}. For each system \eqref{e.locnse} and \eqref{e.nonlocnse}, we will split the pressure into a Helmholtz part, which comes from the Helmholtz-Leray decomposition of the source term, and a harmonic part, which is due to the boundary.

\smallskip

In this paragraph, we concentrate on the linear pressure $p_{li}$. 
We obtain a representation formula and estimate it directly. Notice that the pressure $p_{li}$ is the solution of 
\begin{equation}
\label{e.pressureu0loc}
\left\{ 
\begin{aligned}
& -\Delta p_{li} = 0  & \mbox{in} &\ \R^3_+, \\
& \partial_d p_{li} = \gamma|_{x_3=0}\Delta u_{li,3}  & \mbox{on} &\ \partial \R^3_+. 
\end{aligned}
\right.
\end{equation}
In other words, the Helmholtz part of the pressure is zero, so $p_{li}$ is equal to its harmonic part. The representation formula for $p_{li}$ follows by inverse Laplace transform from the formula for the pressure of the resolvent problem given in \cite[Section 2]{MMP17}. 
Hence, it is formally written as 
\begin{align}\label{e.forpressureu0loc'}
\begin{split}
p_{li} (x,t)
& := \frac{1}{2\pi i} \int_\Gamma e^{\lambda t}\int_{\R^3_+}q_\lambda(x'-z',x_3,z_3)\cdot u_0'(z',z_3)dz'dz_3d\lambda\\
& = \frac{1}{2\pi i} \int_\Gamma e^{\lambda t}\int_{\R^3_+}q_\lambda(x'-z',x_3,z_3)\cdot \chi_{\!_{\, 4}} u_0'(z',z_3)dz'dz_3d\lambda\\
& \quad + \int_\Gamma e^{\lambda t}\int_{\R^3_+}q_\lambda(x'-z',x_3,z_3)\cdot (1-\chi_{\!_{\, 4}}) u_0'(z',z_3)dz'dz_3d\lambda\\
& = : p_{loc}^{u_0} (x,t) + \tilde p_{nonloc}^{u_0} (x,t).
\end{split}
\end{align}
Here 
\begin{equation}\label{e.defqlambda}
q_\lambda:\ \R^{2}\times(0,\infty)\times (0,\infty)\rightarrow\mathbb C^{2} 
\end{equation}
is the harmonic pressure kernel for the resolvent problem, and $\Gamma=\Gamma_\kappa$ with $\kappa\in (0,1)$ is the curve 
\begin{equation}\label{e.defgammacurve}
\{\lambda\in \mathbb{C}~|~|{\rm arg}\, \lambda|=\eta, ~|\lambda|\geq \kappa\} \cup \{\lambda\in \mathbb{C}~|~|{\rm arg}\, \lambda|\leq \eta, ~|\lambda|=\kappa\}
\end{equation}
for some $\eta \in (\frac{\pi}{2},\pi)$. The local pressure $p_{loc}^{u_0}$ is certainly the most subtle term to analyze. Because $\chi_{\!_{\, 4}} u_0$ is barely in $L^2$, it is not regular enough to be in the admissible class for the initial data of \cite{GS91}. On the other hand, due to the decay properties of the kernel $q_\lambda$ the representation of $\tilde p_{nonloc}^{u_0}$ is not well defined. This however motivates the following definition of the harmonic nonlocal pressure as follows:
\begin{align}\label{e.forharmpressure.u_0}
p_{nonloc}^{u_0} (x,t):= \frac{1}{2\pi i} \int_\Gamma e^{\lambda t}\int_{\R^3_+}q_{\lambda,x,x_0}(z',z_3)\cdot (1-\chi_{\!_{\, 4}})u_0'(z',z_3)dz'dz_3d\lambda,
\end{align}
where
\begin{equation}
\label{def.qlambdax0}
q_{\lambda,x,x_0}(z',z_3):=q_\lambda(x'-z',x_3,z_3)-q_\lambda(x_0'-z',x_{0,3},z_3).
\end{equation}
Then \eqref{e.forharmpressure.u_0} is well-defined for nonlocalized data.
Then, since $p_{li}$ can be defined modulo constants, instead of \eqref{e.forpressureu0loc'}, 
we define $p_{li}$ as
\begin{align}\label{e.forpressureu0loc}
\begin{split}
p_{li} (x,t) & = : p_{loc}^{u_0} (x,t) +  p_{nonloc}^{u_0} (x,t).
\end{split}
\end{align}

\smallskip

For the local pressure in \eqref{e.locnse}, we first decompose the source term by using the Helmholtz decomposition. We have
\begin{equation*}
\nabla\cdot(\chi_{\!_{\, 4}}^2u\otimes u)=\mathbb P\nabla\cdot(\chi_{\!_{\, 4}}^2u\otimes u)+\mathbb Q\nabla\cdot(\chi_{\!_{\, 4}}^2u\otimes u)=\mathbb P\nabla\cdot(\chi_{\!_{\, 4}}^2u\otimes u)+\nabla p_{loc,H}^{u\otimes u}.
\end{equation*}
Notice that there exists a constant $c\in\R$ such that
\begin{equation}\label{e.def.pHloc}
p_{loc,H}^{u\otimes u}(x,t):=c\chi_{\!_{\, 4}}^2|u|^2(x,t)+\int_{\R^3_+}\nabla^2_zN(x'-z',x_3,z_3)\chi_{\!_{\, 4}}^2u\otimes u(z',z_3,t)dz'dz_3,
\end{equation}
where $N$ is the Neumann function for the half-space. We now denote by $p_{loc,harm}^{u\otimes u}$ the remaining (harmonic) pressure, defined in the following way
\begin{equation}\label{e.defplocharmuu}
p_{loc,harm}^{u\otimes u}:=p_{loc}^{u\otimes u}-p_{loc,H}^{u\otimes u}.
\end{equation}
By definition the pair $(u_{loc,harm}^{u\otimes u},p_{loc,harm}^{u\otimes u})$ solves a system akin \eqref{e.locnse} but with source term $-\mathbb P\nabla\cdot(\chi_{\!_{\, 4}}^2u\otimes u)$. 
By the Poincar\'e-Sobolev-Wirtinger inequality (see \cite[Section II.6]{Galdi_book}) 
there exists a constant $c(t)\in\R$ and a constant $C<\infty$ such that 
\begin{equation}\label{e.psw}
\|p_{loc,harm}^{u\otimes u}-c(t)\|_{L^{\frac32}(0,T;L^{\frac95}(\R^3_+))}\leq C\|\nabla p_{loc,harm}^{u\otimes u}\|_{L^{\frac32}(0,T;L^{\frac98}(\R^3_+))}.
\end{equation}
Since $p_{loc,harm}^{u\otimes u}$ is defined up to a constant, we assume, without loss of generality that $c=0$. The difficulty for this pressure terms comes from the fact that one has to estimate singular integral operators in space and time. When possible, we will directly rely on maximal regularity results for the Stokes system in the half-space, see \cite{GS91}. 

\smallskip

Let us now spend some time explaining how to get a formula for the nonlocal pressure $p_{nonloc}^{u\otimes u} (x,t)$ at a point $x\in\cu(x_0)$ making sense for non decaying data. 
The difficulty comes from ensuring that the kernels in the representation formulas have enough decay at infinity to make sense for non decaying data. Such issues are of course already present in the whole space. Nevertheless, the case of the half-space is more involved for two reasons: (i) besides the Helmholtz pressure, one has to analyze a harmonic pressure driven by the trace of $\Delta u_3$ on the boundary $\partial \R^3_+$, (ii) the expression for the Helmholtz-Leray projection is more complicated, see \eqref{e.leraydivtan}. We decompose the pressure in \eqref{e.nonlocnse} into $\nabla p_{nonloc}^{u\otimes u} =\nabla p_{nonloc,H}^{u\otimes u}+\nabla p_{harm}^{u\otimes u}$, where $p_{nonloc,H}^{u\otimes u}$ is the Helmholtz pressure and $p_{harm}^{u\otimes u}$ is the harmonic pressure due to the presence of the boundary $\partial\R^3_+$. The Helmholtz pressure is given by the decomposition of $\nabla\cdot((1-\chi_{\!_{\, 4}}^2)u\otimes u)$ into
\begin{align*}
\nabla\cdot((1-\chi_{\!_{\, 4}}^2)u\otimes u)=\ &\mathbb P\nabla\cdot((1-\chi_{\!_{\, 4}}^2)u\otimes u)+\mathbb Q\nabla\cdot((1-\chi_{\!_{\, 4}}^2)u\otimes u)\\
=\ &\mathbb P\nabla\cdot((1-\chi_{\!_{\, 4}}^2)u\otimes u)+\nabla p_{nonloc,H}^{u\otimes u},
\end{align*}
where $\mathbb P$ is the Helmholtz-Leray projection in $\R^3_+$ defined in \cite[Section 6]{MMP17}. Hence $p_{nonloc,H}^{u\otimes u}$ is a solution of the following Neumann problem, 
\begin{equation*}
\label{e.Helmholtzpressure}
\left\{ 
\begin{aligned}
& -\Delta p_{nonloc,H}^{u\otimes u} =  \nabla\cdot\nabla((1-\chi_{\!_{\, 4}}^2)u\otimes u) & \mbox{in} &\ \R^3_+, \\
& \partial_d p_{nonloc,H}^{u\otimes u} = \nabla\cdot((1-\chi_{\!_{\, 4}}^2)uu_3)  & \mbox{on} &\ \partial \R^3_+. 
\end{aligned}
\right.
\end{equation*}
which decays away from the boundary. Notice that $p_{nonloc,H}^{u\otimes u}$ is defined up to a constant $p_{nonloc,H}^{u\otimes u,x_0}(t)$. Using the Neumann function $N$ for the half-space, we can express $p_{nonloc,H}^{u\otimes u}(x,t)+p_{nonloc,H}^{u\otimes u,x_0}(t)$ as a singular integral with the kernel $\nabla^2_zN$. This kernel has the critical decay $\frac{1}{|x-z|^3}$, which is not enough to handle non localized data. Having the constant depend on $x_0$ and $t$ makes it possible to gain additional decay of the kernel. Aspects related to the definition of the pressure for non decaying data have been investigated intensively in \cite{KS07}, though in the case of the whole space $\R^3$. Here, we adapt their ideas to the case of $\R^3_+$. Formally, we would like to choose
\begin{equation*}
p_{nonloc,H}^{u\otimes u,x_0}(t):=\int_{\R^3_+}\nabla^2_zN(x_0'-z',x_{0,3},z_3)(1-\chi_{\!_{\, 4}}^2(z',z_3))u\otimes u(z',z_3,t)dz'dz_3.
\end{equation*}
Although this quantity is not well defined because the kernel $\nabla^2_zN$ is decaying too slowly, it motivates the following definition of the Helmholtz pressure
\begin{multline}\label{e.forphelm}
p_{nonloc,H}^{u\otimes u}(x,t)
:=\int_{\R^3_+}\nabla^2_zN_{x,x_0}(z',z_3)(1-\chi_{\!_{\, 4}}^2(z',z_3))u\otimes u(z',z_3,t)dz'dz_3,
\end{multline}
with 
\begin{equation}\label{e.defNxx_0}
N_{x,x_0}(z',z_3):=N(x'-z',x_3,z_3)-N(x'_0-z',x_{0,3},z_3).
\end{equation}
Formula \eqref{e.forphelm} makes now sense for data $u$ bounded in the local energy norm. 

\smallskip

The harmonic pressure is the solution of the following Neumann boundary value problem
\begin{equation}\label{e.pbpharm}
\left\{ 
\begin{aligned}
& -\Delta p_{harm}^{u\otimes u} = 0 & \mbox{in} &\ \R^3_+, \\
& \partial_d p_{harm}^{u\otimes u} = \gamma|_{x_3=0}\Delta u_{nonloc,3}^{u\otimes u}  & \mbox{on} &\ \partial \R^3_+.
\end{aligned}
\right.
\end{equation}
For the harmonic pressure, the equation and the Neumann condition in the system \eqref{e.pbpharm} are automatically compatible. The representation formula for $p_{harm}^{u\otimes u}$ follows by inverse Laplace transform from the formula for the pressure of the resolvent problem given in \cite[Section 2]{MMP17}. Again, notice that the pressure is defined up to some constant $p_{harm}^{x_0}(t)$ depending only on $x_0$ and on time. As above for the Helmholtz pressure, the reason for being of this constant is to ensure that we have enough decay at the large scales. Formally, we would like to take 
\begin{align*}
 p_{harm}^{u\otimes u}(x_0, t) :=\ &\frac{1}{2\pi i} \int_0^t\int_\Gamma e^{\lambda (t-s)}\int_{\R^3_+}q_\lambda(x_0'-z',x_{0,3},z_3) \\
& \qquad \cdot \Big (\mathbb P\nabla\cdot ((1-\chi_{\!_{\, 4}}^2)u\otimes u)\Big)'(z',z_3,s)dz'dz_3d\lambda ds,
\end{align*}
where $q_\lambda$ and $\Gamma$ are defined as above in \eqref{e.defqlambda} and \eqref{e.defgammacurve}, but due to the decay properties of the kernel $q_\lambda$ this constant is not well defined. 
As in the case of $p_{li}$, we therefore define the harmonic pressure $p_{harm}^{u\otimes u}$ as 
\begin{align}\label{e.forharmpressure}
\begin{split}
p_{harm}^{u\otimes u} (x,t) :=\ & \frac1{2\pi i}\int_0^t\int_\Gamma e^{\lambda (t-s)}\int_{\R^3_+}q_{\lambda}(x'-z',x_3,z_3)\\
&\quad\cdot \chi_2^2 \left(\mathbb P\nabla\cdot ((1-\chi_{\!_{\, 4}}^2)u\otimes u)\right)'(z',z_3,s)dz'dz_3d\lambda ds,\\
& + \frac1{2\pi i}\int_0^t\int_\Gamma e^{\lambda (t-s)}\int_{\R^3_+}q_{\lambda,x,x_0}(z',z_3)\\
&\quad\cdot  (1-\chi_2^2) \left(\mathbb P\nabla\cdot ((1-\chi_{\!_{\, 4}}^2)u\otimes u)\right)'(z',z_3,s)dz'dz_3d\lambda ds\\
 =:\ & p_{harm,\leq 1}^{u\otimes u} (x,t) + p_{harm,\geq 1}^{u\otimes u} (x,t),
\end{split}
\end{align}
where $q_{\lambda,x,x_0}$ is defined by \eqref{def.qlambdax0}. The formula \eqref{e.forharmpressure} makes sense for non localized data, 
and this argument again essentially relies on the fact that the pressure can be determined up to constants. 

\smallskip

To put it in a nutshell, we remark that we have formally decomposed the pressure $p$ in the system \eqref{e.nse} in the following way
\begin{align}\label{e.decomppressure}
\begin{split}
p:=p^{(x_0)} & =p_{li} + p_{loc}^{u\otimes u} +p_{nonloc}^{u\otimes u} \\
& = \underbrace{p_{loc}^{u_0}+ p_{nonloc}^{u_0}}_{=p_{li}} + \underbrace{p_{loc,H}^{u\otimes u}+p_{loc,harm}^{u\otimes u}}_{= p_{loc}^{u\otimes u}}+ \underbrace{p_{nonloc,H}^{u\otimes u} + p_{harm,\leq 1}^{u\otimes u} + p_{harm,\geq 1}^{u\otimes u}}_{=p_{nonloc}^{u\otimes u}}.
\end{split}
\end{align}
It is essential to keep in mind that every term in decomposition above depends on $x_0$. However, for two points $x_0$ and $x_0'$, the definition $p^{(x_0)}-p^{(x_0')}$ is a constant that depends only on time. 
In Section \ref{sec.property.weak} the decomposition \eqref{e.decomppressure} will be verified for any local weak solutions in the sense of Definition \ref{def.weakles}.
We aim now at estimating every term in the right hand side of \eqref{e.decomppressure} in $\cu(x_0)$ for a fixed $x_0\in \R^3_+$. The results are summarized in the following three propositions.

\begin{proposition}[Estimates for the linear pressure terms]\label{prop.estlipressure}
Let $T>0$. There exists a constant $C(T)<\infty$ such that for $t\in (0,T)$, 
\begin{align}
\frac t{\log(e+t)} \| \nabla p_{li} (t) \|_{L^2_{uloc}(\R^3_+)} & \leq C \| u_0 \|_{L^2_{uloc} (\R^3_+)},\label{e.estpli.1}\\
t^{\frac34} \|p_{loc}^{u_0} (t)\|_{L^2(\cu(x_0))} & \leq C\|u_0\|_{L^2(5\cu(x_0))},\label{e.estplocu0}\\
t^\frac34 \|p_{nonloc}^{u_0} (t) \|_{L^\infty(\cu(x_0)))} + t^\frac34 \| \nabla p_{nonloc}^{u_0} (t) \|_{L^\infty(\cu(x_0)))} & \leq C\|u_0\|_{L^2_{uloc}(\R^{3}_+)}.\label{e.estpharmu0}
\end{align}
Moreover, \eqref{e.estpli.1} holds with $C$ independent of $T$. 
\end{proposition}

Let $T>0$ be fixed. Notice that \eqref{e.estplocu0} implies that for all $p\in[1,\frac43)$, there exists a constant $C(T)<\infty$ such that
\begin{equation*}
\|p_{loc}^{u_0}\|_{L^p(0,T;L^2(\cu(x_0)))}\leq C\|u_0\|_{L^2(5\cu(x_0))}.
\end{equation*}
Moreover, for all $\delta\in(0,T)$, there exists a constant $C(T,\delta)<\infty$ such that 
\begin{equation*}
\|p_{loc}^{u_0}\|_{L^\infty(\delta,T;L^2(\cu(x_0)))}\leq C\|u_0\|_{L^2(5\cu(x_0))}.
\end{equation*}

\begin{proposition}[Estimates for the local pressure terms]\label{prop.estlocpressure}
Let $T>0$. There exists a constant $C(T)<\infty$ such that 
\begin{align}\label{e.estplocuunew}
\begin{split}
&\left\|p^{u\otimes u}_{loc,H}\right\|_{L^\frac32(0,T;L^\frac32(\cu(x_0)))}+\left\|p^{u\otimes u}_{loc,harm}\right\|_{L^\frac32(0,T;L^\frac32(\cu(x_0)))} + \left\| \nabla p^{u\otimes u}_{loc,harm}\right\|_{L^\frac32(0,T;L^\frac98(\R^3_+))}\\
\leq\  &
C
\left(\left\|u\right\|_{L^\infty(0,T;L^2(5\cu(x_0)))}^2
+\left\|\nabla u\right\|_{L^2(0,T;L^2(5\cu(x_0)))}^2\right).
\end{split}
\end{align}
\end{proposition}

\begin{proposition}[Estimates for the nonlocal pressure terms]\label{prop.nlpressure}
Let $T>0$ and $1\leq q<\infty$. There exist constants $C(T), C_q (T) <\infty$ such that for almost all $t\in (0,T)$,
\begin{align}
\|p_{nonloc,H}^{u\otimes u}(\cdot,t)\|_{L^\infty(\cu(x_0))} + \|\nabla p_{nonloc,H}^{u\otimes u}(\cdot,t)\|_{L^\infty(\cu(x_0))} \leq\ & C\|u(\cdot,t)\|_{L^2_{uloc}(\R^3_+)}^2,\label{e.boundnlochelmpressure}\\
\|p_{harm,\leq 1}^{u\otimes u} (\cdot,t)\|_{L^\infty(\cu(x_0))} + \|\nabla p_{harm,\leq 1}^{u\otimes u} (\cdot,t)\|_{L^q(\cu(x_0))}\leq\ & C_q \|u\|_{L^\infty (0,t; L^2_{uloc}(\R^3_+))}^2, \label{e.locestpharmuu}
\end{align}
and 
\begin{align}
\begin{split}
& \|p_{harm,\geq 1}^{u\otimes u}\|_{L^2 (0,T; L^\infty (\cu(x_0)))} + \|\nabla p_{harm,\geq 1}^{u\otimes u}(\cdot,t)\|_{L^2 (0,T; L^\infty (\cu(x_0)))}  \\
& \leq C \Big ( \sup_{\eta\in \Z^3_+} \|\nabla u\|_{L^2 (0,T;L^2 (\cu(\eta)))}^2 + \int_0^T \|u (s)\|_{ L^2_{uloc}(\R^3_+)}^2\, d s \Big ).
\end{split}\label{e.nonlocestpharmuu}
\end{align}
\end{proposition}

\subsection{Estimates for the linear pressure terms: proof of Proposition \ref{prop.estlipressure}}
\label{sec.lipressure}

\begin{proof}[Proof of estimate \eqref{e.estpli.1} for $p_{li}$]
As for the estimate of $\nabla p_{li}(t)$, we use the equation $\nabla p_{li} = -\partial_t u_{li} + \Delta u_{li}$ and hence by Proposition 5.3 in \cite{MMP17} there exists a constant $C$ such that for all $t\in(0,\infty)$,
\begin{align*}
\| \nabla p_{li} (t) \|_{L^2_{uloc} (\R^3_+)} \leq\ & \| \partial_t u_{li} (t) \|_{L^2_{uloc} (\R^3_+)} + \| \Delta u_{li} (t) \|_{L^2_{uloc} (\R^3_+)}\\
\leq\ & \frac{C\log(e+t)}{t} \| u_0 \|_{L^2_{uloc} (\R^3_+)}.
\end{align*}
The result is proved.
\end{proof}

We now turn to the estimate of $p^{u_0}_{loc}$. Let us again emphasize that this is the term that requires most care. Indeed we cannot rely of the maximal regularity of \cite{GS91} since $u_0$ is no more than locally in $L^2$. Therefore, we have to estimate the integral formula \eqref{e.forpressureu0loc} directly. We need to be careful so as to avoid dealing with singular integrals in time.

\begin{proof}[Proof of estimate \eqref{e.estplocu0} for $p_{loc}^{u_0}$]
The proof is based on a direct estimate of formula \eqref{e.forpressureu0loc'}. 
Applying Minkowski's inequality, we have for fixed $x_3\in\cu(x_0)$,
\begin{align}\label{e.decompnn+1}
&\left\|\int_0^\infty\int_{\R^{2}}q_\lambda(\cdot-z',x_3,z_3)\cdot\chi_{\!_{\, 4}} u_0(z',z_3)dz'dz_3\right\|_{L^2(\cu'(x_0))}\\
\leq\ & \sum_{n=0}^\infty\int_n^{n+1}\left\|\int_{\R^{2}}q_\lambda(\cdot-z',x_3,z_3)\cdot\chi_{\!_{\, 4}} u_0(z',z_3)dz'\right\|_{L^2(\cu'(x_0))}dz_3\nonumber
\end{align}
Let $x\in\cu(x_0)$. Then we have 
\begin{align*}
\lefteqn{\left|\int_{\R^3_+}q_\lambda(x'-z',x_3,z_3)\cdot \chi_{\!_{\, 4}} u_0(z',z_3)dz'dz_3\right|}\\
\leq\ & C\int_{\R^3_+}\frac{e^{-|\lambda|^\frac12z_3}}{(|x'-z'|+x_3+z_3)^{2}}|\chi_{\!_{\, 4}} u_0|dz'dz_3\\
\leq\ & C\int_0^\infty e^{-|\lambda|^\frac12z_3} \int_{\R^{2}}\frac{\mathbf{1}_{(-6,6)^{2}}(x'-z')}{(|x'-z'|+x_3+z_3)^{2}}|\chi_{\!_{\, 4}} u_0(z',z_3)|dz'dz_3.
\end{align*}
Then Young's inequality for convolutions gives, for almost all $z_3\in(0,\infty)$,
\begin{align*}
& \lefteqn{\left\|\int_{\R^{2}}\frac{\mathbf{1}_{(-6,6)^{2}}(x'-z')}{(|x'-z'|+x_3+z_3)^{2}}|\chi_{\!_{\, 4}} u_0(z',z_3)|dz'\right\|_{L^2_{z'}(\R^{2})}}\\
& \leq
\begin{cases} 
& C|\log (x_3+z_3)| \|\chi_{\!_{\, 4}} u_0(\cdot,z_3)\|_{L^2_{z'}(\R^{2})} \qquad x_3 + z_3\leq \frac12,\\
& C \|\chi_{\!_{\, 4}} u_0(\cdot,z_3)\|_{L^2_{z'}(\R^{2})} \qquad x_3 + z_3 >\frac12.
\end{cases}
\end{align*}
Now, combining this with \eqref{e.decompnn+1}, we obtain
\begin{align}
\lefteqn{\left\|\int_0^\infty\int_{\R^{2}}q_\lambda(\cdot-z',x_3,z_3)\cdot\chi_{\!_{\, 4}} u_0(z',z_3)dz'dz_3\right\|_{L^2(\cu(x_0))}}\nonumber\\
\leq\ & \sum_{n=0}^\infty\left\|\int_n^{n+1}\left\|\int_{\R^{2}}q_\lambda(\cdot-z',x_3,z_3)\cdot\chi_{\!_{\, 4}} u_0(z',z_3)dz'\right\|_{L^2(\cu'(x_0))}dz_3\right\|_{L^2(\cu_3(x_0))}\nonumber\\
\begin{split}
\leq\ & C\sum_{n=1}^\infty\left\|\int_n^{n+1} e^{-|\lambda|^\frac12 z_3}\|\chi_{\!_{\, 4}} u_0(\cdot,z_3)\|_{L^2_{z'}(\R^{2})}dz_3\right\|_{L^2(\cu_3(x_0))} \\
& \quad + 
\begin{cases}
& C  \left\| |\log x_3 | \int_0^{1} e^{-|\lambda|^\frac12 z_3}\|\chi_{\!_{\, 4}} u_0(\cdot,z_3)\|_{L^2_{z'}(\R^{2})}dz_3\right\|_{L^2(\cu_3(x_0))} \qquad (x_3\leq \frac12),\\
& C  \left\|\int_0^{1} e^{-|\lambda|^\frac12 z_3}\|\chi_{\!_{\, 4}} u_0(\cdot,z_3)\|_{L^2_{z'}(\R^{2})}dz_3\right\|_{L^2(\cu_3(x_0))}  \qquad (x_3 > \frac12).
\end{cases}
\end{split}\label{e.decompnn+1bis}
\end{align}
We study each term in the sum in the right hand side of \eqref{e.decompnn+1bis}. Of course, most of the terms in the sum on $n$ are $0$, due to the fact that $\chi_{\!_{\, 4}} u_0$ is compactly supported. For $n\geq 1$, we have
\begin{align*}
&\int_n^{n+1} e^{-|\lambda|^\frac12z_3}\|\chi_{\!_{\, 4}} u_0(\cdot,z_3)\|_{L^2_{z'}(\R^{2})}dz_3\leq  e^{-|\lambda|^\frac12n}\|\chi_{\!_{\, 4}} u_0\|_{L^2(\R^{3}_+)}.
\end{align*}
Thus,
\begin{align*}
&\sum_{n=1}^\infty\left\|\int_n^{n+1} e^{-|\lambda|^\frac12z_3}\|\chi_{\!_{\, 4}} u_0(\cdot,z_3)\|_{L^2_{z'}(\R^{2})}dz_3\right\|_{L^2(\cu_3(x_0))}\\
\leq\ & \|\chi_{\!_{\, 4}} u_0\|_{L^2(\R^{3}_+)}\sum_{n=1}^\infty e^{-|\lambda|^\frac12n}\leq\frac{C}{|\lambda|^\frac12}\|\chi_{\!_{\, 4}} u_0\|_{L^2(\R^{3}_+)}.
\end{align*}
As for the last term in the right-hand side of \eqref{e.decompnn+1bis}, the direct computation yields
the bound such as $C|\lambda|^{-\frac14} \|\chi_{\!_{\, 4}} u_0\|_{L^2(\R^{3}_+)}$ by using the H{\"o}lder inequality for the integral $\int_0^1\ldots d z_3$. 
In other words, for $\kappa\in(0,1)$ fixed, the right-hand side in \eqref{e.decompnn+1bis} is bounded by
\begin{equation*}
C\left(\frac{1}{|\lambda|^\frac12}+\frac{1}{|\lambda|^\frac{1}{4}}\right)\|\chi_{\!_{\, 4}} u_0\|_{L^2(\R^{3}_+)}\leq \frac{C}{|\lambda|^\frac{1}{4}}\|\chi_{\!_{\, 4}} u_0\|_{L^2(\R^{3}_+)}
\end{equation*}
for $|\lambda|\geq \kappa$ and a constant $C(\kappa)<\infty$, which is the situation in which we are interested; see \eqref{e.defgammacurve} for the choice of the curve $\Gamma$ on which $q_\lambda$ will be integrated. We are now close to the conclusion. Indeed,
\begin{align*}
\lefteqn{\|p_{loc}^{u_0} (t)\|_{L^2(\cu(x_0))}}\\
\leq\ & \int_\Gamma e^{\Rel(\lambda)t}\left\|\int_{\R^3_+}q_\lambda(x'-z',x_3,z_3)\cdot \chi_{\!_{\, 4}} u_0(z',z_3)dz'dz_3\right\|_{L^2(\cu(x_0))}|d\lambda| \\
\leq\ & \int_\Gamma \frac{e^{\Rel(\lambda)t}}{|\lambda|^\frac{1}{4}}|d\lambda| \|\chi_{\!_{\, 4}} u_0\|_{L^2(\R^{3}_+)}\\
\leq\ & C t^{-\frac{3}{4}} \|\chi_{\!_{\, 4}} u_0\|_{L^2(\R^{3}_+)}.
\end{align*}
The result is proved. 
\end{proof}

\begin{proof}[Proof of estimate \eqref{e.estpharmu0} for $p_{nonloc}^{u_0}$]
We directly estimate the first term in the right hand side of \eqref{e.forharmpressure}. The main huge simplification compared to the estimate for $p_{loc}^{u_0}$ comes from the fact that $(1-\chi_{\!_{\, 4}}) u_0$ is supported away from the singularity of the kernel at $x_0$. Minkowski's inequality implies that, 
\begin{align*}
&\left\|\int_0^\infty\int_{\R^{2}}q_{\lambda,x,x_0}(z',z_3)\cdot (1-\chi_{\!_{\, 4}}) u_0(z',z_3)dz'dz_3\right\|_{L^\infty(\cu(x_0))}\nonumber\\
\leq\ & \sum_{n=0}^\infty\sum_{\eta'\in\Z^{2}}\left\|\int_n^{n+1}\int_{\cu'(\eta')}\frac{e^{-|\lambda|^\frac12z_3}}{(1+|x'-z'|)^3}\cdot|u_0(z',z_3)|dz'dz_3\right\|_{L^\infty(\cu(x_0))}\nonumber\\
\leq\ & C(|\lambda|^{-\frac14} + |\lambda|^{-\frac12} )\|u_0\|_{L^2_{uloc}(\R^{3}_+)}.
\end{align*}
Hence we have as above,
\begin{equation*}
\|p_{nonloc}^{u_0} (t)\|_{L^\infty(\cu(x_0)))}
\leq \int_\Gamma e^{\Rel(\lambda)t}  |\lambda|^{-\frac14} |d\lambda| \|u_0\|_{L^2_{uloc}(\R^{3}_+)}
\leq C t^{-\frac34} \|u_0\|_{L^2_{uloc}(\R^{3}_+)}.
\end{equation*}
The estimate of $\nabla p^{u_0}_{nonloc} (t)$ is obtained in the same manner.
This yields the result.
\end{proof}

\subsection{Estimates for the local pressure terms: proof of Proposition \ref{prop.estlocpressure}}
\label{sec.lpressure}

\begin{proof}[Proof of estimate \eqref{e.estplocuunew}]
The estimate for $p^{u\otimes u}_{loc,H}$ is a consequence of the $L^3$ boundedness of singular integral operators. 
The proof of the estimate for $p^{u\otimes u}_{loc,harm}$ is based on the maximal regularity theory for the Stokes system of Giga and Sohr \cite{GS91}. We recall that by assumption $p_{loc,harm}^{u\otimes u}$ satisfies \eqref{e.psw} with $c=0$. 
The pressure $p_{loc,harm}^{u\otimes u}$ is the pressure of the system \eqref{e.locnse} where the right hand side has been replaced by the divergence-free field:
\begin{equation}\label{e.rhslocnse}
F:=\mathbb P\left(\nabla(\chi_{\!_{\, 4}}^2)u\otimes u\right)+\mathbb P\left(\chi_{\!_{\, 4}}^2u\cdot\nabla u\right).
\end{equation}
We recall that the Helmholtz-Leray projection $\mathbb P$ is bounded on $L^q(\R^3_+)$ for $1<q<\infty$. Hence, it is clear that the least regular term in \eqref{e.rhslocnse} is $\mathbb P\left(\chi_{\!_{\, 4}}^2u\cdot\nabla u\right)\in L^\frac32(0,T;L^\frac98(\R^3_+))$. Therefore, we aim at controlling $F$ in \eqref{e.rhslocnse} in $L^\frac32(0,T;L^\frac98(\R^3_+))$. By \eqref{e.psw} (with $c=0$ by assumption) and estimate (2.22) of \cite[Theorem 2.8]{GS91} we then have
\begin{align*}
&\left\|p^{u\otimes u}_{loc,harm}\right\|_{L^\frac32(0,T;L^\frac32(\cu(x_0)))} +  \left\| \nabla p^{u\otimes u}_{loc,harm}\right\|_{L^\frac32(0,T;L^\frac98(\R^3_+))}\\
\leq\ & C\left\|\nabla p^{u\otimes u}_{loc,harm}\right\|_{L^\frac32(0,T;L^\frac98(\R^3_+))}\leq C\|F\|_{L^\frac32(0,T;L^\frac98(\R^3_+))}.
\end{align*}
Let us now estimate each term in \eqref{e.rhslocnse}. For the first one, Gagliardo-Nirenberg's inequality implies
\begin{align*}
&\left(\int_{\R^3_+}|\mathbb P(\chi_{\!_{\, 4}}^2 u\cdot \nabla u)|^\frac98\right)^\frac89\\
\leq\ &\|\chi_{\!_{\, 4}} u(\cdot,t)\|_{L^\frac{18}7(\R^3_+)}\|\chi_{\!_{\, 4}}\nabla u(\cdot,t)\|_{L^2(\R^3_+)}\\
\leq\ &C\|u(\cdot,t)\|_{L^2(5\cu(x_0))}\|\nabla u(\cdot,t)\|_{L^2(5\cu(x_0))}+C\|u(\cdot,t)\|_{L^2(5\cu(x_0))}^\frac23\|\nabla u(\cdot,t)\|_{L^2(5\cu(x_0))}^\frac43.
\end{align*}
Therefore,
\begin{align*}
\left\|\mathbb P(\chi_{\!_{\, 4}}^2 u\cdot \nabla u)\right\|_{L^\frac32(0,T;L^\frac98(\R^3_+))}\leq\ &CT^\frac16\left\|u\right\|_{L^\infty(0,T;L^2(5\cu(x_0)))}\left\|\nabla u\right\|_{L^2(0,T;L^2(5\cu(x_0)))}\nonumber\\
& +C\left\|u\right\|_{L^\infty(0,T;L^2(5\cu(x_0)))}^\frac23\left\|\nabla u\right\|_{L^2(0,T;L^2(5\cu(x_0)))}^\frac43
\nonumber
\\
\leq\ &
C\left\|u\right\|_{L^\infty(0,T;L^2(5\cu(x_0)))}^2
+C
\left\|\nabla u\right\|_{L^2(0,T;L^2(5\cu(x_0)))}^2,
\end{align*}
where $C$ depends only on $T>0$. We also have
\begin{align*}
\left\|\mathbb P(\nabla(\chi_{\!_{\, 4}}^2)u\otimes u)\right\|_{L^\frac32(0,T;L^\frac98(\R^3_+))}\leq\ & CT^\frac1{15}\left\|u\otimes u\right\|_{L^\frac53(0,T;L^\frac53(5\cu(x_0)))}\\
\leq\ & CT^\frac1{15}\|u\|_{L^\infty(0,T;L^2(5\cu(x_0)))}^2\\
&+CT^\frac1{15}\left\|u\right\|_{L^\infty(0,T;L^2(5\cu(x_0)))}^\frac45\left\|\nabla u\right\|_{L^2(0,T;L^2(5\cu(x_0)))}^\frac65
\\
\leq\ &
C\left\|u\right\|_{L^\infty(0,T;L^2(5\cu(x_0)))}^2
+C
\left\|\nabla u\right\|_{L^2(0,T;L^2(5\cu(x_0)))}^2.
\end{align*}
This completes the proof.
\end{proof}

Let us notice that from the proof we actually have slightly better integrability in space for $p^{u\otimes u}_{loc,harm}$. Indeed,
\begin{equation*}
\left\|p^{u\otimes u}_{loc,harm}\right\|_{L^\frac32(0,T;L^\frac95(\cu(x_0)))}
\leq C\left\|\nabla p^{u\otimes u}_{loc,harm}\right\|_{L^\frac32(0,T;L^\frac98(\R^3_+))}.
\end{equation*}
However, the exponent $\frac32$ in both time and space is enough for our purposes.

\subsection{Estimates for the nonlocal pressure terms: proof of Proposition \ref{prop.nlpressure}}
\label{sec.nlpressure}

One key advantage of estimating the nonlocal part (versus the local part) of the pressure is that we are away from the singularity. 
Hence, since the kernels in \eqref{e.forharmpressure} have enough decay for the integrals to converge for non localized data, we have some room. In particular we can put additional derivatives on the kernels by integrations by parts. We rely on the decomposition for the Helmholtz-Leray projection given in Lemma \ref{lem.termsB}.

\begin{proof}[Proof of estimate \eqref{e.boundnlochelmpressure} for $p_{nonloc,H}^{u\otimes u}$]
The Helmholtz pressure is estimated in the same way as for the whole space. Indeed, the kernel decays as follows: for all $x\in\cu(x_0)$, for all $(z',z_3)\in\R^3_+$,
\begin{equation*}
\left|\nabla^2_zN(x'-z',x_3,z_3)-\nabla^2_zN(x'_0-z',x_{0,3},z_3)\right|\leq \frac{C|x-x_0|}{|x-z|^{4}}\leq \frac{C}{|x-z|^{4}}
\end{equation*}
with a constant $C<\infty$. Hence, for all $x\in\cu(x_0)$, for all $t\in (0,T)$, 
\begin{align*}
|p_{nonloc,H}^{u\otimes u}(x,t)|\leq\ & C\int_{\R^3}\frac{|u\otimes u(z,t)|}{1+|x-z|^4}dz\\
=\ &\sum_{\eta\in\Z^3}\frac{1}{1+|\eta|^4}\int_{\cu(\eta)}|u(\cdot,t)|^2\leq C\|u(\cdot,t)\|_{L^2_{uloc}(\R^3_+)}^2,
\end{align*}
where we have extended $u$ by $0$ on $\R^3\setminus\R^3_+$ as usual.
The estimate of $\nabla p_{nonloc,H}^{u\otimes u}$ is obtained in the same manner. The result is proved.
\end{proof}

\begin{proof}[Proof of estimate \eqref{e.locestpharmuu} for $p_{harm,\leq 1}^{u\otimes u}$]
Let us recall the formula
\begin{align}\label{proof.fromYM.1}
\begin{split}
p_{harm,\leq 1}^{u\otimes u} (x,t) &= \frac1{2\pi i}\int_0^t\int_\Gamma e^{\lambda (t-s)}\int_{\R^3_+}q_{\lambda}(x'-z',x_3,z_3)\\
&\qquad\cdot \Big( \chi_2^2 \mathbb P\nabla\cdot ((1-\chi_{\!_{\, 4}}^2)u\otimes u)\Big)'(z',z_3,s)dz'dz_3d\lambda ds.
\end{split}
\end{align}
The key observation is that, in virtue of the support of the cut-off functions,
\begin{align}\label{e.relchi2ppH}
\chi_2^2 \mathbb P\nabla\cdot ((1-\chi_{\!_{\, 4}}^2)u\otimes u) = \chi_2^2 \nabla p_{nonloc,H}^{u\otimes u}
\end{align}
and the right-hand side has an enough regularity as estimated in the proof of \eqref{e.boundnlochelmpressure} above, i.e., 
\begin{align}\label{proof.fromYM.2}
\|  \chi_2^2 \mathbb P\nabla\cdot ((1-\chi_{\!_{\, 4}}^2)u\otimes u) \|_{L^1\cap L^\infty} \leq C \| u (\cdot, t)\|_{L^2_{uloc} (\R^3_+)}^2.
\end{align}
Let us give the estimate of $\nabla p^{u\otimes u}_{harm,\leq 1}$. 
From the pointwise estimate of $\nabla q_\lambda$ and the compactness of the support of $\chi_2$, we see
\begin{align*}
&|\nabla p^{u\otimes u}_{harm,\leq 1} (x,t)|\\
 \leq\ & C \int_0^t \int_{\R^3_+} \int_\Gamma e^{\Re (\lambda) (t-s) -c |\lambda|^\frac12 z_3} |d\lambda| \frac{1}{(|x'-z'|+x_3+z_3|)^3} |\chi_2^2 \nabla p_{nonloc,H}^{u\otimes u} | d z d s\\
 \leq\ & C \int_0^t (t-s)^{-1+\frac{\sigma}{2}} \int_{\R^3_+} \frac{1}{z_3^\sigma (|x'-z'|+x_3+z_3)^3} |\chi_2^2 \nabla p_{nonloc,H}^{u\otimes u} | d z d s\\
 \leq\ & C \int_0^t (t-s)^{-1+\frac{\sigma}{2}} \int_0^\infty \frac{1}{z_3^\sigma (x_3+z_3)} d z_3 d s \| u\|_{L^\infty (0,t; L^2_{uloc} (\R^3_+))}^2\\
 \leq\ & C t^\frac{\sigma}{2} x_3^{-\sigma}\| u\|_{L^\infty (0,t; L^2_{uloc} (\R^3_+))}^2,
\end{align*}
where $\sigma \in (0,1)$. We used the bound $e^{-|\lambda|^\frac12z_3}\leq C_\sigma z_3^{-\sigma}|\lambda|^{-\frac\sigma 2}$. Notice that $C$ depends only on $\sigma$. 
This proves the derivative estimate in \eqref{e.locestpharmuu}.
The estimate of $p^{u\otimes u}_{harm,\leq 1} (x,t)$ is shown similarly or even easily, by observing for any $\sigma\in (0,1)$ and $0<\ep\ll 1$,
\begin{align*}
&| p^{u\otimes u}_{harm,\leq 1} (x,t)|\\
 \leq\ & C \int_0^t \int_{\R^3_+} \int_\Gamma e^{\Re (\lambda) (t-s) -c |\lambda|^\frac12 z_3} |d\lambda| \frac{1}{(|x'-z'|+x_3+z_3|)^2} |\chi_2^2 \nabla p_{nonloc,H}^{u\otimes u} | d z d s\\
 \leq\ & C \int_0^t (t-s)^{-1+\frac{\sigma}{2}} \int_{\R^3_+} \frac{1}{z_3^\sigma (|x'-z'|+x_3+z_3)^2} |\chi_2^2 \nabla p_{nonloc,H}^{u\otimes u} | d z d s\\
 \leq\ & C \int_0^t (t-s)^{-1+\frac{\sigma}{2}} \int_{\max(x_3-3,0)}^{x_3+3} \frac{1}{z_3^\sigma (x_3+z_3)^\ep} d z_3 d s \| u\|_{L^\infty (0,t; L^2_{uloc} (\R^3_+))}^2\\
 \leq\ & C t^\frac{\sigma}{2}\| u\|_{L^\infty (0,t; L^2_{uloc} (\R^3_+))}^2.
\end{align*}
Here we have used \eqref{proof.fromYM.2} and the compactness of the support of $\chi_2$. The details are omitted here. The proof of \eqref{e.locestpharmuu} is complete.
\end{proof}

\begin{proof}[Proof of estimate \eqref{e.nonlocestpharmuu} for $p_{harm,\geq 1}^{u\otimes u}$]
Considering the expression \eqref{e.leraydivtan}, we notice that there are two types of terms we have to deal with. Let 
$$
v,\, w\in L^\infty\big(0,T;L^2_{uloc}(\R^3_+)\big)\cap L^2\big(0,T; H^1_{0,loc}(\R^3_+)\big)
$$ 
(typically $v=u$ and $w=u$ or variants). We have to estimate
\begin{equation}\label{e.intqlambdanonloc}
\int_0^t\int_\Gamma e^{\lambda (t-s)}\int_{\R^3_+}q_{\lambda,x,x_0}(z',z_3) \cdot (1-\chi_2^2) F(z',z_3,s)dz'dz_3d\lambda ds
\end{equation}
with $F$ replaced by
\begin{equation}
\tag{type A}\label{e.termA}
F_A(z',z_3,s):=\partial_\alpha((1-\chi_{\!_{\, 4}}^2)v\otimes w)\big)'(z',z_3,s)
\end{equation}
for some $\alpha\in\{1,\ldots 3\}$, or
\begin{multline}
\tag{type B}\label{e.termB}
F_B(z',z_3,s)\\
:=m_0(D')\nabla'\otimes\nabla'\int_0^\infty\left[P(|z_3-y_3|)+P(z_3+y_3)\right](1-\chi_{\!_{\, 4}}^2)v\otimes w(z',y_3,s)dy_3
\end{multline}
where $m_0(D')$ is a (tangential) Fourier multiplier homogeneous of order $0$ (see Appendix \ref{sec.lerayproj}). Here $q_{\lambda,x,x_0}$ stands for $q_\lambda(x'-\cdot,x_3,\cdot)-q_\lambda(x_0'-\cdot,x_{0,3},\cdot)$, according to definition \eqref{def.qlambdax0}. The idea for both \eqref{e.termA} and \eqref{e.termB} is to transfer some derivatives from the source term to the kernel. Of course, integrating by parts implies that some derivatives fall on the cut-off $1-\chi_{\!_{\, 2}}^2$. These terms are much simpler to analyze since 
\begin{equation}\label{e.disjsupp}
\dist\left(\supp(\nabla(\chi_{\!_{\, 2}}^2)),\supp(1-\chi_{\!_{\, 4}}^2)\right)\geq 1\quad\mbox{and}\quad \dist\left(\supp(\nabla(\chi_{\!_{\, 2}}^2)),\cu(x_0)\right)\geq 1,
\end{equation}
so that neither the singularity of the Neumann kernel, nor the one of the Helmholtz-Leray projection are seen. 
Below, we focus on the terms where none of the derivatives falls on the cut-off $1-\chi_{\!_{\, 2}}^2$.

\smallskip

\noindent\emph{Terms \eqref{e.termA}.} Notice that $x\in\cu(x_0)$ and by definition of the cut-off $\chi_{\!_{\, 2}}$, the integral in $z$ in \eqref{e.intqlambdanonloc} is on $2\cu(x_0)^c$. Integrating by parts in \eqref{e.intqlambdanonloc} and using the pointwise bound on $\nabla^2 q_\lambda$ proved in \cite[Section 3]{MMP17}, we reduce the problem to estimating
\begin{align*}
&\bigg|\int_0^t\int_\Gamma e^{\lambda (t-s)}\int_{\R^3_+}\partial_\alpha q_{\lambda,x,x_0}(z',z_3)\cdot (1-\chi_{\!_{\, 2}}^2)\big((1-\chi_{\!_{\, 4}}^2)v\otimes w)\big)'(z',z_3,s)dz'dz_3d\lambda ds\bigg|\\
\leq\ & C\int_0^t\int_\Gamma e^{\Rel(\lambda) (t-s)}\int_{\R^3_+}
\frac{e^{-c|\lambda|^\frac12z_3}}{(1+|x'-z'|)^{3}} (|\lambda|^\frac12 + \frac{1}{1+|x'-z'|})|v\otimes w(z',z_3,s)|dz'dz_3|d\lambda| ds\\
\leq\ & C\int_0^t\int_\Gamma |\lambda|^\frac12 e^{\Rel(\lambda) (t-s)}\int_{\R^3_+}
\frac{e^{-c|\lambda|^\frac12z_3}}{(1+|x'-z'|)^{3}} |v\otimes w(z',z_3,s)|dz'dz_3|d\lambda| ds.
\end{align*} 
Here we have used $|\lambda|\geq \kappa$ by the choice of the curve $\Gamma$.
We then have
\begin{align*}
&\int_{\R^3_+}\frac{e^{-c|\lambda|^\frac12z_3}}{(1+|x'-z'|)^{3}}|v\otimes w|dz'dz_3\\
\leq\ &  \sum_{\eta'\in\Z^2}\frac{1}{1+|\eta'|^{3}}\int_{0}^\frac12 e^{-c|\lambda|^\frac12 z_3}\int_{\cu'(\eta')}|v\otimes w|dz'dz_3\\
& + \sum_{\eta'\in\Z^2}\frac{1}{1+|\eta'|^{3}}\int_\frac12^\infty e^{-c|\lambda|^\frac12z_3}\int_{\cu'(\eta')}|v\otimes w|dz'dz_3\\
\leq\ &  \sum_{\eta'\in\Z^2}\frac{1}{1+|\eta'|^{3}}\int_{0}^\frac12 e^{-c|\lambda|^\frac12 z_3}\int_{\cu'(\eta')}|v\otimes w|dz'dz_3 \\
& + C\sum_{n=1}^\infty e^{-c'|\lambda|^\frac12n} \|v(\cdot,s)\|_{L^2_{uloc}(\R^3_+)}\|w(\cdot,s)\|_{L^2_{uloc}(\R^3_+)}\\
\leq\ & \sum_{\eta'\in\Z^2}\frac{1}{1+|\eta'|^{3}}\int_{0}^\frac12 e^{-c|\lambda|^\frac12 z_3}\int_{\cu'(\eta')}|v\otimes w|dz'dz_3  \\
& + e^{-c'|\lambda|^\frac12} \|v(\cdot,s)\|_{L^2_{uloc}(\R^3_+)}\|w(\cdot,s)\|_{L^2_{uloc}(\R^3_+)}.
\end{align*}
Here $c'=c/2$. 
We eventually get, for almost every $t\in (0,T)$,
\begin{align*}
&\bigg \|\int_0^t\int_\Gamma e^{\lambda (t-s)}\int_{\R^3_+}\partial_\alpha q_{\lambda,x,x_0}(z',z_3)\cdot (1-\chi_{\!_{\, 2}}^2)\big((1-\chi_{\!_{\, 4}}^2)v\otimes w)\big)'(z',z_3,s)dz'dz_3d\lambda ds\bigg \|_{L^\infty (\cu(x_0))}\\
\leq\ & C \sum_{\eta'\in\Z^2} \frac{1}{1+|\eta'|^{3}} \int_{\cu'(\eta')}  \int_{0}^\frac12 \int_\Gamma |\lambda|^\frac12 e^{-c|\lambda|^\frac12 z_3}   \int_0^t  e^{\Re(\lambda) (t-s)}|v (s) \otimes w (s) | \, ds  |d \lambda| dz_3 d z' \\
& + \, C \int_0^t \|v(s) \|_{L^2_{uloc}(\R^3_+)}\|w (s) \|_{L^2_{uloc}(\R^3_+)} \, ds.
\end{align*}
We notice here that there exists $C(T)>0$ such that $e^{\Re(\lambda) (t-s)} \leq C e^{-c|\lambda| (t-s)}$ for $0<t-s\leq T$ and $\lambda \in \Gamma$.
Therefore, by the Young inequality for convolution in the time variable,
the norm of  $L^2(0,T; L^\infty (\cu(x_0)))$ of this term is bounded from above as 
\begin{align*}
& C  \sum_{\eta'\in\Z^2}  \frac{1}{1+|\eta'|^{3}} \int_{\cu'(\eta')}  \int_{0}^\frac12 \int_\Gamma e^{-c|\lambda|^\frac12 z_3} \int_0^T |v (s) \otimes w (s) | \, ds  |d \lambda| dz_3 d z'  \\
& + \,  C T \int_0^T \|v(s) \|_{L^2_{uloc}(\R^3_+)}\|w (s) \|_{L^2_{uloc}(\R^3_+)} \, ds\\
\leq\ & C  \sum_{\eta'\in\Z^2}  \frac{1}{1+|\eta'|^{3}} \int_{\cu'(\eta')}  \int_{0}^\frac12 \int_0^T z_3^{-2} |v (s) \otimes w (s) | \, ds dz_3 d z'  \\
& + \, C \int_0^T \|v(s) \|_{L^2_{uloc}(\R^3_+)}\|w (s) \|_{L^2_{uloc}(\R^3_+)} \, ds\\
\leq\ & C  \sum_{\eta'\in\Z^2}  \frac{1}{1+|\eta'|^{3}} \int_0^T \|\partial_3 v (s) \|_{L^2 (\cu(\eta',0))} \| \partial_3 w (s) \|_{L^2(\cu(\eta',0))} \, ds  \\
& + \, C  \int_0^t \|v(s) \|_{L^2_{uloc}(\R^3_+)}\|w (s) \|_{L^2_{uloc}(\R^3_+)} \, ds,
\end{align*}
where we have used the Hardy inequality about $z_3$ variable and $\cu(\eta',0)=\cu'(\eta')\times [0,1/2]$. Then the first term is clearly bounded from above by 
$$C \sup_{\eta'\in \Z^2}  \| \nabla v \|_{L^2(0,T; L^2(\cu (\eta',0)))}  \| \nabla w \|_{L^2(0,T; L^2(\cu (\eta',0)))},$$ as desired.  

\smallskip

\noindent\emph{Terms \eqref{e.termB}.} We rely on Lemma \ref{lem.termsB} to estimate these terms. Using this lemma we have
\begin{align}
\label{e.splitb12}
\begin{split}
&\bigg|\int_0^t\int_\Gamma e^{\lambda (t-s)}\int_{\R^3_+}q_{\lambda,x,x_0}(z',z_3)\cdot (1-\chi_{\!_{\, 2}}^2)F_Bdz'dz_3d\lambda ds\bigg|\\
\leq\ &\bigg|\int_0^t\int_\Gamma e^{\lambda (t-s)}\int_{\R^3_+}q_{\lambda,x,x_0}(z',z_3)\cdot (1-\chi_{\!_{\, 2}}^2)\mathcal B_1dz'dz_3d\lambda ds\bigg|\\
&+\bigg|\int_0^t\int_\Gamma e^{\lambda (t-s)}\int_{\R^3_+}\nabla'^{2}q_{\lambda,x,x_0}(z',z_3)\cdot(1-\chi_{\!_{\, 2}}^2)\mathcal B_2dz'dz_3d\lambda ds\bigg|\\
&+\quad\mbox{commutator terms}.
\end{split}
\end{align} 
The terms designated by ``commutator terms'' correspond to one or two derivatives falling on the the cut-off $1-\chi_{\!_{\, 2}}^2$. We explained above that these terms are much easier to handle so we focus on the two first terms in the right hand side of \eqref{e.splitb12}. 
We now use the bounds on $\mathcal B_1$ and $\mathcal B_2$ proved in Lemma \ref{lem.termsB} below. The estimate is similar to the one for \eqref{e.termA} above. It is actually simpler, since we do not have here the additional 	factor $|\lambda|^\frac12$. We sketch how to estimate the first term in the right hand side above. Using Young's inequality for convolutions in time, we have that the $L^2(0,T; L^\infty (\cu(x_0)))$ of
\begin{align*}
&\left|\int_0^t\int_\Gamma e^{\lambda (t-s)}\int_{\R^d_+}q_{\lambda,x,x_0}(z',z_3)\cdot (1-\chi_{\!_{\, 2}}^2)\mathcal B_1dz'dz_3d\lambda ds\right|
\end{align*}
is bounded by
\begin{align*}
& C  \sum_{\eta'\in\Z^2}  \frac{1}{1+|\eta'|^{3}} \int_{\cu'(\eta')}  \int_{0}^\frac12 \int_0^T z_3^{-1} |v (s) \otimes w (s) | \, ds dz_3 d z'  \\
& + \, C \|v\|_{L^\infty(0,T;L^2_{uloc}(\R^3_+))}\|w\|_{L^\infty(0,T;L^2_{uloc}(\R^3_+))}\\
\leq\ & C  \sum_{\eta'\in\Z^2}  \frac{1}{1+|\eta'|^{3}} \int_0^T \|\partial_3 v (s) \|_{L^2 (\cu(\eta',0))} \| w (s) \|_{L^2(\cu(\eta',0))} \, ds  \\
& + \, C \int_0^T \|v(s) \|_{L^2_{uloc}(\R^3_+)}\|w (s) \|_{L^2_{uloc}(\R^3_+)} \, ds.
\end{align*}
The second term in the right hand side of \eqref{e.splitb12} is even simpler to handle, since the kernel $\nabla'^{2}q_{\lambda,x,x_0}$ has even more spatial decay.
The same is true for the estimate of  $\nabla p^{u\otimes u}_{harm}$. 
This concludes the proof of the estimate.
\end{proof}

\section{Properties of weak solutions}\label{sec.property.weak}
In this section we show the basic properties of local energy weak solutions in the sense of Definition \ref{def.weakles}.
The goal is to prove that any local energy weak solution is a mild solution and admits additional regularity which enable us to apply the $\ep$-regularity theorem for the half-space \cite{Ser02,SSS04,SS14}.
The $\ep$-regularity for solutions of \eqref{e.nse} is the key in Section \ref{sec.global.weak} for the global existence and Section \ref{sec.blowup} for the blow-up result. The first result of this section is stated as follows,
where the uniqueness result for solutions to the Stokes system \cite[Theorem 5]{MMP17} plays a crucial role.
\begin{proposition}\label{prop.property.weaksol} Let $(u,p)$ be any local energy weak solutions with initial data $u_0\in \mathcal{L}^2_{uloc,\sigma} (\R^3_+)$ in the sense of Definition \ref{def.weakles}. Then 
\begin{align}
u(t) = e^{-t{\bf A}} u_0  - \int_0^t e^{-(t-s){\bf A}} \mathbb{P}\nabla \cdot (u\otimes u) ds, 
\end{align}
$p$ admits the decomposition of \eqref{e.decomppressure}, and 
\begin{align}
\partial_t u, ~ \nabla^2 u, ~ \nabla p \in L^\frac32_{loc} ((0,T]; L^\frac98_{loc} (\overline{\R^3_+})).
\end{align}
\end{proposition}

\begin{proof} Set 
\begin{align*}
v(t) = e^{-t{\bf A}} u_0 - \int_0^t e^{-(t-s){\bf A}} \mathbb{P} \nabla \cdot (u \otimes u ) d s.
\end{align*}
Let $\{u_0^\ep\}_{0<\ep<1}\subset C_{0,\sigma}^\infty (\R^3_+)$ be a sequence such that $u_0^\ep\rightarrow u_0$ in $L^2_{uloc} (\R^3_+)^3$ as $\ep\rightarrow 0$, and let $\varphi_\ep\in C_0^\infty (\R^3)$ be a smooth cut-off satisfying $\varphi_\ep=1$ on $|x|\leq \frac{1}{\ep}$ and $\varphi=0$ on $|x|\geq \frac{2}{\ep}$.
Then we set 
\begin{align*}
v^\ep(t) = e^{-t{\bf A}} u_0^\ep - \int_0^t e^{-(t-s){\bf A}} \mathbb{P} \nabla \cdot (\varphi_\ep u \otimes u ) d s.
\end{align*}
We see that $\|e^{-t{\bf A}} u_0^\ep-e^{-t{\bf A}} u_0\|_{L^2_{uloc}(\R^3_+)}\leq C \| u_0^\ep - u_0 \|_{L^2_{uloc}(\R^3)}\rightarrow 0$ as $\ep\rightarrow 0$ and  that, for $q\in [1,\frac32)$,
\begin{align}\label{proof.prop.property.weaksol.1}
& \|  \int_0^t e^{-(t-s){\bf A}} \mathbb{P} \nabla \cdot (\varphi_\ep u \otimes u -u\otimes u ) d s\|_{L^{q}_{uloc}(\R^3_+)} \nonumber \\
& \leq \int_0^t \|  e^{-(t-s){\bf A}} \mathbb{P} \nabla \cdot (\varphi_\ep u \otimes u -u\otimes u )\|_{L^{q}_{uloc} (\R^3_+)}  d s \nonumber \\
& \leq C\int_0^t (t-s)^{-\frac12-\frac32(1-\frac1q)} \| (1-\varphi_\ep) u\otimes u \|_{L^{1}_{uloc} (\R^3_+)} ds \nonumber \\
& \text{ (by  applying \cite[Theorem 3]{MMP17})}\nonumber\\
& \leq C \int_0^T |t-s|^{-2+\frac{3}{2q}} \| u \|_{L^2_{uloc}(\R^3_+)} \| (1-\varphi_\ep) u \|_{L^2_{uloc}(\R^3_+)} ds.
\end{align}
Since $u\in L^\infty (0,T; \mathcal{L}^2_{uloc,\sigma} (\R^3_+))$ we have $\| (1-\varphi_\ep) u (s) \|_{L^2_{uloc}(\R^3_+)} \rightarrow 0$ as $\ep\rightarrow 0$ for $a.e.\, s\in (0,T)$, while we have the trivial bound $\| (1-\varphi_\ep) u (s) \|_{L^2_{uloc}(\R^3_+)}\leq \| u (s) \|_{L^2_{uloc}(\R^3_+)}$. Thus, splitting 
the time integral around the singularity $s=t$ in \eqref{proof.prop.property.weaksol.1}  
and applying the Lebesgue convergence theorem 
in the region $|t-s| \ge \delta$ with small $\delta>0$, we obtain
\begin{align*}
& \int_0^t e^{-(t-s){\bf A}} \mathbb{P} \nabla \cdot (\varphi_\ep u \otimes u ) ds \,  \rightarrow  \, \int_0^t e^{-(t-s){\bf A}} \mathbb{P} \nabla \cdot (u\otimes u ) d s \\
& \qquad \qquad \text{in }~L^\infty (0,T; L^q_{uloc} (\R^3_+)), \qquad q\in [1,\textstyle{\frac32}).
\end{align*}
As a consequence, we have 
\begin{align*}
v^\ep \rightarrow v \quad \text{as } ~ \ep\rightarrow 0 \qquad \text{in }~L^\infty (0,T; L^q_{uloc} (\R^3_+)), \qquad q\in [1,\textstyle{\frac32}).
\end{align*}
In virtue of this convergence the local regularity of $v$ and of the associated pressure is obtained from the one of $(v^\ep,p^\ep)$ by taking the limit.
Since $\nabla \cdot (\varphi_\epsilon u\otimes u )\in L^\frac54 ((0,T)\times \R^3_+)$ for each $\epsilon \in (0,1)$ by the Sobolev embedding theorem and the regularity assumption on the local energy weak solutions, each action of $e^{-(t-s){\bf A}}$, $\mathbb{P}$, and $\nabla \cdot$ in the definition of $v^\ep$ is well-defined in a classical $L^q$ framework. Moreover, the maximal regularity gives the bound $\partial_t v^\ep, \nabla ^2 v^\ep\in L^\frac 54 ((0,T)\times \R^3_+)$ and $v^\ep$ satisfies 
\begin{align}\label{proof.prop.property.weaksol.2}
\partial_t v^\ep - \Delta v^\ep + \nabla p^\ep = - \nabla \cdot (\varphi_\ep u\otimes u),\quad \nabla \cdot v^\ep =0 \qquad (t,x) \in (0,T)\times \R^3_+   
\end{align}
and $v^\ep|_{\partial\R^3_+}=0$ in $t\in (0,T)$, $v^\ep|_{t=0} =u_{0,\ep}$.
Here $p^\ep \in L^\frac54 ((0,T)\times \R^3_+)$ is the pressure associated with $v^\ep$, which admits the representation and the decomposition as in \eqref{e.decomppressure} with $u_0$ and $u\otimes u$ simply replaced by $u_0^\ep$ and $\varphi_\ep u\otimes u$, respectively: $p^\ep = p^\ep_{li} +p^{\varphi_\ep u\otimes u}_{loc} + p^{\varphi_\ep u\otimes u}_{nonloc} = p^{u_{0,\ep}}_{loc} + p^{u_{0,\ep}}_{nonloc} + p^{\varphi_\ep u\otimes u}_{loc,harm}+ p_{loc,H}^{\varphi_\ep u\otimes u} + p^{\varphi_\ep u\otimes u}_{loc} + p_{nonloc,H}^{\varphi_\ep u\otimes u} +  p^{\varphi_\ep u\otimes u}_{harm,\leq 1} +  p^{\varphi_\ep u\otimes u}_{harm,\geq 1}$.
Hence, each term in this decomposition satisfies the similar estimates in Propositions \ref{prop.estlipressure}, \ref{prop.estlocpressure}, and \ref{prop.nlpressure}, which are uniform in $\ep\in (0,1)$. More precisely, for all $T>0$ there exists a constant $C(T)<\infty$ such that for all $t\in(0,T)$,
\begin{align*}
\frac t{\log(e+t)} \| \nabla p_{li}^\ep  (t)\|_{L^2_{uloc}(\R^3_+)} & \leq C \| u_{0,\ep} \|_{L^2_{uloc} (\R^3_+)},\\
t^{\frac34} \|p_{loc}^{u_{0,\ep}} (t)\|_{L^2(\cu(x_0))}  & \leq C\|u_{0,\ep}\|_{L^2_{uloc} (\R^3_+)},\\
t^{\frac34} \|p_{nonloc}^{u_{0,\ep}} (t)\|_{L^2(\cu(x_0))} + t^\frac34 \| \nabla p_{nonloc}^{u_{0,\ep}} (t) \|_{L^2 (\R^3_+)} & \leq C\|u_{0,\ep}\|_{L^2_{uloc} (\R^3_+)}.
\end{align*}
Furthermore, for all $T>0$ and $1\leq q<\infty$, there exist $C(T), C_q(T)<\infty$, such that for all $t\in(0,T)$,
\begin{align*}
& \left\|p^{\varphi_\ep u\otimes u}_{loc,H}\right\|_{L^\frac32(0,T;L^\frac32(\cu(x_0)))}+ \left\|p^{\varphi_\ep u\otimes u}_{loc,harm}\right\|_{L^\frac32(0,T;L^\frac32(\cu(x_0)))} + \left\| \nabla p^{\varphi_\ep u\otimes u}_{loc,harm}\right\|_{L^\frac32(0,T;L^\frac98(\R^3_+))}\\
\leq\  &
C\sup_{\eta \in \Z^3_+}\left(\left\|u\right\|_{L^\infty(0,T;L^2(\cu(\eta)))}^2
+\left\|\nabla u\right\|_{L^2(0,T;L^2(\cu(\eta)))}^2\right),
\end{align*}
and 
\begin{align*}
& \|p_{nonloc,H}^{\varphi_\ep u\otimes u} (\cdot,t)\|_{L^\infty(\cu(x_0))} + \|\nabla p_{nonloc,H}^{\varphi_ \ep u\otimes u} (\cdot,t)\|_{L^\infty(\cu(x_0))} \leq C\|u(\cdot,t)\|_{L^2_{uloc}(\R^3_+)}^2,\\
& \|p_{harm,\leq 1}^{\varphi_\ep u\otimes u} (\cdot,t)\|_{L^\infty(\cu(x_0))} + \|\nabla p_{harm,\leq 1}^{\varphi_\ep u\otimes u} (\cdot,t)\|_{L^q (\cu(x_0))} \leq  C_q \|u\|_{L^\infty(0,t;L^2_{uloc}(\R^3_+))}^2,\\
& \|p_{harm,\geq 1}^{\varphi_\ep u\otimes u} \|_{L^2 (0,T; L^\infty(\cu(x_0)))} + \|\nabla p_{harm,\geq 1}^{\varphi_\ep u\otimes u} \|_{L^2 (0,T; L^\infty(\cu(x_0)))}\\
\leq\ & C( \sup_{\eta\in \Z^3} \| \nabla u \|_{L^2 (0,T; L^2(\cu(\eta)))}^2 +  \|u\|_{L^\infty(0,T;L^2_{uloc}(\R^3_+))}^2 ).
\end{align*}
Here the constant $C$ depends on $T$ but is independent of $\ep\in (0,1)$.
Since we have obtained the estimates for the pressure in positive time, by regarding $\nabla p^\ep$ as a given forcing term in \eqref{proof.prop.property.weaksol.2},
we can apply the local regularity estimate of the inhomogeneous heat equations, which results in, for any $\delta\in (0,T)$,
\begin{align}\label{proof.prop.property.weaksol.3}
\begin{split}
& \| \partial_t v^\ep \|_{L^\frac32 (\delta,T; L^\frac98 (\cu(x_0)))} + 
\| \nabla^2 v^\ep \|_{L^\frac32 (\delta,T; L^\frac98 (\cu(x_0)))} \\
& \leq  C\Big (\left\|u\right\|_{L^\infty(0,T;L^2_{uloc}(\R^3_+))}^2
+   \sup_{\eta \in \Z^3_+} \left\|\nabla u\right\|_{L^2(0,T;L^2(\cu(\eta)))}^2 \\
& \quad  + \| \nabla  v^\ep \|_{L^\frac32 (\frac{\delta}{2},T; L^\frac98 (2\cu(x_0)))}  + \| v^\ep \|_{L^\frac32 (\frac{\delta}{2},T; L^\frac98 (2\cu(x_0)))}  \Big ).
\end{split}
\end{align} 
Here the constant $C$ depends only on $T$ and $\delta$.
By using the bound 
\begin{align*}
\| v^\ep\|_{L^\infty (0,T; L^q_{uloc}(\R^3_+))} \leq C \| u\|_{L^\infty (0,T; L^2_{uloc} (\R^3_+))}, \qquad q\in 
[1,\textstyle{\frac32})
\end{align*}
which follows as in the computation of \eqref{proof.prop.property.weaksol.1},
we also have 
\begin{align*}
& \| \nabla  v^\ep \|_{L^\frac32 (\frac{\delta}{2},T; L^\frac98 (2\cu(x_0)))}  + \| v^\ep \|_{L^\frac32 (\frac{\delta}{2},T; L^\frac98 (2\cu(x_0)))}  \\
& \leq C \big ( \left\|u\right\|_{L^\infty(0,T;L^2_{uloc} (\R^3_+))}^2
+ \sup_{\eta \in \Z^3_+} \left\|\nabla u\right\|_{L^2(0,T;L^2(\cu(\eta)))}^2 \big).
\end{align*}
Thus we conclude from \eqref{proof.prop.property.weaksol.3} and by taking the limit $\ep\rightarrow 0$ that 
\begin{align}\label{proof.prop.property.weaksol.4}
\begin{split}
& \| \partial_t v \|_{L^\frac32 (\delta,T; L^\frac98 (\cu(x_0)))} + 
\| \nabla^2 v \|_{L^\frac32 (\delta,T; L^\frac98 (\cu(x_0)))} \\
& \leq C\big ( \left\|u\right\|_{L^\infty(0,T;L^2_{uloc}(\R^3_+))}^2 +\sup_{\eta \in \Z^3_+}\left\|\nabla u\right\|_{L^2(0,T;L^2(\cu(\eta)))}^2\big),
\end{split}
\end{align}
and $v$ satisfies 
\begin{align*}
\partial_t v - \Delta v + \nabla p_v = - \nabla \cdot (u\otimes u), \quad \nabla \cdot v=0 \qquad (t,x) \in (0,T)\times \R^3_+
\end{align*}
and $v|_{\partial\R^3_+}=0$ in $(0,T)$ and $v|_{t=0}=u_0$.
Here $p_v$ is the associated pressure for $v$, which is obtained as a limit of $p^\ep$.
Then $p_v$ satisfies the representation and the decomposition of \eqref{e.decomppressure},
and each term in \eqref{e.decomppressure} satisfies the estimates in Propositions \ref{prop.estlipressure}, \ref{prop.estlocpressure}, and \ref{prop.nlpressure}.
It is  easy to see that the map $[0,T) \mapsto \int_{\R^3_+} v (x,t)  \cdot \varphi  (x) dx$ belongs to $C([0,T))$ for any $\varphi \in C_0^\infty (\overline{\R^3_+})^3$. 
Indeed, the linear term $e^{-t{\bf A}} u_0$ belongs to $C([0,\infty); L^2_{uloc,\sigma} (\R^3_+))$ since $\{e^{-t{\bf A}}\}_{t\geq 0}$ defines a bounded analytic semigroup in $L^2_{uloc,\sigma} (\R^3_+)$ (by \cite[Theorem 2]{MMP17}) and is a $C_0$-analytic semigroup in $L^2_\sigma (\R^3_+)$, which implies that $e^{-t{\bf A}} u_0\in C([0,\infty); L^2_{uloc,\sigma} (\R^3_+))$ for $u_0\in \mathcal{L}^2_{uloc,\sigma} (\R^3_+)$ by the density argument.
On the other hand, the inhomogeneous term in the definition of $v$ belongs to $C([0,T); L^q_{uloc,\sigma} (\R^3_+))$ for $q\in (1,\frac32)$: this is proved by using the fact that $\{e^{-t{\bf A}}\}_{t\geq 0}$ is a bounded analytic semigroup in $L^q_{uloc,\sigma} (\R^3_+)$ for $1<q<\infty$ again by \cite[Theorem 2]{MMP17} and the estimate for $e^{-t{\bf A}}\mathbb{P}\nabla \cdot$ in \cite[Theorem 3]{MMP17} as in the proof of \eqref{proof.prop.property.weaksol.1}. The details are omitted here.
Thus, from the uniqueness result of the weak solution to the Stokes system, proved in \cite[Theorem 5]{MMP17}, we have $u=v$ and also $p=p_v$ (up to some constant). The proof is complete.
\end{proof}

In virtue of the additional regularity obtained in Proposition \ref{prop.property.weaksol}, the $\ep$-regularity theorem by Seregin et al \cite{SSS04} can be applied for our class of weak solutions.

\begin{theorem}[{\cite[Theorem 1.1]{SSS04}, \cite[Theorem 14.4]{lemariebook}}]\label{thm.ep.regularity} There exist $\ep_*>0$ and $R_*>0$ such that the following statement holds.
Let $(u,p)$ be any local energy weak solution to \eqref{e.nse} with initial data $u_0\in \mathcal{L}^2_{uloc,\sigma} (\R^3_+)$ in the sense of Definition \ref{def.weakles}. Let $(t_0,x_0)\in (0,T]\times \overline{\R^3_+}$.
 If 
\begin{align*}
\frac{1}{\rho_*^2}\int_{t_0-\rho_*^2}^{t_0} \int_{B (x_0,\rho_*)\cap \R^3_+} \Big ( |u|^3 + |p|^\frac32 \Big ) d x d t <\ep_* 
\end{align*}
for some $\rho_*\in (0, \min\{R_*, \sqrt{t_0}\}]$ then $u$ is H{\"o}lder continuous on $[t_0-\frac{\rho_*}{16},t_0] \times \overline{B (x_0,\frac{\rho_*}{8})\cap \R^3_+}$.
\end{theorem}

\begin{proof} When $x_0\in \partial\R^3_+$ then the result follows from Seregin et al \cite[Theorem 1.1]{SSS04}, and thus the same is shown when $\{x\in \overline{\R^3_+}~|~0\leq x_{0,3} \leq \min\{R_*, \sqrt{t_0}\}\}$ by taking $R_*$ smaller than in \cite[Theorem 1.1]{SSS04} if necessary. When $x_0\in \{x\in \R^3_+~|~x_{0,3}\geq \min\{R_*,\sqrt{t_0}\}\}$ then the statement falls into the interior $\ep$-regularity theorem, see, e.g., \cite[Theorem 14.4]{lemariebook}. The proof is complete.
\end{proof}

\section{Decay of the Leray solutions at $\infty$}
\label{sec.decay}

Let $T>0$ and $\delta>0$ be fixed. For the whole section, we work under Assumption \ref{assu.A}. The goal in this section is to prove Theorem \ref{prop.decayinfty}, i.e. to show that if the initial data has some decay at $\infty$, then any weak local energy solution $u$ to \eqref{e.nse}, with initial data $u_0$, will decay at infinity. The assumption $u_0\in\mathcal L^2_{uloc,\sigma}(\R^3_+)$ is redundant for solutions in the sense of Definition \ref{def.weakles}. However, we choose to add it in the statement of Theorem \ref{prop.decayinfty} in order to stress that this is the key to the decay of the solution at space infinity. It is easy to get that the third term in the left hand side of \eqref{e.decayatinfty} is bounded by the two first terms. Indeed
\begin{multline}\label{e.gagnirulocR}
\int_0^t\int_{\cu(\eta)}|\vartheta_R u|^3\\
\leq C\left(\int_0^t\left(\int_{\cu(\eta)}|\vartheta_R u(x,s)|^2dx\right)^3ds\right)^\frac14\left(\int_0^t\int_{\cu(\eta)}|\vartheta_R u|^2+|\vartheta_R \nabla u|^2+R^{-2}|u|^2\right)^\frac34,
\end{multline}
for all $t\in[0,T]$ and all $\eta\in\Z^3_+$. Remember that $\vartheta_R$ cuts off the part of $u$ around $0$. Hence display \eqref{e.decayatinfty} shows that the local energy of $u$ goes to zero at spatial infinity. Let us denote by $A_{T,\delta}$ the constant in the right hand side of the a priori estimate \eqref{e.apriorilenassu}. We also define the quantities
\begin{align*}
\alpha_R(t):=
\sup_{\eta\in\Z^3_+}\int_{\cu(\eta)}|\vartheta_Ru(\cdot,t)|^2, &\qquad
\beta_R(t):=
\sup_{\eta\in\Z^3_+}
\int_0^{t}\int_{\cu(\eta)}|\vartheta_R\nabla u|^2,\\
\gamma_R(t):=\sup_{\eta\in\Z^3_+}
\left(\int_0^{t}\int_{\cu(\eta)}|\vartheta_Ru|^3\right)^\frac23&.
\end{align*}
Notice that our definition of $\gamma_R$ differs from the one of \cite{KS07}. Our quantity has the same homogeneity with respect to $u$ as $\alpha_R$ and $\beta_R$. By \eqref{e.gagnirulocR} it is straightforward to see that
\begin{equation}\label{e.estgammaR}
\gamma_R(t)\leq C\left(\int_0^t\alpha_R^3(s)ds
\right)^\frac16\left(\int_0^t\alpha_R(s)ds+\beta_R(t)+R^{-2}A_{T,\delta}\right)^\frac12,
\end{equation}
for all $t\in[0,T]$. The following inequalities will be useful to give a simpler form to our estimates: for all $t\in[0,T]$,
\begin{subequations}\label{e.eqsimpl}
\begin{align}
\int_0^t\alpha_R^3(s)ds\leq\ & T^\frac67\left(\int_0^t\alpha_R^{21}(s)ds\right)^\frac17,\\
\int_0^t\alpha_R(s)ds\leq\ & T^\frac{20}{21}\left(\int_0^t\alpha_R^{21}(s)ds\right)^\frac1{21},\\
\int_0^t\alpha_R^\frac34(s)ds\leq\ & T^\frac{27}{28}\left(\int_0^t\alpha_R^{21}(s)ds\right)^\frac1{28}.
\end{align}
Moreover, it follows from \eqref{e.estgammaR} that for all $\delta>0$, there exists a constant $C(\delta,T)<\infty$ such that
\begin{equation}
\gamma_R(t)\leq \delta\beta_R(t)+C(\delta,T)\left(\int_0^t\alpha_R^{21}(s)ds\right)^\frac1{21}+CA_{T,\delta}R^{-2},
\end{equation}
\end{subequations}
for all $t\in[0,T]$.

\smallskip

An estimate similar to \eqref{e.decayatinfty} was derived by Kikuchi and Seregin in \cite{KS07} for $\R^3$ (see also \cite[Chapter 32]{lemariebook}). Here the difficulty comes again from the pressure estimates, which are more subtle than in the whole space. Our main tool for the proof of Theorem  \ref{prop.decayinfty} is the following a priori estimate.

\begin{lemma}\label{lem.aprioridecay}
Assume that Assumption \ref{assu.A} holds. Then for all $u_0\in\mathcal L^2_{uloc,\sigma}(\R^3_+)$, there exists a constant $C(T,\|u_0\|_{L^2_{uloc}})<\infty$ such that all weak local energy solution $u$ to \eqref{e.nse} on $Q_T$ in the sense of Definition \ref{def.weakles} with initial data $u_0$ satisfies 
for all $R\geq 1$, for all $t\in[0,T]$,
\begin{equation}\label{e.lem.aprioridecay}
\alpha_R(t)+\beta_R(t)\leq C(T,u_0)\left(\left(\int_0^t\alpha_R^{21}(s)ds\right)^\frac1{21}+R^{-1}(\log R)+\|\vartheta_Ru_0\|_{L^2_{uloc}(\R^3_+)}\right).
\end{equation}
\end{lemma}

This a priori estimate will be established below. With Lemma \ref{lem.aprioridecay}, we can now prove Proposition \ref{lem.aprioridecay}.

\begin{proof}[Proof of Proposition \ref{e.decayatinfty}]
Denoting by 
\begin{equation*}
Y_R(t):=\alpha_R^{21}(t)+\beta_R^{21}(t),
\end{equation*}
for all $t\in(0,T)$, we get from \eqref{e.lem.aprioridecay} the following differential inequality
\begin{equation}\label{e.diffineqY}
Y_R(t)\leq C\left(\int_0^tY_R(s)ds+R^{-1}(\log R)+\|\vartheta_Ru_0\|_{L^2_{uloc}(\R^3_+)}\right),
\end{equation}
with a constant $C(T,u_0)<\infty$. The convergence result \eqref{e.decayatinfty} follows now from a classical Gronwall-type argument (see \cite{KS07} for a similar argument). 
\end{proof}

A remarkable point is that the differential inequality \eqref{e.diffineqY} derived from \eqref{e.lem.aprioridecay} is linear. This comes from the fact that it is an inequality involving $\vartheta_Ru$. All the nonlinear terms, e.g. $I_3$ below, have a structure ressembling
\begin{equation*}
\vartheta_R^2|u|^2u=|\vartheta_Ru|^2u.
\end{equation*}
The remaining term $u$ which is not paired with $\vartheta_R$ will be estimated by the a priori estimate \eqref{e.apriorilenassu}.

\smallskip

The remainder of this section is devoted to the proof of Lemma \ref{lem.aprioridecay}. We assume Assumption \ref{assu.A}. Let $u_0\in\mathcal L^2_{uloc}(\R^3_+)$ such that $\|u_0\|_{L^2_{uloc}(\R^3_+)}\leq\delta$ and $u$ be any solution to \eqref{e.nse} in $Q_T$ in the sense of Definition \ref{def.weakles} with initial data $u_0$. For fixed $x_0\in\R^3_+$ and $R\geq 1$, the idea is to test the local energy inequality with $\psi:=\vartheta_R^2\chi_{\!_{\, x_0,1}}\in C^\infty_c(\R^d_+)$. This test function is constant in time.  According to Remark \ref{rem.def.weakles} (3), such test functions are admissible in the local energy inequality. Let us emphasize that the strong convergence \eqref{e.cvinit} is fundamental here and enables to transfer the decay at infinity of the initial data $u_0$ to the solution $u$. 
We have from \eqref{e.locenineqbis}
\begin{align}\label{e.locenineqR}
\begin{split}
&\int_{\R^3_+}\psi|u(\cdot,t)|^2+2\int_0^t\int_{\R^3_+}\psi|\nabla u|^2\\
\leq\ &\int_{\R^3_+}\psi|u_0|^2+\int_0^t\int_{\R^3_+}\Delta\psi |u|^2+\int_0^t\int_{\R^3}\nabla\psi\cdot u|u|^2+2\int_0^t\int_{\R^3_+}\nabla\psi\cdot up\\
=:\ &I_1+I_2+I_3+I_4.
\end{split}
\end{align}
The aim is on the one hand to take advantage of the fact that one gains $R^{-1}$ when one derivative falls on $\vartheta_R$, and on the other hand to combine $\vartheta_R$ and $u$. When a $R^{-1}$ (or better) has been gained, one has won enough decay in $R$, so that we can simply use the global a priori estimate \eqref{e.apriorilenassu} on $u$ in the local energy norm. For the other terms where no $R^{-1}$ has been gained, it is of course important to estimate them in terms of $\vartheta_R$ in order to be able to apply a Gronwall-type lemma. 

\smallskip 

\noindent\emph{Step 1.} The three first terms in the right hand side of \eqref{e.locenineqR} can be handled identically to \cite{KS07}. For $I_1$, we simply have
\begin{equation}\label{e.I1}
|I_1|\leq C\|\vartheta_R u_0\|_{L^2_{uloc}(\R^3_+)}^2.
\end{equation}
For $I_2$, a direct computation yields
\begin{equation*}
\Delta\psi=2|\nabla\vartheta_R|^2\chi_{\!_{\, x_0,1}}+2\vartheta_R\Delta\vartheta_R\chi_{\!_{\, x_0,1}}+4\vartheta_R\nabla\vartheta_R\cdot\nabla\chi_{\!_{\, x_0,1}}+\vartheta_R^2\Delta\chi_{\!_{\, x_0,1}},
\end{equation*}
so that we gain at least $R^{-1}$ for every term but the last. It follows the rough bound
\begin{align}\label{e.I2}
|I_2(t)|\leq\ & C\int_0^t\int_{2\cu(x_0)}|\vartheta_R u|^2+CR^{-1}A_{T,\delta}\nonumber\\
\leq\ &C\int_0^t\sup_{\eta\in\Z^3_+}\int_{\cu(\eta)}|\vartheta_R u|^2+CR^{-1}A_{T,\delta}\nonumber\\
\leq\ &C\int_0^t\alpha_R(s)ds+CA_{T,\delta}R^{-1},
\end{align}
for all $t\in[0,T]$. For $I_3$, we simply get for all $\delta>0$, there exists a constant $C(\delta,T,u_0)<\infty$ such that for all $t\in [0,T]$,
\begin{align}\label{e.I3}
|I_3(t)|\leq\ & CA_{T,\delta}^\frac32R^{-1}+C\left(\int_0^t\int_{2\cu(x_0)}|\vartheta_Ru|^3\right)^\frac23\left(\int_0^t\int_{2\cu(x_0)}|u|^3\right)^\frac13\nonumber\\
\leq\ &CA_{T,\delta}^\frac32R^{-1}+CA_{T,\delta}^\frac12\gamma_R(t)\nonumber\\
\leq\ &CA_{T,\delta}^\frac32R^{-1}+\delta\beta_R(t)+C(\delta,T,u_0)A_{T,\delta}^\frac12\left(\int_0^t\alpha_R^{21}(s)ds\right)^\frac1{21}.
\end{align}
The estimate of $I_4$ is the heart of the matter. We decompose $I_4$ as follows:
\begin{equation*}
I_4=4\int_0^t\int_{\R^3_+}\chi_{\!_{\, x_0,1}}^\frac12p\nabla\vartheta_R\cdot \vartheta_R\chi_{\!_{\, x_0,1}}^\frac12u+\int_0^t\int_{\R^3_+} p\vartheta_R \nabla\chi_{\!_{\, x_0,1}}\cdot \vartheta_Ru=:I_{4,1}+I_{4,2}.
\end{equation*}
The first term in the right hand side, $I_{4,1}$ is easy to handle, because it has a $\nabla\vartheta_R$ which allows for the gain of $R^{-1}$. One can thus rely on the estimates derived in Section \ref{sec.pressureest} and on the global a priori bound \eqref{e.apriorilenassu} to obtain
\begin{equation}\label{e.I41}
|I_{4,1}(t)|\leq CA_{T,\delta}^\frac12\left(\|u_0\|_{L^2_{uloc}(\R^3_+)}+A_{T,\delta}\right)R^{-1}\leq CA_{T,\delta}^\frac32R^{-1},
\end{equation}
for all $t\in [0,T]$. 

\smallskip

The rest of this section is devoted to the estimate for $I_{4,2}$. For $I_{4,2}$, the main difficulty  is that we lack an estimate for $\vartheta_Rp$ in terms of quantities for $\vartheta_Ru$. It is not enough to just bound 
\begin{equation*}
|I_{4,2}|\leq C\left(\|u_0\|_{L^2_{uloc}(\R^3_+)}+A_{T,\delta}\right)\gamma_R^\frac12(t),
\end{equation*}
because that would lead to a nonlinear differential inequality of the type
\begin{equation*}
Z_R(t)\leq C\left(\int_0^tZ_R^\frac12(s)ds+\eta(R)\right),
\end{equation*}
with $\eta(R)\rightarrow 0$ when $R\rightarrow\infty$. 
Though we do have an estimate of $p$ in terms of $u$ (see Section \ref{sec.pressureest}), we have no information about the dependence in $R$. Therefore, we need to go back to the representation formula for $p$ given in \eqref{e.decomppressure} and estimate term by term. We have $\vartheta_R p=\vartheta_R p_{li}+\vartheta_R p_{loc}+\vartheta_R p_{nonloc}$, following the notations introduced in Section \ref{sec.pressureest}. Step 2 below is devoted the analysis of the linear pressure terms $\vartheta_R p_{li}$ related to the initial data, Step 3 to the local pressure $\vartheta_R p_{loc}$, while in Step 4 we handle the nonlocal pressure $\vartheta_R p_{nonloc}$. There are two recurrent ideas. The first one is to decompose the cut-off $\vartheta_R$ as follows,
\begin{equation}\label{e.deccutoff}
\vartheta_R(x',x_3)=\vartheta_R(x',x_3)-\vartheta_R(y',y_3)+\vartheta_R(y',y_3).
\end{equation} 
The second, is to use the following inequality
\begin{equation}\label{e.cutoffR-1}
\left|\vartheta_R(x',x_3)-\vartheta_R(y',y_3)\right|\leq C\min(R^{-1}|x-y|,1),
\end{equation}
for the difference in the right hand side of \eqref{e.deccutoff}.

\smallskip

\noindent\emph{Step 2: linear pressure terms.} For the linear pressure, we have from \eqref{e.decomppressure} $\vartheta_R p_{li}=\vartheta_R p_{loc}^{u_0}+\vartheta_R p_{nonloc}^{u_0}$. Thanks to the representation formula \eqref{e.forpressureu0loc}, for the first term we have\begin{align*}
&\vartheta_Rp_{loc}^{u_0}(x,t)
:=\int_\Gamma e^{\lambda t}\int_{\R^3_+}q_\lambda(x'-z',x_3,z_3)\cdot \vartheta_R(z',z_3)\chi_{\!_{\, 4}} u_0'(x',x_3)dz'dz_3d\lambda\\
=\ &\int_\Gamma e^{\lambda t}\int_{\R^3_+}q_\lambda(x'-z',x_3,z_3)\cdot \vartheta_R(z',z_3)\chi_{\!_{\, 4}} u_0'(z',z_3)dz'dz_3d\lambda\\
&+\int_\Gamma e^{\lambda t}\int_{\R^3_+}q_\lambda(x'-z',x_3,z_3)\cdot (\vartheta_R(x',x_3)-\vartheta_R(z',z_3))\chi_{\!_{\, 4}} u_0'(z',z_3)dz'dz_3d\lambda\\
=:\ &J_1+J_2.
\end{align*}
The estimates of Proposition \ref{prop.estlipressure} give for the first term in the right hand side above, for all $p\in[1,\frac{4}{3})$, there exists a constant $C(p)<\infty$ such that
\begin{equation*}
\left\|J_1\right\|_{L^p(0,T;L^2(2\cu(x_0)))}
\leq C\|\vartheta_Ru_0\|_{L^2(5\cu(x_0))}.
\end{equation*}
As for the second term in the right hand side above, using \eqref{e.cutoffR-1}, 
and the fact that $|x-z|\leq 6$ for all $x\in\cu(x_0)$ and $z\in\supp(\chi_{\!_{\, 4}})$, we obtain
\begin{align*}
&\left|\int_\Gamma e^{\lambda t}\int_{\R^3_+}q_\lambda(x'-z',x_3,z_3)\cdot (\vartheta_R(x',x_3)-\vartheta_R(z',z_3))\chi_{\!_{\, 4}} u_0'(z',z_3)dz'dz_3d\lambda\right|\\
\leq\ &CR^{-1}\int_\Gamma e^{\Rel(\lambda) t}\int_{\R^3_+}|q_\lambda(x'-z',x_3,z_3)||\chi_{\!_{\, 4}} u_0'(z',z_3)|dz'dz_3d\lambda.
\end{align*}
The estimates of Proposition \ref{prop.estlipressure} then give, for all $p\in[1,\frac{4}{3})$, there exists a constant $C(p)<\infty$ such that
\begin{equation*}
\left\|J_2\right\|_{L^p(0,T;L^2(2\cu(x_0)))}\leq CR^{-1}\|u_0\|_{L^2(5\cu(x_0))}.
\end{equation*}
Hence, for $p\in[1,\frac{4}{3})$, there exists a constant $C(p)<\infty$ such that
\begin{align}
&\left|\int_0^t\int_{\R^3_+} \vartheta_Rp_{loc}^{u_0} \nabla\chi_{\!_{\, x_0,1}}\cdot \vartheta_R u\right|\nonumber\\
\leq\ &C\|\vartheta_R p_{loc}^{u_0}\|_{L^p(0,T;L^2(2\cu(x_0)))}\|\vartheta_R u\|_{L^\infty(0,T;L^2(\cu(x_0)))}\nonumber\\
\leq\ & CA_{T,\delta}^\frac12\left(\|\vartheta_Ru_0\|_{L^2_{uloc}(\R^3_+)}+R^{-1}\|u_0\|_{L^2_{uloc}(\R^3_+)}\right).\label{e.estthetaRplocu0}
\end{align}
It remains to handle $p_{nonloc}^{u_0}$. Let $x\in\cu(x_0)$. From $\dist\left(\cu(x_0),\supp(1-\chi_{\!_{\, 4}})\right)\geq 1$, it follows that the singularity of the kernel in the representation formula \eqref{e.forharmpressure} is not seen. We have 
\begin{align*}
&\left|\int_{\R^3_+}q_{\lambda,x,x_0}(z',z_3)\cdot \vartheta_R(x)(1-\chi_{\!_{\, 4}})(z) u_0(z',z_3)dz'dz_3\right|\\
\leq\ &\int_{\R^3_+}\frac{e^{-|\lambda|^\frac12z_3}}{(1+|x-z|)^3}\left|\vartheta_R(x)-\vartheta_R(z)\right| |u_0(z',z_3)|dz'dz_3\\
&+\int_{\R^3_+}\frac{e^{-|\lambda|^\frac12z_3}}{(1+|x-z|)^3} \vartheta_R(z)|u_0(z',z_3)|dz'dz_3\\
=:\ &J_3+J_4.
\end{align*}
The second term is easily bounded as follows
\begin{equation*}
\|J_4(\cdot,t)\|_{L^\infty(\cu(x_0))}\leq C|\lambda|^{-\frac14}\|\vartheta_R u_0\|_{L^2_{uloc}(\R^3_+)}.
\end{equation*}
For the first term, we have
\begin{align*}
&\|J_3(\cdot,t)\|_{L^\infty(\cu(x_0))}\\
\leq\ &CR^{-1}\int_{x_0+(-R,R)^2\times(0,R)}\frac{e^{-|\lambda|^\frac12z_3}}{(1+|x_0-z|)^2}|u_0(z',z_3)|dz'dz_3\\
&+C\int_0^\infty\int_{\R^2\setminus (x_0'+(-R,R)^2)}\frac{e^{-|\lambda|^\frac12z_3}}{(1+|x_0-z|)^3}|u_0(z',z_3)|dz'dz_3\\
&+C\int_{\R_+\setminus (x_{0,3}-R,x_{0,3}+R)}\int_{\R^2}\frac{e^{-|\lambda|^\frac12z_3}}{(1+|x_{0,3}-z_3|+|x_0'-z'|)^3}|u_0(z',z_3)|dz'dz_3\\
\leq\ &C|\lambda|^{-\frac14}R^{-1}(\log R)\|u_0\|_{L^2_{uloc}(\R^3_+)}.\\
\end{align*}
Hence, Minkowski's inequality implies that
\begin{align*}
&\|\vartheta_Rp_{nonloc}^{u_0}(\cdot,t)\|_{L^\infty(\cu(x_0))}\\
\leq\ &\int_\Gamma e^{\Rel(\lambda)t}\left\|\int_{\R^3_+}q_{\lambda,x,x_0}(z',z_3)\cdot \vartheta_R(x)(1-\chi_{\!_{\, 4}})(z) u_0(z',z_3)dz'dz_3\right\|_{L^\infty(\cu(x_0))}|d\lambda|\\
\leq\ &C\int_\Gamma e^{\Rel(\lambda)t}|\lambda|^{-\frac14}|d\lambda|\left(\|\vartheta_R u_0\|_{L^2_{uloc}(\R^3_+)}+R^{-1}(\log R)\|u_0\|_{L^2_{uloc}(\R^3_+)}\right).
\end{align*}
It follows that
\begin{multline*}
\|\vartheta_Rp_{nonloc}^{u_0}(\cdot,t)\|_{L^\infty(\cu(x_0))}\leq Ct^{-\frac34}\left(\|\vartheta_R u_0\|_{L^2_{uloc}(\R^3_+)}+R^{-1}(\log R)\|u_0\|_{L^2_{uloc}(\R^3_+)}\right).
\end{multline*}
Therefore, for all $p\in[1,\frac43)$, there exists a constant $C(p)<\infty$ such that 
\begin{equation*}
\|\vartheta_Rp_{nonloc}^{u_0}\|_{L^p(0,T;L^\infty(\cu(x_0)))}
\leq C\left(\|\vartheta_R u_0\|_{L^2_{uloc}(\R^3_+)}+R^{-1}(\log R)\|u_0\|_{L^2_{uloc}(\R^3_+)}\right),
\end{equation*}
and 
\begin{multline}\label{e.estthetaRpharmu0}
\left|\int_0^t\int_{\R^3_+} \vartheta_Rp_{nonloc}^{u_0} \nabla\chi_{\!_{\, x_0,1}}\cdot \vartheta_R u\right|\\
\leq CA_{T,\delta}^\frac12T^\frac16\left(\|\vartheta_R u_0\|_{L^2_{uloc}(\R^3_+)}+R^{-1}(\log R)\|u_0\|_{L^2_{uloc}(\R^3_+)}\right),
\end{multline}
with a constant $C<\infty$. This concludes the study of the linear pressure terms. 

\smallskip

\noindent\emph{Step 3: local pressure terms.} We now turn to the term $\vartheta_R p_{loc,harm}^{u\otimes u}$. It requires some preliminary work. The idea is to write the system satisfied by $\vartheta_R u_{loc,harm}^{u\otimes u}$. It will satisfy \eqref{e.locnse} up to lower order terms, loss of incompressibility on $\supp(\nabla\vartheta_R)$ and zero initial data. We have
 \begin{equation}
\label{e.locnsecutoff}
\left\{ 
\begin{aligned}
& \partial_t\left(\vartheta_R u_{loc,harm}^{u\otimes u}\right)-\Delta\left(\vartheta_R u_{loc,harm}^{u\otimes u}\right)+\nabla\left(\vartheta_R p_{loc,harm}^{u\otimes u}\right)= F,\\
& \nabla\cdot \left(\vartheta_R u_{loc,harm}^{u\otimes u}\right)=G & \mbox{in} & \ (0,T)\times\R^3_+,\\
& \vartheta_R u_{loc,harm}^{u\otimes u} = 0  & \mbox{on} & \ (0,T)\times\partial \R^3_+,\\
& \vartheta_R u_{loc,harm}^{u\otimes u}(\cdot,0)= 0, & &
\end{aligned}
\right.
\end{equation}
where the source term is
\begin{align}\label{e.defF}
F:=\ &\nabla\vartheta_R p_{loc,harm}^{u\otimes u}-\Delta\vartheta_R u_{loc,harm}^{u\otimes u}-2\nabla\vartheta_R\cdot\nabla u_{loc,harm}^{u\otimes u}\\
&-\vartheta_R\mathbb P(\chi_{\!_{\, 4}}^2 u\cdot\nabla u)-\vartheta_R\mathbb P(\nabla(\chi_{\!_{\, 4}}^2) u\otimes u)\notag 
\end{align}
and 
\begin{equation}\label{e.defG}
G:=\nabla\vartheta_R\cdot u_{loc,harm}^{u\otimes u}.
\end{equation}
We perform one additional decomposition in order to deal separately with the right hand side and the lack of incompressibility. We have
\begin{equation*}
\left(\vartheta_R u_{loc}^{u\otimes u},\vartheta_R p_{loc}^{u\otimes u}\right)=\left(v_F,p_F\right)+\left(v_G,p_G\right),
\end{equation*}
where $\left(v_F,p_F\right)$ solves the Stokes system \eqref{e.locnsecutoff} with $G=0$ and $\left(v_G,p_G\right)$ solves the Stokes system \eqref{e.locnsecutoff} with $F=0$. The least regular term in the right hand side of \eqref{e.defF} is the fourth term, $\vartheta_R\mathbb P(\chi_{\!_{\, 4}}^2 u\cdot\nabla u)$. It belongs to $L^\frac32(0,T;L^\frac98(\R^3_+))$, which is the energy space. In this space $\mathbb P(\vartheta_R\chi_{\!_{\, 4}}^2 u\cdot\nabla u)$ is bounded via Gagliardo-Nirenberg's inequality as follows
\begin{align*}
&\|\vartheta_R\chi_{\!_{\, 4}}^2 u\cdot\nabla u\|_{L^\frac32(0,t;L^\frac98(\R^3_+))}\\
\leq\ &\left(\int_0^t\|\vartheta_Ru(\cdot,s)\|_{L^\frac{18}{7}(2\cu(x_0))}^\frac32\|\nabla u(\cdot,s)\|_{L^2(2\cu(x_0))}^\frac32ds\right)^\frac23\\
\leq\ &C\left(\int_0^t\|\vartheta_Ru(\cdot,s)\|_{L^2(2\cu(x_0))}^\frac32\|\nabla u(\cdot,s)\|_{L^2(2\cu(x_0))}^\frac32ds\right)^\frac23\\
& +\left(\int_0^t\|\vartheta_Ru(\cdot,s)\|_{L^2(2\cu(x_0))}\|\nabla(\vartheta_Ru)(\cdot,s)\|_{L^2(2\cu(x_0))}^\frac12\|\nabla u(\cdot,s)\|_{L^2(2\cu(x_0))}^\frac32ds\right)^\frac23\\
\leq\ &CA_{T,\delta}^\frac12\|\vartheta_Ru\|_{L^6(0,t;L^2(2\cu(x_0)))}\\
&+CA_{T,\delta}R^{-\frac13}+CA_{T,\delta}^\frac23\|\vartheta_Ru\|_{L^\infty(0,t;L^2(2\cu(x_0)))}^\frac23.
\end{align*}
The $L^\infty(0,t;L^2(2\cu(x_0)))$ norm of $\vartheta_R u$ is however difficult to handle in view of the Gronwall estimate. Therefore, we estimate $\vartheta_R\mathbb P(\chi_{\!_{\, 4}}^2 u\cdot\nabla u)$, and actually $F$ in whole, in a subcritical energy space. We have room for that. Any space $L^\frac32(0,T;L^q(\R^3_+))$ with $1<q<\frac98$ would work. Notice that $q=\frac98$ is excluded because of the reason mentioned above (energy space). Moreover, $q=1$ is excluded because it is ruled out in the maximum regularity theorem for the Stokes system of Giga and Sohr \cite[Theorem 3.1]{GS91}, which we apply to estimate $p_F$. According to \cite[Theorem 3.1]{GS91}, we will get that $p_F\in L^\frac32(0,T;L^r(\R^3_+))$, with $1+\frac3r=\frac3q$, so that necessarily $r\in(\frac32,\frac95)$ and
\begin{equation*}
p_F\nabla\chi_{\!_{\, x_0,1}}\in L^\frac32(0,T\times\R^3_+)\subset L^\frac32(0,T;L^r(\R^3_+)).
\end{equation*}
This is clearly enough to bound the integral $I_{4,2}$. Therefore, we choose to estimate $F$ in $L^\frac32(0,T;L^\frac{10}9(\R^3_+))$ but this choice is somewhat arbitrary. In this case, $q=\frac{10}9$, which yields $r=\frac{30}{17}$. Let us now carry out the estimates for $p_F$ first, and then $p_G$. Since the Helmholtz-Leray projection is bounded on $L^\frac{10}{9}(\R^3_+)$, we have
\begin{equation*}
\left\|\mathbb P(\chi_{\!_{\, 4}}^2 u\cdot \nabla u)\right\|_{L^\frac32(0,T;L^\frac{10}9(\R^3_+))}+\left\|\mathbb P(\nabla(\chi_{\!_{\, 4}}^2)u\otimes u)\right\|_{L^\frac32(0,T;L^\frac{10}9(\R^3_+))}\leq A_{T,\delta}.
\end{equation*}
Theorem 3.1 in \cite{GS91} implies that
\begin{multline*}
\|u_{loc,harm}^{u\otimes u}\|_{L^\frac32(0,T;L^\frac92(\R^3_+))}+\|\nabla u_{loc,harm}^{u\otimes u}\|_{L^\frac32(0,T;L^\frac95(\R^3_+))}+\|p_{loc,harm}^{u\otimes u}\|_{L^\frac32(0,T;L^\frac95(\R^3_+))}\leq A_{T,\delta}.
\end{multline*} 
Subsequently, the first three terms in the right hand side of \eqref{e.defF} are bounded by
\begin{multline*}
\left\|\nabla\vartheta_R p_{loc,harm}^{u\otimes u}\right\|_{L^\frac32(0,T;L^\frac95(\R^3_+))}+\left\|\Delta\vartheta_R u_{loc,harm}^{u\otimes u}\right\|_{L^\frac32(0,T;L^\frac92(\R^3_+))}\\
+\left\|2\nabla\vartheta_R\cdot\nabla u_{loc,harm}^{u\otimes u}\right\|_{L^\frac32(0,T;L^\frac95(\R^3_+))}\leq CA_{T,\delta}R^{-1}.
\end{multline*}
It remains to estimate the last three terms. For these terms, we rely on Lemma \ref{lem.comm} in order to commute the cut-off $\vartheta_R$ and the Helmhotz-Leray projection $\mathbb P$. The commutator term itself, is actually more regular, so below, we focus always on the term where the cut-off falls on the function. We have
\begin{align*}
&\|\vartheta_R\mathbb P(\chi_{\!_{\, 4}}^2 u\cdot\nabla u)\|_{L^\frac32(0,t;L^\frac{10}9(\R^3_+))}\\
\leq\ &\|\mathbb P(\vartheta_R\chi_{\!_{\, 4}}^2 u\cdot\nabla u)\|_{L^\frac32(0,t;L^\frac{10}9(\R^3_+))}+\|[\vartheta_R,\mathbb P](\chi_{\!_{\, 4}}^2 u\cdot\nabla u)\|_{L^\frac32(0,t;L^\frac{30}{17}(\R^3_+))}.
\end{align*}
From the second term, we gain $R^{-1}$. By the boundedness of the Helmholtz-Leray projection and H\"older's inequality, we have
\begin{align*}
&\|\vartheta_R\chi_{\!_{\, 4}}^2 u\cdot\nabla u\|_{L^\frac32(0,t;L^\frac{10}9(\R^3_+))}\\
\leq\ &\left(\int_0^t\|\vartheta_Ru(\cdot,s)\|_{L^\frac{5}{2}(2\cu(x_0))}^6ds\right)^\frac16\left(\int_0^t\|\nabla u(\cdot,s)\|_{L^2(2\cu(x_0))}^2ds\right)^\frac12.
\end{align*}
We now estimate the first factor in the right hand side. We have
\begin{align*}
&\left(\int_0^t\|\vartheta_Ru(\cdot,s)\|_{L^\frac52(2\cu(x_0))}^6ds\right)^\frac16\\
\leq\ &C\|\vartheta_Ru\|_{L^{42}(0,t;L^2(2\cu(x_0)))}^\frac{7}{10}\|\vartheta_Ru\|_{L^{2}(0,t;H^1(2\cu(x_0)))}^\frac{3}{10}\\
\leq\ &C\left(\int_0^t\alpha_R^{21}(s)ds\right)^\frac1{60}\left(\int_0^t\alpha_R(s)ds+\beta_R(t)+R^{-2}A_{T,\delta}\right)^\frac{3}{20}
\end{align*}
For the fifth term in the right hand side of \eqref{e.defF} we have
\begin{align*}
&\|\vartheta_R\nabla(\chi_{\!_{\, 4}}^2) u\otimes u\|_{L^\frac32(0,t;L^\frac98(\R^3_+))}\leq C\|\vartheta_Ru\otimes u\|_{L^\frac32(\R^3_+\times (0,t))}\\
\leq\ & C\|\vartheta_Ru\|_{L^3(\R^3_+\times (0,t))}\|u\|_{L^3(\R^3_+\times (0,t))}\leq CA_{T,\delta}^\frac12\gamma_R^\frac12(t).
\end{align*}
Therefore, $p_F$ is bounded by 
\begin{align}\label{e.estpf}
&\|p_F\|_{L^\frac32((0,t);L^\frac{30}{17}(2\cu(x_0)))}\\
\leq\ &CA_{T,\delta}^\frac12\left(\int_0^t\alpha_R^{21}(s)ds\right)^\frac1{60}\left(\int_0^t\alpha_R(s)ds+\beta_R(t)+R^{-2}A_{T,\delta}\right)^\frac{3}{20}\nonumber\\
&+CA_{T,\delta}^\frac12\gamma_R^\frac12(t)+C(\log R)A_{T,\delta}R^{-1}+CA_{T,\delta}^\frac12\left(\int_0^t\alpha_R^\frac34(s)\right)^\frac23.\nonumber
\end{align}
where we have again applied \cite[Theorem 3.1]{GS91}. We now turn to the estimate for $p_G$. Mimicking an idea of \cite{FGS06}, we introduce the solution $(E,q)$ of the following stationary Stokes problem with non homogeneous divergence
\begin{equation}
\label{e.statstokesdiv}
\left\{ 
\begin{aligned}
& -\Delta E+\nabla q= 0&\mbox{in} & \ \R^3_+,\\
& \nabla \cdot E = \nabla\vartheta_R\cdot u_{loc,harm}^{u\otimes u},  &&\\
&E=0& \mbox{on} & \ \partial \R^3_+.
\end{aligned}
\right.
\end{equation}
Notice that $E$ depends on $t$. At time $t=0$, $E(0)=0$. Furthermore $\partial_tE$ is the solution of \eqref{e.statstokesdiv} with $\nabla\vartheta_R\cdot \partial_tu_{loc}^{u\otimes u}$ as the inhomogeneous source term in the equation on the divergence. It follows from \cite[Theorem IV.3.3]{Galdi_book} that there exists a constant $C<\infty$ such that
\begin{equation*}
\|\nabla E\|_{L^\frac32(0,T;L^\frac92(\R^3_+))}+\|q\|_{L^\frac32(0,T;L^\frac92(\R^3_+))}\leq C\|\nabla\vartheta_R\cdot u_{loc,harm}^{u\otimes u}\|_{L^\frac32(0,T;L^\frac92(\R^3_+))}.
\end{equation*}
Moreover,
\begin{equation*}
\|\nabla \partial_tE\|_{L^\frac32(0,T;L^\frac98(\R^3_+))}+\|\partial_tq\|_{L^\frac32(0,T;L^\frac98(\R^3_+))}\leq C\|\nabla\vartheta_R\cdot \partial_tu_{loc,harm}^{u\otimes u}\|_{L^\frac32(0,T;L^\frac98(\R^3_+))}.
\end{equation*}
Hence, since $\partial_tE=0$ on $\partial \R^3_+$, we have by the Poincar\'e-Sobolev inequality that
\begin{align}\label{e.partialtE}
\|\partial_tE\|_{L^\frac32(0,T;L^\frac95(\R^3_+))}\leq\ & C\|\nabla \partial_tE\|_{L^\frac32(0,T;L^\frac98(\R^3_+))}\\
\leq\ &C\|\nabla\vartheta_R\cdot \partial_tu_{loc,harm}^{u\otimes u}\|_{L^\frac32(0,T;L^\frac98(\R^3_+))}\nonumber\\
\leq\ &CR^{-1}A_{T,\delta}\nonumber.
\end{align}
We subsequently decompose the pair $(v_G,p_G)$ into 
\begin{equation*}
v_G=\tilde{v}_G+E,\qquad p_G=\tilde{p}_G+q,
\end{equation*}
where the pair $(\tilde{v}_G,\tilde{p}_G)$ solves of course
\begin{equation}
\label{e.tildevG}
\left\{ 
\begin{aligned}
& \partial_t\tilde{v}_G-\Delta\tilde{v}_G+\nabla\tilde{p}_G= -\partial_tE&\mbox{in} & \ (0,T)\times\R^3_+,\\
& \nabla\cdot \tilde{v}_G=0, & &\\
& \tilde{v}_G = 0  & \mbox{on} & \ (0,T)\times\partial \R^3_+,\\
& \tilde{v}_G(\cdot,0) = 0. & &
\end{aligned}
\right.
\end{equation}
The maximal regularity of Theorem 3.1 in \cite{GS91} together with estimate \eqref{e.partialtE} implies that, up to adding a constant to $\tilde{p}_G$,
\begin{equation*}
\|\tilde{p}_G\|_{L^\frac32(0,T;L^\frac92(\R^3_+))}\leq C\|\nabla\tilde{p}_G\|_{L^\frac32(0,T;L^\frac95(\R^3_+))}\leq CA_{T,\delta}R^{-1}.
\end{equation*}
To conclude, we have proved
\begin{align}\label{e.estpg}
&\|p_G\|_{L^\frac32(0,T;L^\frac92(\R^3_+))}\\
\leq\ & \|\tilde{p}_G\|_{L^\frac32(0,T;L^\frac92(\R^3_+))}+\|q\|_{L^\frac32(0,T;L^\frac92(\R^3_+))}\leq CA_{T,\delta}R^{-1}.\nonumber
\end{align}
The estimate for $\vartheta_Rp_{loc,H}^{u\otimes u}$ follows from the combination of the $L^3$ boundedness of singular integral operators and of the commutator lemma, Lemma \ref{lem.comm}. In the end, \eqref{e.estpf} and \eqref{e.estpg} imply the following estimate: there exists a constant $C<\infty$ such that
\begin{align*}
&\left|\int_0^t\int_{\R^3_+} \vartheta_Rp_{loc}^{u\otimes u} \nabla\chi_{\!_{\, x_0,1}}\cdot \vartheta_R u\right|\nonumber\\
\leq\ &C\|\vartheta_R u\|_{L^3(0,t;L^3(\cu(x_0)))}\|\vartheta_R p_{loc}^{u\otimes u}\|_{L^\frac32(0,t;L^\frac32(\cu(x_0)))}\nonumber\\
\leq\ &CA_{T,\delta}^\frac12\gamma_R(t)^\frac12\left(\left(\int_0^t\alpha_R^{21}(s)ds\right)^\frac1{60}\left(\int_0^t\alpha_R(s)ds+\beta_R(t)+R^{-2}A_{T,\delta}\right)^\frac{3}{20}\right.\nonumber\nonumber\\
&\left.+C\gamma_R^\frac12(t)+CA_{T,\delta}^\frac12R^{-1}(\log R)+C\left(\int_0^t\alpha_R^\frac34(s)\right)^\frac23\right).
\end{align*}
We will rewrite this last inequality in a simpler form, more useful for the Gronwall estimate. Using the inequalities \eqref{e.eqsimpl}, we obtain for $\nu>0$
\begin{align}\label{e.estthetaRplocuu}
&\left|\int_0^t\int_{\R^3_+} \vartheta_Rp_{loc}^{u\otimes u} \nabla\chi_{\!_{\, x_0,1}}\cdot \vartheta_R u\right|\nonumber\\
\leq\ &\nu\beta_R(t)+C(\delta,T,\nu)A_{T,\delta}^\frac12\left(\int_0^t\alpha_R^{21}(s)ds\right)^\frac1{21}+C(T)A_{T,\delta}^\frac32R^{-1}(\log R),
\end{align}
with a constant $C(\delta,T,\nu)<\infty$ and $C(T)<\infty$. This concludes the study of $\vartheta_Rp_{loc}$.

\smallskip

\noindent\emph{Step 4: nonlocal pressure terms.} According to \eqref{e.decomppressure}, we have $\vartheta_R p_{nonloc}=\vartheta_R p_{nonloc,H}^{u\otimes u}+\vartheta_R p_{harm}^{u\otimes u}$. For these terms we rely on the decomposition \eqref{e.deccutoff} and the inequality \eqref{e.cutoffR-1}. The Helmholtz pressure is easy to estimate. We have for almost every $t\in(0,T)$, 
\begin{align*}
\|\vartheta_Rp_{nonloc,H}^{u\otimes u}(\cdot,t)\|_{L^\infty(\cu(x_0))}\leq\ & CR^{-1}(\log R)A_{T,\delta}+CA_{T,\delta}^\frac12\|\vartheta_R u(\cdot,t)\|_{L^2_{uloc}(\R^3_+)}\\
\leq\ &CA_{T,\delta}R^{-1}(\log R)+CA_{T,\delta}^\frac12\alpha_R^\frac12(t).\nonumber
\end{align*}
Therefore,
\begin{align}\label{e.estvarthetapH}
\left|\int_0^t\int_{\R^3_+} \vartheta_Rp_{nonloc,H}^{u\otimes u} \nabla\chi_{\!_{\, x_0,1}}\cdot \vartheta_R u\right|\leq\ & CA_{T,\delta}^\frac32R^{-1}(\log R)T^\frac23\\
&+CA_{T,\delta}^\frac12\gamma_R^\frac12(t)\left(\int_0^t\alpha_R(s)^\frac34ds\right)^\frac23.\nonumber
\end{align}
For the terms $p_{harm,\leq 1}^{u\otimes u}$ below, we will need the estimate of $\nabla(\vartheta_R p_{nonloc,H}^{u\otimes u})$. According to \eqref{e.forphelm}
\begin{align*}
&\nabla(\vartheta_Rp_{nonloc,H}^{u\otimes u})(x,t)\\
:=\ &\int_{\R^3_+}(\nabla_z',\partial_{x_3})\nabla^2_zN(x'-z',x_3,z_3)(1-\chi_{\!_{\, 4}}^2(z',z_3))(\vartheta_R(x)-\vartheta_R(z))u\otimes u(z',z_3,t)dz'dz_3\\
&+\nabla\vartheta_R(x)\int_{\R^3_+}\nabla^2_zN_{x,x_0}(x'-z',x_3,z_3)(1-\chi_{\!_{\, 4}}^2(z',z_3))u\otimes u(z',z_3,t)dz'dz_3\\
&+\int_{\R^3_+}(\nabla_z',\partial_{x_3})\nabla^2_zN(x'-z',x_3,z_3)(1-\chi_{\!_{\, 4}}^2(z',z_3))\vartheta_R(z) u\otimes u(z',z_3,t)dz'dz_3\\
=:\ & K_1+K_2+K_3.
\end{align*}
where $N_{x,x_0}$ is defined by \eqref{e.defNxx_0}. 
We clearly have
\begin{align*}
\|K_1(\cdot,t)\|_{L^\infty(\cu(x_0))}\leq\ &CA_{T,\delta}(R^{-1}\log R),\\
\|K_2(\cdot,t)\|_{L^\infty(\cu(x_0))}\leq\ &CA_{T,\delta}R^{-1},\\
\|K_1(\cdot,t)\|_{L^\infty(\cu(x_0))}\leq\ &CA_{T,\delta}^\frac12\alpha_R^\frac12(t),
\end{align*}
so that 
\begin{equation}\label{e.estnablathetaRpH}
\|\nabla(\vartheta_Rp_{nonloc,H}^{u\otimes u})(\cdot,t)\|_{L^\infty(\cu(x_0))}\leq\ CA_{T,\delta}R^{-1}(\log R)+CA_{T,\delta}^\frac12\alpha_R^\frac12(t).
\end{equation}
For $\vartheta_Rp_{harm}^{u\otimes u}$, we decompose again into $\vartheta_Rp_{harm}^{u\otimes u}=\vartheta_Rp_{harm,\leq 1}^{u\otimes u}+\vartheta_Rp_{harm,\geq 1}^{u\otimes u}$ and analyze the two terms separately. Let us start with the analysis of $\vartheta_Rp_{harm,\leq 1}^{u\otimes u}$, which is parallel to the proof of \eqref{e.locestpharmuu}. We have
\begin{align*}
&\left|\int_0^t\int_\Gamma e^{\lambda (t-s)}\int_{\R^3_+}q_{\lambda}(x'-z',x_3,z_3)\cdot \chi_2^2 \left(\mathbb P\nabla\cdot ((1-\chi_{\!_{\, 4}}^2)u\otimes u)\right)'(z',z_3,s)dz'dz_3d\lambda ds\right|\\
\leq\ &\int_0^t\int_\Gamma e^{\Rel(\lambda)(t-s)}\int_{\R^3_+}\frac{e^{-c|\lambda|^\frac12z_3}\left|\vartheta_R(x)-\vartheta_R(z)\right|}{(1+|x-z|)^{2}}\left| \chi_2^2 \left(\mathbb P\nabla\cdot ((1-\chi_{\!_{\, 4}}^2)u\otimes u)\right)'\right|dz'dz_3|d\lambda| ds\\
&+\int_0^t\int_\Gamma e^{\Rel(\lambda)(t-s)}\int_{\R^3_+}\frac{e^{-c|\lambda|^\frac12z_3}}{(1+|x-z|)^{2}}\left| \vartheta_R(z)\chi_2^2 \left(\mathbb P\nabla\cdot ((1-\chi_{\!_{\, 4}}^2)u\otimes u)\right)'\right|dz'dz_3|d\lambda| ds\\
=:\ &K_4+K_5.
\end{align*}
For both terms, the relation \eqref{e.relchi2ppH} is the basis of our estimates. For $K_4$ we rely on \eqref{proof.fromYM.2} for the estimate of $\left(\mathbb P\nabla\cdot ((1-\chi_{\!_{\, 4}}^2)u\otimes u)\right)'$ and on \eqref{e.cutoffR-1} to bound $\vartheta_R(x)-\vartheta_R(z)$. Using the fact that $\chi_2$ is compactly supported, we obtain for all $\sigma\in(0,1)$
\begin{align*}
|K_4|\leq\ & C \int_0^t \int_{\R^3_+} \int_\Gamma e^{\Re (\lambda) (t-s) -c |\lambda|^\frac12 z_3} |d\lambda| \frac{1}{|x'-z'|} |\vartheta_R(x)-\vartheta_R(z)||\chi_2^2 \nabla p_{nonloc,H}^{u\otimes u} | d z d s\\
\leq\ & CR^{-1}\int_0^t (t-s)^{-1+\frac{\sigma}{2}} \int_{\R^3_+} \frac{1}{z_3^\sigma |x'-z'|}|\chi_2^2 \nabla p_{nonloc,H}^{u\otimes u} | d z d s\\
\leq\ & CR^{-1}\int_0^t (t-s)^{-1+\frac{\sigma}{2}} \int_{\max(x_3-3,0)}^{x_3+3} \frac{1}{z_3^\sigma} d z_3 d s \| u\|_{L^\infty (0,t; L^2_{uloc} (\R^3_+))}^2\\
\leq\ & CR^{-1} t^\frac{\sigma}{2}A_{T,\delta}.
\end{align*}
Concerning $K_5$, we have 
\begin{equation*}
\vartheta_R(z)\chi_2^2 \left(\mathbb P\nabla\cdot ((1-\chi_{\!_{\, 4}}^2)u\otimes u)\right)'=\vartheta_R\chi_2^2\nabla p_{nonloc,H}^{u\otimes u}=\chi_2^2\nabla(\vartheta_R p_{nonloc,H}^{u\otimes u})-\chi_2^2(\nabla\vartheta_R) p_{nonloc,H}^{u\otimes u}.
\end{equation*}
Hence, using the estimate \eqref{e.estnablathetaRpH} for $\nabla(\vartheta_R p_{nonloc,H}^{u\otimes u})$ and the estimate \eqref{e.boundnlochelmpressure} for $p_{nonloc,H}^{u\otimes u}$, we eventually get for all $\sigma\in(0,1)$,
\begin{align*}
|K_5|\leq\ & Ct^\sigma\left(A_{T,\delta}R^{-1}(\log R)+A_{T,\delta}^\frac12\alpha_R^\frac12(t)\right).
\end{align*}
This ends the estimate for $p_{harm,\leq 1}^{u\otimes u}$: there exists $C(T,\sigma)<\infty$ and $C<\infty$ such that
\begin{align}\label{e.estthetaRpharmleq1}
&\left|\int_0^t\int_{\R^3_+} \vartheta_Rp_{harm,\leq 1}^{u\otimes u} \nabla\chi_{\!_{\, x_0,1}}\cdot \vartheta_R u\right|\nonumber\\
\leq\ &C(T,\sigma)A_{T,\delta}^\frac32R^{-1}(\log R)+CA_{T,\delta}^\frac12\int_0^ts^\sigma\alpha_R(s)ds.
\end{align}
We now turn to the term $p_{harm,\geq 1}^{u\otimes u}$. We analyze separately the terms \eqref{e.termA'} and \eqref{e.termB'} according to the decomposition given in Appendix \ref{sec.lerayproj}. The analysis is in the same vein as the one carried out in the proof of \eqref{e.locestpharmuu} in Section \ref{sec.nlpressure}. Here as above additional difficulties comes from handling the dependence in $R$. We start with the terms \eqref{e.termA'}. The argument is in fact similar to the proof of \eqref{e.termA} in \eqref{e.nonlocestpharmuu}. Integrating by parts, we are led to studying
\begin{align*}
&\bigg|\int_0^t\int_\Gamma e^{\lambda (t-s)}\int_{\R^3_+}\partial_\alpha q_{\lambda,x,x_0}(z',z_3)\cdot \vartheta_R(x)\big((1-\chi_{\!_{\, 4}}^2)v\otimes w)\big)'(z',z_3,s)dz'dz_3d\lambda ds\bigg|\\
\leq\ & C\int_0^t\int_\Gamma |\lambda|^\frac12 e^{\Rel(\lambda) (t-s)}\int_{\R^3_+}\frac{e^{-c|\lambda|^\frac12z_3}\left|\vartheta_R(x)-\vartheta_R(z)\right|}{(1+|x-z|)^{3}}|v\otimes w(z',z_3,s)|dz'dz_3|d\lambda| ds\\
&+C\int_0^t\int_\Gamma |\lambda|^\frac12 e^{\Rel(\lambda) (t-s)}\int_{\R^3_+}\frac{e^{-c|\lambda|^\frac12z_3}}{(1+|x-z|)^{3}}|\vartheta_Rv\otimes w(z',z_3,s)|dz'dz_3|d\lambda| ds\\
=:\ &K_6+K_7.
\end{align*} 
As for $K_7$, the argument is exactly the same as the proof of \eqref{e.termA} in \eqref{e.nonlocestpharmuu}, and we have 
\begin{align*}
\| K_7 \|_{L^2 (0,t; L^\infty(\cu(x_0)))} \leq C A_{T,\delta}^\frac12 \Big ( \int_0^t \alpha_R (s)^\frac12 \, d s + \beta_R (t)^\frac12 \Big ).
\end{align*}
The estimate of $K_6$ is obtained also by a simple modification of the above proof. 
Indeed, in this case it suffices to use the bound 
\begin{align*}
& \int_{\R^3_+}\frac{e^{-c|\lambda|^\frac12z_3}\left|\vartheta_R(x)-\vartheta_R(z)\right|}{(1+|x-z|)^{3}}|v\otimes w|dz'dz_3\\
& \leq \frac{C}{R} \sum_{\eta'\in \Z^2,|\eta'|\leq R} \frac{1}{1+|\eta'|^2} (\int_0^\frac12 + \int_\frac12^\infty ) \int_{\cu'(\eta')} | v\otimes w |dz'dz_3\\
& + C \sum_{\eta'\in \Z^2, |\eta'|\geq R} \frac{1}{1+|\eta'|^3} (\int_0^\frac12 + \int_\frac12^\infty ) \int_{\cu'(\eta')} | v\otimes w |dz'dz_3.
\end{align*}
Then the proof of \eqref{e.termA} in \eqref{e.nonlocestpharmuu} yields 
\begin{equation*}
\|K_6(\cdot,t)\|_{L^2(0,t; L^\infty(\cu(x_0)))}\leq CA_{T,\delta}R^{-1} (\log R).
\end{equation*}
The details are omitted here. The terms of \eqref{e.termB'} are proved in the same manner by applying the argument of \eqref{e.termB} in the proof of \eqref{e.nonlocestpharmuu} and in fact easier, due to the absence of the factor $|\lambda|^\frac12$ in the resolvent kernel in this case. In particular, the estimates as in $K_6+K_7$ are valid also for the terms  \eqref{e.termB'}. Therefore, we have for all $\nu>0$,
\begin{align}\label{e.estthetaRpharmgeq1}
&\left|\int_0^t\int_{\R^3_+} \vartheta_Rp_{harm,\geq1}^{u\otimes u} \nabla\chi_{\!_{\, x_0,1}}\cdot \vartheta_R u\right|\nonumber\\
\leq\ &\nu\beta_R(t)+C(\delta,T,\nu)A_{T,\delta}^\frac12\left(\int_0^t\alpha_R(s)ds\right)+C(T)A_{T,\delta}^\frac32R^{-1}(\log R),
\end{align}
with a constant $C(\delta,T,\nu)<\infty$ and $C(T)<\infty$. This concludes the study of $p_{harm,\geq1}^{u\otimes u}$. Hence, combining \eqref{e.eqsimpl} with \eqref{e.estvarthetapH}, \eqref{e.estthetaRpharmleq1} and \eqref{e.estthetaRpharmgeq1} we eventually get for all $\nu>0$, there exists a constant $C(\delta,T,\nu)<\infty$ and $C(T)<\infty$ such that
\begin{align}\label{e.estthetaRpnonloc}
\begin{split}
&\left|\int_0^t\int_{\cu(x_0)}\vartheta_Rp_{nonloc}\nabla\chi_{\!_{\, x_0,1}}\cdot \vartheta_R u\right|\\
\leq\ &\nu\beta_R(t)
+C(\delta,T,\nu)A_{T,\delta}^\frac12\left(\int_0^t\alpha_R^{21}(s)ds\right)^\frac1{21}+C(T)A_{T,\delta}^\frac32R^{-1}(\log R).
\end{split}
\end{align}
This concludes the estimates for the pressure and Step 4. 

\smallskip

Taking $\nu=\frac13$, we have that estimates \eqref{e.I1}, \eqref{e.I2}, \eqref{e.I3}, \eqref{e.I41}, \eqref{e.estthetaRplocu0}, \eqref{e.estthetaRpharmu0}, \eqref{e.estthetaRplocuu} and \eqref{e.estthetaRpnonloc} immediately imply the differential inequality of Lemma~\ref{lem.aprioridecay}.

\section{Local existence of local energy weak solutions}
\label{sec.locexiles}

This section is devoted to the proof of the local existence of the local energy weak solution.

\begin{proposition}\label{prop.local}
For all $u_0 \in L^2_{uloc,\sigma}(\R^3_+)$, there exist $T_0>0$ depending only on $\|u_0\|_{L^2_{uloc}}$
and a local energy weak solution of \eqref{e.nse} in $(0,T_0)$. 
\end{proposition}

\subsection{Regularized problem}

For $\ep
>0$, we first study  the regularized problem for 
$(v,q)=(v^\ep, q^\ep)$:
\begin{equation}
\label{e.rnse}
\left\{ 
\begin{aligned}
& \partial_t v+F_\ep(v) \cdot\nabla v-\Delta v+\nabla q = 0,  \quad \nabla\cdot v=0 & \mbox{in} & \ (0,T)\times\R^3_+, \\
& v = 0  & \mbox{on} & \ (0,T)\times\partial \R^3_+,\\
&v|_{t=0}=v_0:=F_\ep(u_0)
& \mbox{in}  & \ \R^3_+.
\end{aligned}
\right.
\end{equation} 
Here $F_\ep$ is a mollification operator for the vector fields in $\R^3_+$ defined as
$F_{\ep}(u)= (\omega_\ep * \tilde{u})|_{\R^3_+}$, where $\omega_\ep$ is 
a standard radial symmetric mollifier supported in the ball $B(0,\ep)$ and 
$\tilde{u}$ is the extension of $u$ given by 
\begin{align*}
\tilde{u}(x) & = u(x)~~~{\rm if}~~x_3>0\,,\\
\tilde{u}'(x', x_3) & = u'(x', -x_3)~~{\rm and}~~
\tilde{u}_3 (x',x_3) = - u_n (x', -x_3)~~~{\rm if}~~x_3<0\,.
\end{align*}
Then 
$F_\ep(u)$ satisfies ${\rm div}\, F_\ep(u) =0$ in $\R^n$ 
and $F_{\ep}(u_3)_n=0$ for $x_3=0$ by the symmetry.
The following local well-posedness result can be shown by 
the contraction principle.

\begin{proposition}
\label{prop.localmild}
Let $u_0 \in L^2_{uloc,\sigma}(\R^3_+)$. 
For $\ep>0$ there exist $T_*=T_*(\ep)>0$ and a unique mild solution to the 
problem \eqref{e.rnse} in $C([0,T_*);L^2_{uloc,\sigma}(\R^3_+))
\cap C((0,T_*);L^\infty(\R^3_+))$. Moreover if $T_*$ is 
the maximal existence time for the mild solution, $v$ satisfies 
$\lim_{t \uparrow T_*} \|v(t)\|_{L^2_{uloc}}=\infty$.
\end{proposition}

\begin{proof} The proof is  
based on the standard Banach fixed point theorem as in \cite{MMP17}.
Set $\|u\|_T$ as
\begin{align*}
\| u\|_T = \sup_{0<t<T} \big ( \| u (t) \|_{L^2_{uloc}} + t^\frac34 \| u(t) \|_{L^\infty} ).
\end{align*}
Let $C_0>0$ be a constant such that 
\begin{align*}
\| e^{-\cdot {\bf A}} v_0 \|_T \leq C_0 ( 1+ T^\frac{3}{4}) \| v_0 
\|_{L^2_{uloc}}\,, \qquad f\in L^2_{uloc,\sigma}(\R^3_+),
\end{align*}
which is well-defined by virtue of the Stokes estimate \cite[Proposition 5.2]{MMP17}.
Then let us introduce the set 
\begin{align*}
\begin{split}
X_T & =\Big \{  u\in L^\infty (0,T; L^2_{uloc,\sigma} (\R^3_+)) \cap C(0,T; 
L^{\infty}_\sigma (\R^3_+)) ~|~ \\
& \qquad \| u \|_T \leq  2 C_0 (1+T^\frac34) \| u_0 \|_{L^2_{uloc}} \Big \}.
\end{split}
\end{align*}
For each $f\in X_T$ we define the map $\Phi_\ep [f] (t) =e^{-t{\bf A}} v_0 + B_\ep[f,f] (t)$, where
\begin{align*}
B_\ep[f,g] (t) = - \int_0^t e^{-(t-s){\bf A}} \mathbb{P} \nabla \cdot (F_\ep(f) \otimes g ) d s,\qquad t>0, ~f,g\in X_T.
\end{align*}
We will show that if $T_\ep$ is sufficiently small,
 then $\Phi_\ep$ defines a contraction map in $X_T$. 
 Indeed by using the estimate for the Stokes semigroup \cite[Theorem 3]{MMP17} and the elementary inequality
\begin{equation*}
\|F_\ep(f)\|_{L^2_{uloc}(\R^3_+)} +(1+\ep^{-3})^{-1}\|F_\ep(f)\|_{L^\infty(\R^3_+)} \le \|f\|_{L^2_{uloc}(\R^3_+)},
\end{equation*}
we have 
\begin{align}
\label{proof.thm.mild.ns.5}
\| B_\ep[f,g] (t) \|_{L^2_{uloc}} & \leq C \int_0^t (t-s)^{-\frac12} 
\| F_\ep(f)\|_{L^\infty} \|g \|_{L^2_{uloc}} d s  
\\
& \leq C \int_0^t (t-s)^{-\frac12} 
(1+\ep^{-3})\| f\|_{L^2_{uloc}} \|g \|_{L^2_{uloc}} d s \notag
\\
& \leq C (1+\ep^{-3})T^{\frac12}\sup_{0<t<T}\| f\|_{L^2_{uloc}} \sup_{0<<tT} 
\|g \|_{L^2_{uloc}}. 
\notag
\end{align}
Similarly, we have for $f,g\in X_T$, 
\begin{align*}
\| B_\ep[f,g] (t) \|_{L^\infty} & \leq C \int_0^t (t-s)^{-\frac12} 
\| F_\ep(f)\|_{L^\infty} \|g \|_{L^\infty} d s \nonumber \\
& \leq C \int_0^t (t-s)^{-\frac12} 
(1+\ep^{-3})\| f\|_{L^2_{uloc}} \|g \|_{L^\infty} d s \nonumber \\
& \leq C (1+\ep^{-3})t^{-\frac14}\sup_{0<t<T}\| f\|_{L^2_{uloc}} 
\sup_{0<t<T} t^{\frac34}\|g \|_{L^\infty}.
\end{align*}
Thus we obtain 
\begin{align}
\label{Bep}
\| B_\ep[f,g] \|_T \leq C_1 (1+\ep^{-3})T^\frac12 \| f\|_T 
\| g\|_T, \qquad f, g \in X_T.
\end{align}
If $T$ is small so that 
\begin{align}
\label{proof.thm.mild.ns.9}
C_1 (1+\ep^{-3}) T^\frac12 2 C_0 (1 + T^\frac34)
\|u_0\|_{L^2_{uloc}} \leq \frac14,
\end{align}
then \eqref{proof.thm.mild.ns.9} and the definition of $C_0$ imply that $\Phi_\ep$ defines a contraction map from $X_T$ into $X_T$. Hence, there exists a unique fixed point 
$u$ of $\Phi_\ep$ in $X_T$, which is the mild solution to \eqref{e.rnse} in $X_T$.
Since $v_0 \in BUC_\sigma (\R^d_+)$, $e^{-t\mathbf A}v_0$ is continuous at $t=0$, and hence
the continuity in time of $u$ also follows from the standard argument.
\eqref{Bep} also shows the uniqueness of the solution in the class $\|v\|_T <\infty$.
This guarantees the existence of the maximal interval $[0,T_*)$ 
where $v$ satisfies $\|v\|_T <\infty$ for any $T<T_*$.
If $T_*$ is finite, from \eqref{proof.thm.mild.ns.9} we have $\|v(t)\|_{L^2_{uloc}} \ge 
\frac{1}{8C_1 (1+\ep^{-3}) (T_*-t)^\frac12 C_0 (1 + (T_*-t)^\frac34)}$, which, 
in particular, implies $\lim_{t \uparrow T_*} \|v(t)\|_{L^2_{uloc}}=\infty$.
\end{proof}

Due to the argument in \cite[Lemma 4.1]{GHM}, 
the mild solution is smooth in $\R^3_+ \times (0,T_*)$. 
For each $x_0 \in \R^3_+$ we can use the cut-off argument in the proof of Proposition \ref{prop.property.weaksol} 
to define the pressure  $q=q^\ep_{(x_0)}$ so that $(v,q)$ defines a classical solution to \eqref{e.rnse}
even when $v$ does not decay at spatial infinity.  
Moreover by the similar argument as in Propositions \ref{prop.estlocpressure} and \ref{prop.nlpressure}, 
the pressure $q$ can be decomposed as $q=q_{li}+q_{loc}+q_{nonloc}=q_{loc}^{v_0}+q_{nonloc}^{v_0}+q_{loc,H}^{F_\ep(u)\otimes u}+q_{loc,harm}^{F_\ep(u)\otimes u}+q_{nonloc,H}^{F_\ep(u)\otimes u}+
q_{harm,\leq 1}^{F_\ep(v)\otimes v}+q_{harm,\geq 1}^{F_\ep(v)\otimes v}
$ 
such that the following estimates hold.

\begin{proposition}[Linear pressure estimates for the regularized problem]\label{prop.estlipressurer}
Let $T>0$. Then there exists a constant $C(T)<\infty$ such that for all $t\in(0,T)$,
\begin{align*}
\frac t{\log(e+t)} \| \nabla q_{li} (t) \|_{L^2_{uloc}(\R^3_+)} & \leq C \| u_0 \|_{L^2_{uloc} (\R^3_+)},\\
t^{\frac34} \|q_{loc}^{v_0} (t)\|_{L^2(\cu(x_0))} & \leq C\|u_0\|_{L^2(5\cu(x_0))},\\
t^\frac34 \|q_{nonloc}^{v_0} (t) \|_{L^\infty(\cu(x_0)))} + t^\frac34 \| \nabla q_{nonloc}^{v_0} (t) \|_{L^\infty(\cu(x_0)))} & \leq C\|u_0\|_{L^2_{uloc}(\R^{3}_+)}.
\end{align*}
\end{proposition}

\begin{proposition}[Local pressure estimates to the regularized problem]
\label{prop.estlocpressurer}
Let $T>0$. There exists a constant $C(T)<\infty$ such that for all $x_0\in\R^3_+$,
\begin{align*}
&\left\|q^{F_\ep(v)\otimes v}_{loc,H}\right\|_{L^\frac32(0,T;L^\frac32(\cu(x_0)))}\!\!+\left\|q^{F_\ep(v)\otimes v}_{loc,harm}\right\|_{L^\frac32(0,T;L^\frac32(\cu(x_0)))}\!\! + \left\| \nabla q^{F_\ep(v)\otimes v}_{loc,harm}\right\|_{L^\frac32(0,T;L^\frac98(\R^3_+))}\\
\leq\  &
C\sup_{\eta \in \Z^3_+}\left(\left\|v\right\|_{L^\infty(0,T;L^2(\cu(\eta)))}^2
+\left\|\nabla v\right\|_{L^2(0,T;L^2(\cu(\eta)))}^2\right).
\end{align*}
\end{proposition}

\begin{proposition}[Nonlocal pressure estimates to the regularized problem]
\label{prop.nlpressurer}
Let $T>0$ and $1\leq q<\infty$. There exist constants $C(T),\ C(T,q)<\infty$ such that 
for all $x_0\in\R^3_+$ and for almost all $t\in (0,T)$, 
\begin{align*}
\|q_{nonloc,H}^{F_\ep(v)\otimes (v)}(\cdot,t)\|_{L^\infty(\cu(x_0))} + \|\nabla q_{nonloc,H}^{F_\ep(v)\otimes (v)}(\cdot,t)\|_{L^\infty(\cu(x_0))} \leq\ & C\|v(\cdot,t)\|_{L^2_{uloc}(\R^3_+)}^2,\\\|q_{harm,\leq 1}^{F_\ep(v)\otimes (v)} (\cdot,t)\|_{L^\infty(\cu(x_0))} + \|\nabla q_{harm,\leq 1}^{F_\ep(v)\otimes (v)} (\cdot,t)\|_{L^q(\cu(x_0))}\leq\ & C_q \|v\|_{L^\infty (0,t; L^2_{uloc}(\R^3_+))}^2, 
\end{align*}
and 
\begin{align*}
& \|q_{harm,\geq 1}^{F_\ep(v)\otimes (v)} \|_{L^2 (0,T; L^\infty (\cu(x_0)))} + \|\nabla q_{harm,\geq 1}^{F_\ep(v)\otimes (v)} \|_{L^2 (0,T; L^\infty(\cu(x_0)))} \\
\leq\ & C \big ( \|v\|_{L^\infty(0,T;L^2_{uloc}(\R^3_+))}^2 + \sup_{\eta\in \Z^3_+} \| \nabla v \|_{L^2(0,T; L^2 (\cu(\eta)))}^2 \big ).
\end{align*}
\end{proposition}
For later use, we summarize these pressure
estimates as the following corollary:

\begin{corollary}
\label{prop.pressurer}
For $T>0$ let $v=v^{\ep}$ be the mild solution to the problem \eqref{e.rnse}
in $(0,T) \times \R^3_+$. There exist $C(T)<\infty$ and $C_\delta=C(T,\delta)$ 
such that for any $\ep>0$ and  $x_0\in\R^3_+$, there exists a pressure $q=q^{\ep}_{(x_0)}$ 
satisfying 
\begin{align}\label{e.qg}
\|\nabla q \|_{L^\frac32(\delta,T;L^{\frac98}(\cu(x_0)))}
\leq \ &C_\delta\|u_0\|_{L^2_{uloc}(\R^3_+)}
\\
\notag
&+ C\left\|v\right\|_{L^\infty(0,T;L^2_{uloc}(\R^3_+))}^2
+C\sup_{\eta \in \Z^3_+}
\left\|\nabla v\right\|_{L^2(0,T;L^2(\cu(\eta)))}^2.
\end{align}
Moreover there exist $q_1$, $q_2$ such that
$$
q=q_1+q_2 \qquad {\textit for} \ (0,T) \times \cu(x_0),
$$
\begin{multline}\label{e.q}
\left\|q_1\right\|_{L^\frac32(0,T;L^\frac32(\cu(x_0)))}+\|q_2\|_{L^{\frac54}(0,T;L^2(\cu(x_0)))}\\
\leq\ C\|u_0\|_{L^2_{uloc}(\R^3_+)}+ C\left\|v\right\|_{L^\infty(0,T;L^2_{uloc}(\R^3_+))}^2
+C\sup_{\eta \in \Z^3_+}
\left\|\nabla v\right\|_{L^2(0,T;L^2(\cu(\eta)))}^2.
\end{multline}
\end{corollary}

We now claim a key local energy estimate which guarantees the 
uniformity in $\ep>0$ of the existence time of the solution obtained in Proposition \ref{prop.localmild}.

\begin{proposition}
\label{prop.vunif}
There exist constants $M>0$ and
 $T_0=T_0(\|u_0\|_{L^2_{uloc}})>0$  independent of $\ep>0$, 
 such that 
\begin{equation}\label{e.apriorilen}
E(T_0):=\sup_{\eta\in\Z^3_+,\, t\in (0,T_0)}\int_{\cu(\eta)}|v(\cdot,t)|^2
+\sup_{\eta\in\Z^3_+} \int_0^{T_0}\int_{\cu(\eta)}|\nabla v|^2
\leq M\|u_0\|_{L^2_{uloc}}^2.
\end{equation}
\end{proposition}
\begin{proof}

We let
\begin{equation}
\alpha(t)=
\sup_{\eta\in\Z^3_+,\, s\in (0,t)}\int_{\cu(\eta)}|v(\cdot,s)|^2, 
\qquad 
\beta(t)=
\sup_{\eta\in\Z^3_+}
\int_0^{t}\int_{\cu(\eta)}|\nabla v|^2.
\end{equation}
Testing the equation \eqref{e.rnse} against the function $\chi_{x_0}^2v$, 
we have the equality
\begin{align}
\begin{split}
&\int_{\R^3_+}\chi_{x_0}^2 |v(t)|^2dx
+
2\int^t_0 \int_{\R^3} \chi^2_{x_0} |\nabla v|^2dxds
\label{e.localenergy}
\\
=\ &
\|\chi_{x_0}F_\ep(u_0)\|^2_{L^2_{uloc}}+2\int^t_0 \int_{\R^3_+} \Delta \chi_{x_0}^2 
|v|^2 dxds 
+ \int^t_0 \int_{\R^3_+} \nabla \chi_{x_0}^2 \cdot F_{\ep}(v)|v|^2 dxds\\
&+ 2 \int^t_0 \int_{\R^3_+} q \nabla \chi_{x_0}^2 \cdot v dxds.
\end{split}
\end{align}
We estimate each term in the right hand side. 
For the second term, we easily see 
\begin{align}
\label{e.02}
\int^t_0 \int_{\R^3_+} \Delta \chi_{x_0}^2 
|v|^2 dxds 
\le
C T \alpha(t).
\end{align}
By the Gagliardo-Nirenberg inequality, we have
\begin{align}
\label{e.gn}
\begin{split}
\|\chi_{x_0} v\|_{L^4(\R^3_+)}
\le\ &
C\|\chi_{x_0} v\|_{L^2(\R^3_+)}^{\frac14}\|\nabla(\chi_{x_0} v)\|_{L^2(\R^3_+)}^{\frac34}\\
\le\ &
C\|\chi_{x_0} v\|_{L^2(\R^3_+)}^{\frac14}
(\|\chi_{x_0} v\|_{L^2(\R^3_+)}+\|\nabla(\chi_{x_0} v)\|_{L^2(\R^3_+)})^{\frac34}.
\end{split}
\end{align}
Therefore the third term is estimated as follows:
\begin{align}
\label{e.03}
\begin{split}
&\int^t_0 \int_{\R^3_+} \nabla \chi_{x_0} \cdot F_{\ep}(v)|v|^2 dxds
\\
\le\ &
C \int^t_0 \|F_{\ep}(v)\|_{L^2_{uloc}(\R^3_+)}  \|\chi_{x_0} v\|_{L^4(\R^3_+)}^2 ds
\\
\le\ & 
C\alpha (t)^{\frac12} \int^t_0  \|\chi_{x_0} v\|_{L^4(\R^3_+)}^2 ds
\\
\le\ & 
C\alpha(t)^{\frac12} (T\alpha(t) +T^{\frac14}\alpha(t)^{\frac14}\beta(t)^{\frac34})
\\
\le\ & 
CT\alpha^{\frac32}(t)+\delta \beta(t) +C_{\delta} T\alpha^3(t),
\end{split}
\end{align}
where we have used Young's inequality in the last line for $\delta>0$.
For the last term in \eqref{e.localenergy}, we decompose the pressure 
as $q=q_1+q_2$ as in Corollary \ref{prop.pressurer}. 
Then we have 
\begin{align}
\label{e.04}
\begin{split}
&\int^t_0 \int_{\R^3_+} q \nabla \chi_{x_0}^2 \cdot v dxds\\
\le\ &
\|q_1\|_{L^{\frac32}(0,t;L^{\frac32}(\cu(x_0)))}
\|\chi_{x_0} v\|_{L^3(0,t;L^{3}(\cu(x_0)))}
+
\|q_2\|_{L^{\frac54}(0,t;L^2(\cu(x_0)))}\|\chi_{x_0} v\|_{L^4(0,t; L^2(\R^3_+))}
\\
\le\ &
CT^{\frac{1}{12}}\left(\|v_0\|_{L^2_{uloc}} +\alpha(t) +\beta(t)\right)\alpha(t)^{\frac14}(\alpha(t)+\beta(t))^{\frac14}
+
CT^{\frac14} \|v_0\|_{L^2_{uloc}} \alpha(t).
\end{split}
\end{align}
Here we have used the estimate 
\begin{equation}
\label{e.05}
\|\chi_{x_0} v\|_{L^3(0,t;L^{3}(\cu(x_0)))} \le T^{\frac{1}{12}}
\alpha(t)^{\frac14}(\alpha(t)+\beta(t))^{\frac14},
\end{equation}
which is easily verified by interpolation as in \eqref{e.gn}.
Applying estimates \eqref{e.02}, \eqref{e.03} and \eqref{e.04} to \eqref{e.localenergy}, 
we can find a constant $C>0$ such that for any $T \in (0, \min\{1, T_*\})$  
\begin{align}\label{e.aprioriE(T)}
E(T) 
\le 
\|u_0\|_{L^2_{uloc}}^2 + CT^{\frac{1}{24}}(1+E(T)+E(T)^3),
\end{align}
where $T_*$ is the maximal existence time given in Proposition \ref{prop.localmild}.
Let $M>1$ be a constant satisfying $\|e^{-t\mathbf{A}}u_0\|_{L^2_{uloc}}^2 \le M \|u_0\|_{L^2_{uloc}}^2$ and 
define 
$$
T_0:=\sup\left\{T>0;\ E(T)\le 2M\|u_0\|_{L^2_{uloc}}^2  \right\}>0.
$$ 
By the continuity of $E$, 
 we must have $E(T_0)=2M\|v_0\|_{L^2_{uloc}}^2$. Therefore
it follows from \eqref{e.aprioriE(T)} that 
\begin{align}
2M\|u_0\|_{L^2_{uloc}}^2 
\le
\|u_0\|_{L^2_{uloc}}^2 + CT_0^{\frac{1}{24}}(1+2M\|u_0\|_{L^2_{uloc}}^2+8M^3\|u_0\|_{L^2_{uloc}}^6),
\end{align}
which leads to the following uniform bound in $\ep$
$$
T_0 \ge \left( \frac{(2M-1)\|u_0\|_{L^2_{uloc}}^2 }{C(1+2M\|u_0\|_{L^2_{uloc}}^2+8M^3\|u_0\|_{L^2_{uloc}}^6)} \right)^{24}.
$$
This completes the proof
\end{proof}

\begin{remark}
\rm{
Note that $\beta(t) \le \|\nabla v\|_{L^2(0,t;L^2_{uloc}(\R^3_+))}^2$. The quantity in the right hand side, however, has no reason to control $\beta(t)$ in general. The main difference between the two quantities is seen in the following property: from $\nabla v\in L^2(0,t;L^2_{uloc}(\R^3_+))$ it is easy to get that for almost every $s\in(0,t)$, $\nabla v(\cdot,s)\in L^2_{uloc}(\R^3_+)$, while this property is not clear in general when just $\beta$ is controlled. Therefore, special care is needed, for example at the beginning of Section \ref{sec.global.weak}.
}
\end{remark}

Another way to reformulate Proposition \ref{prop.vunif} is as follows.

\begin{corollary}\label{prop.apriori}
For all $\delta>0$, there exist $T>0$ and $A_{T,\delta}\geq 1$ such that for all $u_0\in\mathcal L^2_{uloc,\sigma}(\R^3_+)$, for all local energy weak solution $u$ to \eqref{e.nse} in the sense of Definition \ref{def.weakles} on $\R^3_+\times (0,T)$ with initial data $u_0$, if $\|u_0\|_{L^2_{uloc}(\R^3_+)}\leq \delta$, then
\begin{equation}\label{e.apriorilen'}
\sup_{\eta\in\Z^3_+}\int_{\cu(\eta)}|u(\cdot,t)|^2+\int_0^t\int_{\cu(\eta)}|\nabla u|^2+\left(\int_0^t\int_{\cu(\eta)}|u|^3\right)^\frac23\leq A_{T,\delta}.
\end{equation}
\end{corollary}

In other words, for $\delta>0$, there exists $T>0$ such that Assumption \ref{assu.A} holds.

\subsection{Convergence to the weak solutions}

In this subsection we complete the proof of Proposition \ref{prop.local}. Here
we follow the compactness argument used in \cite{KS07} in principle except for
some estimates of the velocity and the pressure. 
Before giving the proof, we first describe their strategy here.
We first consider the regularized problem \eqref{e.rnse} in the unit cube $\cu(0)$
and apply the compactness result to pass to the limit of some subsequence of 
the solutions.  We then apply similar argument in the bigger cubes $n\cu(0)$ for $n=2,3\cdots$.
Note that our pressure is defined only locally in the cube $\cu(x_0)$ for each $x_0$. 
Therefore we have to glue them appropriately to define it in $\R^3_+$.  
To this end we first derive the uniform (in $\ep$) bounds of $v^\ep$
and an appropriate pressure $q^{n,\ep}$ in $n\cu(0)$. 
In what follows we denote $n\cu(0)$ by $n\cu$ for simplicity.

\begin{proposition}
For $\ep>0$,  $n=1,2,\cdots$, and $u_0 \in L^2_{uloc}(\R^3_+)$, there exist  
constants $T_0>0$ and $A$ depending only on $\|u_0\|_{L^2_{uloc}}$
and exists a pair $(v, q^n)=(v^{\ep}, q^{n,\ep})$  satisfying the 
following statements:
\begin{enumerate}
\item  
$(v, q^n)$ is a solution to \eqref{e.rnse} in $[0,T_0) \times n\cu$ and satisfies
\begin{align}
\label{e.v1}
\sup_{t\in (0,T_0)}\int_{n\cu}|v(\cdot,t)|^2
+ \int_0^{T_0}\int_{n\cu}|\nabla v|^2
&\leq C(n) A,
\\
\label{e.v2}
\|\partial_t v \|_{(L^5(0,T_0;W^{1,3}_0(n\cu)))^*} \le C(n)A.
\end{align}
Here $C(n)$ is a constant depending only on $n$, and 
$(L^5(0,T_0;W^{1,3}_0(n\cu)))^*$ stands for the dual 
of the space $L^5(0,T_0;W^{1,3}_0(n\cu))$.
\item For any $\delta>0$ there exists a constant 
$C(\delta)$ such that 
\begin{align}
\label{e.q_ng}
\|\nabla q^n \|_{L^\frac32(\delta,T_0;L^{\frac98}(\cu(x_0) \cap n\cu))}
\leq\ C(\delta)A.
\end{align}
\item There exist $q^n_1$ and $q^n_2$ such that 
 $q^n$ can be decomposed as
$$
q^n=q^n_1+q^n_2,
$$
and the following estimate holds:
\begin{equation}
\label{e.q^n}
\left\|q^n_1\right\|_{L^\frac32(0,T_0;L^\frac32(n\cu))}+\|q^n_2\|_{L^{\frac54}(0,T_0;L^2(n\cu))}
\leq  
C(n)A.
\end{equation}
\end{enumerate}

\end{proposition}
\begin{proof}
Let $v=v^{\ep}$ be the mild solution given in Proposition \ref{prop.localmild}. 
We then easily see that \eqref{e.v1} follows from Proposition \ref{prop.vunif}. 
To see \eqref{e.q^n}, consider
the rescaling $v \mapsto v_{(n)}(x,t)=nv(nx,n^2t)$, 
$q \mapsto q_{(n)}(x,t) = n^2 q(nx,n^2t)$ and $v_0 \mapsto v_{0(n)}(x)=nv_0(nx)$ 
for $n=1,2, \cdots$.
Then $v_{(n)}$ is a mild solution to the problem 
\begin{equation}
\label{e.rnsen}
\left\{ 
\begin{aligned}
& \partial_t v_{(n)}+F_{\frac{\ep}{n}}(v_{(n)}) \cdot\nabla v_{(n)}-\Delta v_{(n)}
+\nabla q_{(n)} = 0, \quad \nabla\cdot v_{(n)}=0 & \mbox{in} & \ (0,{\textstyle\frac{T_0}{n}})\times\R^3_+,  \\
& v_{(n)} = 0  & \mbox{on} & \ (0,{\textstyle\frac{T_0}{n}})\times\partial\R^3_+,\\
&v_{(n)}|_{t=0}=v_{0(n)}
& \mbox{in}  & \ \R^3_+.
\end{aligned}
\right.
\end{equation} 
By the estimate \eqref{e.v1}, $v_{(n)}$ is uniformly bounded in $\ep>0$ 
with respect to the local energy norm. 
Therefore from Corollary \ref{prop.pressurer}  
one can find a pressure $q_{(n)}=q_{(n),1}+q_{(n),2}$ satisfying 
the estimate
\begin{multline}\label{e.q}
\left\|q_{(n), 1}\right\|_{L^\frac32(0,T_0/n;L^\frac32(\cu(0)))}+\|q_{(n), 2}\|_{L^p(0,T_0/n;L^2(\cu(0)))}
\leq\  
C\|v_{0(n)}\|_{L^2_{uloc}(\R^3_+)}
\\+ C\left\|v_{(n)}\right\|_{L^\infty(0,T_0/n;L^2_{uloc}(\R^3_+))}^2
+C\sup_{\eta \in \Z^3_+}
\left\|\nabla v_{(n)}\right\|_{L^2(0,T_0/n;L^2(\cu(\eta)))}^2.
\end{multline}
rescaling back $q_{(n),i}$ ($i=1,2$) and defining the pressure as
$q^n_i(x,t):=\frac{1}{n^2}q_{(n)i}(\frac{x}{n}, \frac{t}{n^2})$, from \eqref{e.q} 
we obtain the estimate \eqref{e.q^n}. 
As for \eqref{e.q_ng}, 
since $\nabla q^n =\nabla q_{(x_0)}$ in $\cu(x_0) \cap n\cu$ by \eqref{e.rnse},
the estimate follows from \eqref{e.qg}.

It remains to show \eqref{e.v2}. Acting the test function 
$\varphi \in C^\infty_{0}(n\cu(0))^3$ to 
\eqref{e.rnse} and using the pressure decompostion
we have
\begin{align*}
&\left| \int^{T_0}_0\int_{n\cu}
\partial_tv \cdot \varphi dxdt
\right|
\\
=\ &\left| \int^{T_0}_0\int_{n\cu}(-\nabla v\cdot \nabla \varphi +v\otimes F_{\ep}(v) 
\nabla \varphi +p^n \nabla \cdot \varphi dxdt \right|
\\
 \le\ &C\left(\int^{T_0}_0\int_{n\cu}|\nabla v|^2dxdt\right)^\frac12 
\left(\int^{T_0}_0\int_{n\cu}| \nabla \varphi|^2dxdt\right)^{\frac12}
\\
&+
C\left(\int^{T_0}_0\int_{n\cu}| v|^3dxdt\right)^\frac13 
\left(\int^{T_0}_0\int_{n\cu}| F_{\ep}(v)|^3dxdt\right)^\frac13 
\left(\int^{T_0}_0\int_{n\cu}| \nabla \varphi|^3dxdt\right)^{\frac13}
\\
&+
\left(\int^{T_0}_0\int_{n\cu}|q^n_1|^{\frac32}dxdt\right)^\frac23 
\left(\int^{T_0}_0\int_{n\cu}| \nabla \varphi|^3dxdt\right)^{\frac13}
\\
&+
\left(\int^{T_0}_0(\int_{n\cu}|q^n_2|^{2}dx)^{\frac58}dt\right)^\frac45 
\left(\int^{T_0}_0(\int_{n\cu}| \nabla \varphi|^2dx)^{\frac52}dt\right)^{\frac15}.
\end{align*}
Since $\omega_\ep$ is supported in $B(0,\ep)$ for $\ep >0$, we  see from \eqref{e.05}
\begin{align}
\label{e.FL^3}
\int^{T_0}_0\int_{n\cu}| F_{\ep}(v)|^3dxdt
\le 
\int^{T_0}_0\int_{(n+\ep)\cu}| v|^3dxdt
\le
C(n)A.
\end{align}
Hence combining this and estimates \eqref{e.v1} and \eqref{e.q^n} we obtain
\begin{align}
\left|\int^{T_0}_0\int_{n\cu}
\partial_tv \cdot \varphi dxdt\right|
\le&
C(n)A \left(\int^{T_0}_0(\int_{n\cu}| \nabla \varphi|^3dx)^\frac53 dt\right)^{\frac15}.
\end{align}
This yields \eqref{e.v2} as desired.
\end{proof}
In order to complete the proof of Proposition \ref{prop.local} we argue by induction in $n$ 
to pass to the limit for $(v^\ep, p^\ep)$ ($\ep>0$). 
For $n=1$ by using estimates \eqref{e.v1},\eqref{e.v2} and \eqref{e.q^n} one 
can apply the Aubin-Lions lemma, and then using
the uniform bound in $L^{10/3}(0,T_0;L^{10/3}(\cu))$ we can extract
 a sequence $\{(v^{k}, q^{1,k})\}_{k=1}^\infty$ from $\{(v^\ep,q^{1,\ep}) \}_{\ep>0}$
such that 
\begin{align}
\ \ &v^{k} \xrightharpoonup{*} v^{(1)}   \qquad {\rm in} \ L^\infty(0,T_0;L^2(\cu)),\notag
\\
&v^{k} \rightharpoonup v^{(1)} \qquad {\rm in} \ L^2(0,T_0;W^{1,2}(\cu)),\notag
\\
&v^{k} \rightarrow v^{(1)} \qquad {\rm in} \ L^3(0,T_0;L^3(\cu)),
\label{e.vL^3}
\\&q^{1,k} \rightharpoonup q^{(1)}\quad \ {\rm in} \  L^{\frac54}(\delta,T_0;W^{1,\frac98}(\cu))
\ \ {\rm for \ any} \ \delta \in (0,T_0), \notag
\\
&q^{1,k}_{1} \rightharpoonup q^{(1)}_1 \quad \ {\rm in} \  L^{\frac32}(0,T_0;L^{\frac32}(\cu)),\notag
\\
&q^{1,k}_{2} \rightharpoonup q^{(1)}_2 \quad \  {\rm in} \  L^{\frac54}(0,T_0;L^2(\cu)), \notag
\end{align}
where $q^{1,k}$ and $q^{(1)}$ are decomposed as $q^{1,k}=q^{1,k}_1+q^{1,k}_2$
and $q^{(1)}=q^{(1)}_1+q^{(1)}_2$ respectively. 
From \eqref{e.FL^3} and \eqref{e.vL^3}, we also deduce
\begin{align}
F_k(v^{k}) \rightarrow v^{(1)} \qquad {\rm in} \ L^3(0,T_0;L^3(\cu(R))) \qquad {\rm for \ any} \ R<1.
\end{align}
Then $(v^{(1)},q^{(1)})$ satisfies \eqref{e.nse} 
in the sense of distributions and the local energy inequality in $(0,T_0) \times \cu$.
So we let $(u,p)=(v^{(1)},q^{(1)})$ in $(0,T_0) \times \cu$. 

For $n=2$ by the same argument we can find a subsequence of 
$\{v^{k}\}_{k=1}^{\infty}$ still denoted by $\{v^{k}\}_{k=1}^{\infty}$ 
and $\{ q^{2,k}\}_{k=1}^\infty$ such that 
\begin{align}
&v^{k} \xrightharpoonup{*} v^{(2)}   \qquad {\rm in}
\ L^\infty(0,T_0;L^2(2\cu)),\notag
\\
&v^{k} \rightharpoonup v^{(2)} \qquad {\rm in} \ L^2(0,T_0;W^{1,2}(2\cu)),\notag
\\
&v^{k} \rightarrow v^{(2)} \qquad {\rm in}  \ L^3(0,T_0;L^3(2\cu)),
\notag
\\
&q^{2,k} \rightharpoonup q^{(2)} \ \quad {\rm in}  \  L^{\frac54}(\delta,T_0; W^{1,\frac98}(2\cu))
\ \ {\rm for \ any} \ \delta \in (0,T_0),
\notag
\\
&q^{2,k}_{1} \rightharpoonup q^{(2)}_1  \ \quad {\rm in}  \  L^{\frac32}(0,T_0;L^{\frac32}(2\cu)),\notag
\\
&q^{2,k}_{2} \rightharpoonup q^{(2)}_2 \ \quad {\rm in}  \  L^{\frac54}(0,T_0;L^2(2\cu)),\notag
\\
&F_{k}(v^{k}) \rightarrow v^{(2)} \qquad {\rm in}  \ L^3(0,T_0;L^3(\cu(R))) \ \ {\rm for \ any} \ R<2,
\notag
\end{align}
where $v^{(2)}$, $q^{(2)}:=q^{(2)}_1+q^{(2)}_{2}$ satisfy  \eqref{e.nse} 
in the sense of distributions and also satisfy the local energy inequality in $(0,T_0) \times 2\cu$.
Moreover \eqref{e.q_ng} implies that for any $x_0 \in 2\cu$ 
\begin{align*}
\|\nabla q^{(2)} \|_{L^\frac32(\delta,T_0;L^{\frac98}(\cu(x_0) \cap 2\cu))}
\leq\ C(\delta)A.
\end{align*}
Since $v^{(2)}=u$ in $(0,T_0) \times \cu$, we may extend $u$ by letting $u=v^{(2)}$ in $(0,T_0) \times 2\cu$.
On the other hand,  we have from \eqref{e.rnse} that $\nabla q^{(2)}=\nabla p$ in 
$(0,T_0) \times \cu$. Hence there exists a function $h^{(2)} \in L^{5/4}(0,T_0)$ such that 
$q^{(2)}(x,t)=p(x,t)-h^{(2)}(t)$ for $(x,t) \in (0,T_0) \times \cu$. 
Therefore we let $p=q^{(2)}-h^{(2)}$ in  
$(0,T_0) \times 2\cu$. 

Repeating this procedure for $n=3,4,\cdots$, we obtain 
$u \in L^\infty(0,T_0;L^2_{uloc,\sigma}(\R^3_+))$ and
$p \in L^{\frac32}_{loc}((0,T_0) \times \overline{\R^3_+})$ satisfying
\eqref{e.nse} and the local energy inequality \eqref{e.locenineq}.
Moreover by the construction the pair $(u,p)$ satisfies
\begin{align} 
\sup_{x\in \R^3_+} \int_0^{T_0} \| \nabla u \|_{L^2 (B(x)\cap \R^3_+)}^2 d t + \sup_{x\in \R^3_+} 
\int_\delta^{T_0} \| \nabla p \|_{L^\frac98 (B(x)\cap \R^3_+)}^\frac32 d t<\infty
\quad 
{\rm for \ any} \ \delta \in (0,T_0),
\notag
\\
\label{e.ut}
\|\partial_t u \|_{(L^5(0,T_0;W^{1,3}_0(n\cu)))^*} <\infty \quad {\rm for \ any} \ n=1,2,\cdots,
\end{align}
\begin{align}
\label{e.leix}
\int_{\R^3_+} \varphi|u(t)|^2 dx +&2 \int^t_0\int_{\R^3_+} \varphi|\nabla u|^2dxds 
\\
&\le \int_{\R^3_+}\varphi |u_0|^2dx +\int ^t_0 \int_{\R^3_+} |u|^2\Delta \varphi 
+\nabla \varphi \cdot u(|u|^2+2p)dxds
\notag
\end{align} 
for all $t\in [0,T_0)$ and $\varphi \in C^\infty_0(\overline{\R^3_+})$. 
\eqref{e.ut} and the uniform $L^2_{uloc}$ bound yield
the continuity of the function $t\mapsto \langle u(t), w \rangle_{L^2(\R^3_+)}$ in $[0,T_0)$ 
for any compactly supported function $ w \in L^2(\R^3_+)^3$.
Since $F_\ep (u_0)$ converges to $u_0$ in $L^2(K)$
for any compact set $K \subset \overline{\R^3_+}$,   
we also see $\lim_{t \rightarrow 0+} \langle u(t),w\rangle= \langle u_0, w\rangle$ 
by taking the limit in the weak formulation of \eqref{e.rnse}.
Combining this with \eqref{e.leix} we obtain
\begin{align*}
\lim_{t\rightarrow 0+} \| u(t) -u_0 \|_{L^2(K)} =0.
\end{align*} 
 This completes the proof of Proposition \ref{prop.local}.

\section{Global existence of local energy weak solutions}\label{sec.global.weak}
 
In this section we prove Theorem \ref{prop.global.weak}. We construct a global in time local energy weak solution, which is an analogue of the weak solution constructed by Lemari\'e-Rieusset \cite{lemariebook} for the whole space case $\R^3$; see also Kikuchi and Seregin \cite{KS07}. In principle, the proof proceeds as in the case of $\R^3$. Nevertheless, our proof does not rely on the weak-strong uniqueness of the local energy weak solutions which was used in \cite{KS07} for the whole space case. In fact, compared with the case of Leray-Hopf weak solutions with finite energy, the weak-strong uniqueness for local energy weak solutions is a more delicate problem and seems to require additional work in handling the pressure term whose structure is more complicated in the presence of the physical boundary than in the whole space case. In this sense our proof below is simpler than the known ones for the whole space case.

\begin{proof}[Proof of Theorem \ref{prop.global.weak}]
\noindent{\emph{Step 1.}} Let us first assume that $T<\infty$. Thanks to Proposition \ref{prop.local} we already know that there exists a local energy weak solution to \eqref{e.nse} in $Q_{T_0}$ with initial data $u_0$ for some $T_0>0$. We may assume that $T_0<T$. 
The key tool to verify the global existence is the $\ep$-regularity theorem, stated as in Theorem \ref{thm.ep.regularity}, is the spatial decay of $(u,p)$ in Theorem  \ref{prop.decayinfty}. Indeed, from Theorem \ref{prop.decayinfty} and Theorem \ref{thm.ep.regularity}, for any $\delta\in (0,T_0)$ there exists $R_\delta >0$ such that $u$ is smooth in $[\delta, T_0] \times \{x\in \R_+^3~|~|x|\geq R_\delta\}$, and in particular, one can show the regularity $u(t)\in W^{1,2}_{uloc} (\{x\in \R_+^3~|~|x|\geq R_\delta\})$ for any $t\in (0,T_0]$. On the other hand, from the definition of the local energy weak solution we have $u\in L^2(0,T_0; W^{1,2} (B_{R_\delta}(0)\cap \R^3_+))$. 
Therefore, we conclude that there exists $t_0\in (0,T_0)$ such that $u(t_0) \in W^{1,2}_{uloc} (\R^3_+) $ and $\lim_{R\rightarrow \infty} \| \vartheta_R u(t_0) \|_{L^2_{uloc}(\R^3_+)} =0$, where the latter assertion holds again from Theorem \ref{prop.decayinfty}. By the embedding property we have $u(t_0)\in L^6_{uloc,\sigma} (\R^3_+)$ and $\lim_{R\rightarrow \infty} \| \vartheta_R u(t_0) \|_{L^p_{uloc}(\R^3_+)} =0$ for any $2\leq p<6$. Fix $\epsilon>0$. By using Lemma \ref{lem.characterize}, $u(t_0)$ is decomposed as 
\begin{align*}
u(t_0) = u^{1,\epsilon} (t_0) + u^{2,\epsilon} u(t_0),
\end{align*} 
where $u^{1,\epsilon} (t_0)\in \mathcal{L}^4_{uloc,\sigma} (\R^3_+)$ with $\| u^{1,\epsilon} (t_0) \|_{L^4_{uloc}(\R^3_+)} \leq \epsilon$, and $u^{2,\epsilon} \in L^2_\sigma (\R^3_+)$.
By Proposition 7.1 in \cite{MMP17} we can construct a mild solution $u^{1,\epsilon}$ to \eqref{e.nse} in $(t_0,T) \times \R^3_+$ with initial data $u^{1,\epsilon} (t_0)$ by taking $\epsilon>0$ small enough, and $u^{1,\epsilon}$ satisfies 
\begin{align*}
\sup_{t_0<t<T} \big ( \| u^{1,\epsilon} (t) \|_{L^4_{uloc}} + t^\frac38 \| u^{1,\epsilon} (t) \|_{L^\infty} + t^\frac12 \| \nabla u^{1,\epsilon} (t) \|_{L^4_{uloc}} \big ) \leq C_* \epsilon .
\end{align*}
Moreover, $u\in C([t_0,T); \mathcal{L}^4_{uloc,\sigma}(\R^3_+))$ and by the bootstrap argument as in \cite{GHM}, $u^{1,\epsilon}$ is smooth in $t>t_0$, and for any $0<\delta\ll 1$, we have 
\begin{align*}
\sup_{t_0+\delta\leq t<T} \big ( \| \partial_t u^{1,\epsilon} (t) \|_{L^\infty (\R^3_+)} + \sum_{k=0,1,2} \| \nabla^k u^{1,\epsilon} (t) \|_{L^\infty (\R^3_+)} \big ) <\infty.
\end{align*}
Note that the associated pressure $q^{1,\epsilon}$ has the structure given in Section \ref{sec.pressureest}, and in particular, $q^{1,\epsilon}$ at least belongs to $L^\frac32_{loc} ([t_0,T) \times \overline{\R^3_+})$ (though it has more regularity up to $t=t_0$ since $u^{1,\epsilon}$ has). 
Thanks to the enough regularity, $(u^{1,\epsilon}, q^{1,\epsilon})$ satisfies the local energy equality:
\begin{align}\label{proof.prop.global.weak.1} 
\begin{split}
& \| \chi u^{1,\epsilon} (t) \|_{L^2(\R^3_+)}^2 + 2\int_{t_0}^t \| \chi \nabla u^{1,\epsilon} \|_{L^2 (\R^3_+)}^2 d s \\
& \qquad = \| \chi u^{1,\epsilon} (t') \|_{L^2(\R^3_+)}^2 \\
& \qquad \qquad + \int_{t'}^t \langle |u^{1,\epsilon}|^2, \partial_s \chi^2 + \Delta \chi^2 \rangle_{L^2 (\R^3_+)}  + \langle u^{1,\epsilon}\cdot \nabla \chi^2, |u^{1,\epsilon}|^2 + 2 q^{1,\epsilon} \rangle_{L^2 (\R^3_+)} d s
\end{split}
\end{align}
for any $\chi\in C^\infty_c( [t_0,T)\times \overline{\R^3_+})$ and all $t_0 \leq t'< t<T$.
Next we construct $(u^{2,\epsilon}, q^{2,\epsilon})$ as a weak solution to the perturbed Navier-Stokes equations
\begin{equation}
\label{e.pnse}
\left\{ 
\begin{aligned}
& \partial_t u^{2,\ep} -\Delta u^{2,\ep} +\nabla q^{2,\ep}  = - {\rm div}\, \big ( u^{2,\ep}\otimes u^{2,\ep} + u^{1,\ep} \otimes u^{2,\ep} + u^{2,\ep} \otimes u^{1,\ep} \big ),  \\
& \nabla\cdot u^{2,\ep}  =0,   \qquad  \mbox{in} ~ \ (t_0,T)\times\R^3_+, \\
& u^{2,\ep}  = 0 \qquad  \mbox{on} ~ \ (0,T)\times\partial \R^3_+,\\
& u^{2,\ep}|_{t=0}  =u^{2,\ep} (t_0) \qquad  \mbox{in} ~ \ \R^3_+.
\end{aligned}
\right.
\end{equation} 
Since $u^{1,\ep}$ has the enough regularity stronger than the scaling as stated as above, it is standard to construct a Leray-Hopf weak solution $u^{2,\ep}\in C_w ([t_0,T); L^2_\sigma (\R^3_+))\cap L^2 (t_0, T; W^{1,2}_0 (\R^3_+))$, $q^{2,\ep}\in L^1_{loc} ([t_0,T)\times \overline{\R^3_+}) + L^\frac56 ([t_0,T); L^2_{uloc} (\R^3_+))$,  satisfying 
$\lim_{t\downarrow t_0} \| u^{2,\ep} (t)-u^{2,\ep} (t_0) \|_{L^2(\R^3_+)} =0$ and the local energy inequality: 
\begin{align}\label{proof.prop.global.weak.2} 
\begin{split}
& \| \chi u^{2,\epsilon} (t) \|_{L^2(\R^3_+)}^2 + 2\int_{t'}^t \| \chi \nabla u^{2,\epsilon} \|_{L^2 (\R^3_+)}^2 d s \\
& \qquad \leq \| \chi u^{2,\epsilon} (t') \|_{L^2(\R^3_+)}^2 \\
& \qquad \qquad + \int_{t'}^t \langle |u^{2,\epsilon}|^2, \partial_s \chi^2 + \Delta \chi^2 \rangle_{L^2 (\R^3_+)}  + \langle u^{2,\epsilon}\cdot \nabla \chi^2, |u^{2,\epsilon}|^2 + 2q^{2,\epsilon} \rangle_{L^2 (\R^3_+)} d s\\
& \qquad \qquad  - 2\int_{t'}^t \langle u^{1,\ep} \cdot \nabla u^{2,\ep} + u^{2,\ep} \cdot \nabla u^{1,\ep}, u^{2,\ep} \chi^2 \rangle_{L^2 (\R^3_+)} d s
\end{split}
\end{align}
for all $t\in (t',T)$ and for $a.e.~t'\in [t_0,T)$ including $t'=t_0$, where $\chi\in C_c^\infty ([t_0,T)\times \overline{\R^3_+})$ is an arbitrary test function. 
Set 
\begin{align*}
v(t) = u^{1,\ep} (t) + u^{2,\ep} (t), \quad q(t) = q^{1,\ep} (t) + q^{2,\ep} (t) , \qquad t\in (t_0, T).
\end{align*}
Then we have $\lim_{t \downarrow t_0} \| v(t) - u(t_0) \|_{L^2 (K)}=0$ for any compact set $K\subset \overline{\R^3_+}$, 
\begin{align*}
&v\in L^\infty (t_0,T; \mathcal{L}^2_{uloc,\sigma}(\R^3_+))\cap L^2 (t_0,T; W^{1,2}_{0,uloc} (\R^3_+)),\\
&q\in L^\frac32_{loc} ([t_0,T)\times \overline{\R^3_+})+ L^\frac56_{loc} ([t_0,T); L^2_{uloc} (\R^3_+)),
\end{align*} 
and $(v,q)$ satisfies the Navier-Stokes equations in $(t_0,T)$ in the sense of distributions
\begin{align}\label{proof.prop.global.weak.3} 
\begin{split}
& \int_{t_0}^T  -\langle v, \partial_s \varphi \rangle_{L^2 (\R^3_+)} + \langle \nabla v, \nabla \varphi \rangle_{L^2 (\R^3_+)} - \langle q, {\rm div}\, \varphi \rangle_{L^2 (\R^3_+)} + \langle v \cdot \nabla v, \varphi \rangle_{L^2 (\R^3_+)} d s\\
& \qquad =\langle u (t_0), \varphi (t_0) \rangle_{L^2 (\R^3_+)}
\end{split}
\end{align}
for any $\varphi \in C_c^\infty ([t_0,T) \times \overline{\R^3_+})^3$ such that $\varphi|_{x_3=0} =0$.
We also have the weak continuity of $v$ in time.
Next we shall show that $(v,q)$ satisfies the local energy inequality:
\begin{align}\label{proof.prop.global.weak.4} 
\begin{split}
& \| \chi v (t) \|_{L^2(\R^3_+)}^2 + 2\int_{t_0}^t \| \chi \nabla v \|_{L^2 (\R^3_+)}^2 d s \\
&  \leq  \| \chi v (t_0) \|_{L^2(\R^3_+)}^2 
 + \int_{t_0}^t \langle |v|^2, \partial_s \chi^2 + \Delta \chi^2 \rangle_{L^2 (\R^3_+)}  + \langle v \cdot \nabla \chi^2, |v|^2 + 2q \rangle_{L^2 (\R^3_+)} d s
\end{split}
\end{align}
for any $\chi\in C^\infty_c( [t_0,T)\times \overline{\R^3_+})$ and for all $t\in (t_0,T)$. 
Note that $v(t_0)=u(t_0)$.
To prove \eqref{proof.prop.global.weak.4} we first choose any $t'\in (t_0,t)$ such that \eqref{proof.prop.global.weak.2} holds. Then it suffices to show \eqref{proof.prop.global.weak.4} but $t_0$ replaced by such $t'$;
then we take the limit $t'\rightarrow t_0$ and by the continuity at the initial time $t=t_0$ in the local $L^2$ topology, we obtain \eqref{proof.prop.global.weak.4}.
The advantage to take $t'>t_0$ first is that we can use the smoothness of $u^{1,\ep}$ in $[t',T)$, which justifies the computation below:
\begin{align*}
 E_{t,t'}:=\ &\| \chi v (t) \|_{L^2(\R^3_+)}^2 + 2\int_{t'}^t \| \chi \nabla v \|_{L^2 (\R^3_+)}^2 d s \\
 =\ & \| \chi u^{1,\ep} (t) \|_{L^2 (\R^3_+)}^2 + \| \chi u^{2,\ep} (t) \|_{L^2 (\R^3_+)}^2 + 2\langle \chi u^{1,\ep} (t), \chi u^{2,\ep} (t) \rangle_{L^2(\R^3_+)} \\
& + 2\int_{t'}^t \| \chi \nabla u^{1,\ep} \|_{L^2 (\R^3_+)}^2 + \| \chi \nabla u^{1,\ep} \|_{L^2 (\R^3_+)}^2 + 2 \langle \chi \nabla u^{1,\ep}, \chi \nabla u^{2,\ep} \rangle_{L^2 (\R^3_+)} d s.
\end{align*}
Hence,
\begin{align*}
E_{t,t'}  \leq\ &  \| \chi u^{1,\ep} (t') \|_{L^2(\R^3_+)}^2  + \| \chi u^{2,\ep} (t')\|_{L^2 (\R^3_+)}^2 \\
&  + \int_{t'}^t \langle |u^{1,\epsilon}|^2, \partial_s \chi^2 + \Delta \chi^2 \rangle_{L^2 (\R^3_+)}  + \langle u^{1,\epsilon}\cdot \nabla \chi^2, |u^{1,\epsilon}|^2 + 2 q^{1,\epsilon} \rangle_{L^2 (\R^3_+)} d s \\
&  + \int_{t'}^t \langle |u^{2,\epsilon}|^2, \partial_s \chi^2 + \Delta \chi^2 \rangle_{L^2 (\R^3_+)}  + \langle u^{2,\epsilon}\cdot \nabla \chi^2, |u^{2,\epsilon}|^2 + 2 q^{2,\epsilon} \rangle_{L^2 (\R^3_+)} d s\\
&   - 2 \int_{t'}^t \langle u^{1,\ep} \cdot \nabla u^{2,\ep} + u^{2,\ep} \cdot \nabla u^{1,\ep}, u^{2,\ep} \chi^2 \rangle_{L^2 (\R^3_+)} d s\\
&  +  2\langle u^{2,\ep} (t), \chi^2 u^{1,\ep} (t) \rangle_{L^2(\R^3_+)} + 4\int_{t'}^t  \langle \chi \nabla u^{1,\ep}, \chi \nabla u^{2,\ep} \rangle_{L^2 (\R^3_+)} d s.
\end{align*}
Thus we have 
\begin{align}\label{proof.prop.global.weak.5} 
\begin{split}
E_{t,t'} \leq\ &  \| \chi v (t') \|_{L^2(\R^3_+)}^2 + \int_{t'}^t \langle |v|^2, \partial_s \chi^2 + \Delta \chi^2 \rangle_{L^2 (\R^3_+)}  d s   \\
&  -2  \langle u^{2,\ep} (t'), \chi^2 u^{1,\ep} (t')\rangle_{L^2 (\R^3_+)}^2  -2\int_{t'}^t \langle u^{2,\ep},  u^{1,\ep} (\partial_s \chi^2 + \Delta \chi^2 ) \rangle_{L^2 (\R^3_+)}  d s  \\
&  + \int_{t'}^t \langle u^{1,\epsilon}\cdot \nabla \chi^2, |u^{1,\epsilon}|^2 + 2 q^{1,\epsilon} \rangle_{L^2 (\R^3_+)} + \langle u^{2,\epsilon}\cdot \nabla \chi^2, |u^{2,\epsilon}|^2 + 2 q^{2,\epsilon} \rangle_{L^2 (\R^3_+)}  d s \\
&   - 2 \int_{t'}^t \langle u^{1,\ep} \cdot \nabla u^{2,\ep} + u^{2,\ep} \cdot \nabla u^{1,\ep}, u^{2,\ep} \chi^2 \rangle_{L^2 (\R^3_+)} d s\\
&  +  2\langle u^{2,\ep} (t), \chi^2 u^{1,\ep} (t) \rangle_{L^2(\R^3_+)} + 4\int_{t'}^t  \langle  \nabla u^{2,\ep}, \chi^2 \nabla u^{1,\ep} \rangle_{L^2 (\R^3_+)} d s.
\end{split}
\end{align}
Since $u^{2,\ep}$ satisfies \eqref{e.pnse} in the sense of distributions,
we have 
\begin{align*}
& \langle u^{2,\ep} (t), \chi^2 u^{1,\ep} (t) \rangle_{L^2(\R^3_+)} -  \langle u^{2,\ep} (t'), \chi^2 u^{1,\ep} (t')\rangle_{L^2 (\R^3_+)}^2 \nonumber \\
=\ &  \int_{t'}^t \langle u^{2,\ep}, \partial_s \big ( \chi^2 u^{1,\ep} ) \rangle_{L^2 (\R^3_+)} + \langle \nabla u^{2,\ep}, \nabla (\chi^2 u^{1,\ep}) \rangle_{L^2 (\R^3_+)}  + \langle q^{2,\ep}, {\rm div}\, (\chi^2 u^{1,\ep}) \rangle_{L^2 (\R^3_+)}  d s\nonumber \\
& \quad - \int_{t'}^t \langle u^{2,\ep}\cdot \nabla u^{2,\ep}, \chi^2 u^{1,\ep} \rangle_{L^2 (\R^3_+)} d s \nonumber \\
=\ & \int_{t'}^t \langle u^{2,\ep}, u^{1,\ep} \partial_s  \chi^2  \rangle_{L^2 (\R^3_+)}  d s + \int_{t'}^t \langle u^{2,\ep}, \chi^2 (\Delta u^{1,\ep} - \nabla q^{1,\ep} - u^{1,\ep} \cdot \nabla u^{1,\ep} ) \rangle_{L^2 (\R^3_+)} d s \nonumber \\
&  - \int_{t'}^t \langle \nabla u^{2,\ep}, \nabla (\chi^2 u^{1,\ep}) \rangle_{L^2 (\R^3_+)}  + \langle q^{2,\ep}, {\rm div}\, (\chi^2 u^{1,\ep}) \rangle_{L^2 (\R^3_+)}  d s \nonumber \\
&  - \int_{t'}^t \langle u^{2,\ep}\cdot \nabla u^{2,\ep} + u^{1,\ep} \cdot \nabla u^{2,\ep} + u^{2,\ep} \cdot \nabla u^{1,\ep} , \chi^2 u^{1,\ep} \rangle_{L^2 (\R^3_+)} d s \nonumber \\
\end{align*}
and hence,
\begin{align}\label{proof.prop.global.weak.6} 
\begin{split}
&\langle u^{2,\ep} (t), \chi^2 u^{1,\ep} (t) \rangle_{L^2(\R^3_+)} -  \langle u^{2,\ep} (t'), \chi^2 u^{1,\ep} (t')\rangle_{L^2 (\R^3_+)}^2  \\
 =\ & \int_{t'}^t \langle u^{2,\ep}, u^{1,\ep} \partial_s  \chi^2  \rangle_{L^2 (\R^3_+)}  d s + \int_{t'}^t -2 \langle \nabla u^{2,\ep}, \chi^2 \nabla u^{1,\ep} \rangle_{L^2 (\R^3_+)} + \langle u^{2,\ep}, u^{1,\ep} \Delta \chi^2 \rangle_{L^2 (\R^3_+)} d s \\
&  + \int_{t'}^t \langle u^{2,\ep} \cdot \nabla \chi^2, q^{1,\ep} \rangle_{L^2 (\R^3_+)} + \langle  u^{1,\ep} \nabla \chi^2, q^{2,\ep}  \rangle_{L^2 (\R^3_+)}  d s\\
&  - \int_{t'}^t \langle  u^{1,\ep} \cdot \nabla u^{1,\ep}, \chi^2 u^{2,\ep}  \rangle_{L^2 (\R^3_+)} d s\\
&  - \int_{t'}^t \langle u^{2,\ep}\cdot \nabla u^{2,\ep} + u^{1,\ep} \cdot \nabla u^{2,\ep} + u^{2,\ep} \cdot \nabla u^{1,\ep} , \chi^2 u^{1,\ep} \rangle_{L^2 (\R^3_+)} d s.
\end{split}
\end{align}
Combining \eqref{proof.prop.global.weak.5} and \eqref{proof.prop.global.weak.6}, we obtain \eqref{proof.prop.global.weak.4} for $t'$ replaced by $t_0$, as desired.
Finally we set $v(t) = u(t)$ and $q(t)=p(t)$ for $t\in [0,t_0]$. 
It is clear that $(v,q)$ satisfies the required regularity as a local energy weak solution in $Q_T$ .
In particular, $t\mapsto \langle v(t), w \rangle_{L^2 (\R^3+)}$ is continuous in $(0,T)$ for any compactly supported $w\in L^2 (\R^3_+)^3$.
Then for any $\varphi \in C_c^\infty ((0,T) \times \overline{\R^3_+})^3$ such that $\varphi|_{x_3=0} =0$,
we have from \eqref{proof.prop.global.weak.3},
\begin{align*}
& \int_{0}^T  -\langle v, \partial_s \varphi \rangle_{L^2 (\R^3_+)} + \langle \nabla v, \nabla \varphi \rangle_{L^2 (\R^3_+)} - \langle q, {\rm div}\, \varphi \rangle_{L^2 (\R^3_+)} + \langle v \cdot \nabla v, \varphi \rangle_{L^2 (\R^3_+)} d s\\
& =\int_{t_0}^T  -\langle v, \partial_s \varphi \rangle_{L^2 (\R^3_+)} + \langle \nabla v, \nabla \varphi \rangle_{L^2 (\R^3_+)} - \langle q, {\rm div}\, \varphi \rangle_{L^2 (\R^3_+)} + \langle v \cdot \nabla v, \varphi \rangle_{L^2 (\R^3_+)} d s\\
& \quad +\int_{0}^{t_0}  -\langle u, \partial_s \varphi \rangle_{L^2 (\R^3_+)} + \langle \nabla u, \nabla \varphi \rangle_{L^2 (\R^3_+)} - \langle p, {\rm div}\, \varphi \rangle_{L^2 (\R^3_+)} + \langle u \cdot \nabla u, \varphi \rangle_{L^2 (\R^3_+)} d s\\
& = \langle u (t_0), \varphi (t_0)\rangle_{L^2 (\R^3_+)} \\
& \quad +\int_{0}^{t_0}  -\langle u, \partial_s \varphi \rangle_{L^2 (\R^3_+)} + \langle \nabla u, \nabla \varphi \rangle_{L^2 (\R^3_+)} - \langle p, {\rm div}\, \varphi \rangle_{L^2 (\R^3_+)} + \langle u \cdot \nabla u, \varphi \rangle_{L^2 (\R^3_+)} d s\\
& = 0.
\end{align*}
It remains to show the local energy inequality in $[0,T)$.
It is clear that the local energy inequality holds for $t\in [0,t_0]$ since $v=u$ on $[0,t_0]$.
When $t>t_0$ we first apply \eqref{proof.prop.global.weak.4} then we use the local energy inequality for $t=t_0$,
which gives \eqref{e.locenineq}. 

\smallskip

\noindent{\emph{Step 2.}} We now construct a solution for $T=\infty$. The proof is recursive. From Step 1 above, we know that there exists a local energy weak solution $u$ in the sense of Definition \ref{def.weakles} on the time interval $(0,1)$ starting from the initial data $u_0\in\mathcal L^2_{uloc}(\R^3_+)$. Let $N\in\N$. Assume that a local energy weak solution $u$ has been constructed on the time interval $(0,N)$. Since $u\in L^\infty(0,N;\mathcal L^2_{uloc}(\R^3_+))$, we have that for almost all $t_0\in (0,N)$, $u(\cdot,t_0)\in\mathcal L^2_{uloc}(\R^3_+)$. In particular $t_0$ can be taken arbitrarily close to $N$. Fix $t_0\in (N-\frac12,N)$ to fix the ideas. Hence, we consider the solution, in the sense of Definition \ref{def.weakles}, $\tilde{u}$ constructed in Step 1 living on the time interval $(t_0,N+1)$ such that $\tilde{u}(\cdot,t_0):=u(\cdot,t_0)$. The function which is equal to $u$ on $(0,t_0)$ and to $\tilde{u}$ on $(t_0,N+1)$ is then a local energy weak solution in the sense of Definition \ref{def.weakles} on $(0,N+1)$. This follows from the exact same arguments as in Step 1 above. The proof is complete.
\end{proof}

\section{Application to a blow-up criteria in the half-space}\label{sec.blowup}

The goal of this section is to prove the following blow-up criterium in the half-space $\R^3_+$. We recall that a point $(x_0,t)$ is called regular if $u$ is bounded in a parabolic ball $B(x_0,r)\times (t-r^2,t)$. If $(x_0,t)$ is not regular it is, by definition, singular. We say that $u$ blows-up at time $T$ if $T$ is the time of the first occurrence of a singularity.

\begin{theorem}\label{theo.blowup}
Let $u$ be a (finite) energy weak solution (i.e. a Leray-Hopf solution) to the Navier-Stokes equations \eqref{e.nse} with initial data $u_0\in L^2_\sigma(\R^3_+)$. Assume that $u$ blows-up at a finite time $T>0$. Then
\begin{equation*}
\|u(\cdot,t)\|_{L^3(\R^3_+)}\longrightarrow \infty\qquad\mbox{as}\quad 
t\rightarrow T-0.
\end{equation*}
\end{theorem}

This result is not new. It has been initially proved by Barker and Seregin in \cite{BS15}. Our goal here is to give another proof of this result, based on the existence theory of local energy weak solutions developed in our present work. Our method is strongly inspired by the one developed by Seregin in \cite{Ser12}. In this paper, Seregin proves the analogous result of the blow-up of the $L^3$ norm for blow-up solutions in the whole space. 

\begin{proof}[Proof of Theorem \ref{theo.blowup}]
The only ingredient which was missing to transpose the proof of \cite{Ser12} to the case of the half-space $\R^3_+$ is Theorem \ref{prop.decayinfty} above. In \cite{BS15} Barker and Seregin avoid the use of decay properties for local energy weak solutions by modifying the technique of proof. They directly show that the rescaled solutions (see below) strongly converge in $L^3_{x,t}$. 
Our point is to show that the technique of \cite{Ser12} based on the convergence to local energy solutions also applies to $\R^3_+$.

The proof is by contraposition. Let $T>0$. Let $u$ be a finite energy weak solution to \eqref{e.nse} on $\R^3_+\times(0,\infty)$ with initial data $u_0\in L^2_\sigma(\R^3_+)$. Assume that there exists a constant $M<\infty$ and a sequence of times $t_k\in(0,T)$, $t_k\rightarrow T$ such that for all $k\in\N$,
\begin{equation}\label{e.bdu}
\|u(\cdot,t_k)\|_{L^3(\R^3_+)}\leq M.
\end{equation}
We aim at showing that $u$ is smooth. Let us consider the space-time point $(x_0,T)$, where $x_0$ is an arbitrary point in $\overline{\R^3_+}$. We will show that $(x_0,T)$ is a regular point for $u$. There are two cases: either $x_0\in\R^3_+$ or $x_0\in\partial\R^3_+$. The first case of an interior point uses the existence of local energy Leray solutions in $\R^3$. This case has been treated in \cite{Ser12} and hence we do not repeat the argument. We concentrate on the second case of the boundary regularity. Note however that the analysis of the interior point is parallel to the analysis of a boundary point, so it is easy to adapt the arguments below to the case $x_0\in\R^3_+$. 
The strategy proceeds in three steps: (i) prove that a properly rescaled sequence of solutions converges to a local energy solution of the Navier-Stokes equations which is zero at final time and has initial data in $\mathcal L^2_{uloc}(\R^3_+)$, (ii) prove a Liouville theorem for such solutions using a backward uniqueness result for parabolic equations, (iii) conclude the proof. The key ingredient of the proof is the $\epsilon$-regularity theorem, Theorem \ref{thm.ep.regularity}. Here we focus on Step (i).  
Step (ii) has been extensively developed by Seregin and his coauthors in \cite{ISS03,ESS03,Ser07,Ser10,Ser11} to name a few and is almost identical in $\R^3_+$ and $\R^3$. The details for $\R^3_+$ are given in \cite{BS15}, so we will just sketch the argument for Step (ii). As we just explained, we assume that $x_0\in\partial \R^3_+$ in the following lines. Without loss of generality we even assume $x_0=0$.

\smallskip

\noindent{\emph{Step (i): rescaling and passing to the limit.}} 
For all $k\in\mathbb N$, for $S>0$ to be determined later, let $\lambda_k:=\sqrt{\frac{T-t_k}{S}}$. For all $S>0$, we introduce the rescaled functions $v^{(k)}$ defined as follows
\begin{equation}\label{e.defvk}
u^{(k)}(y,s):=\lambda_ku(\lambda_ky,T+\lambda_k^2s),
\end{equation}
for all $(y,s)\in\R^3_+\times (-S,\infty)$. Let us emphasize that by definition $v^{(k)}$ depends on $S$, although we do not keep track of this dependence in the notation. Since the scaling is the one leaving invariant the Navier-Stokes equations, $v^{(k)}$ is still a weak solution to the Navier-Stokes equations, though on the domain $\R^3_+\times (-S,\infty)$. Moreover, the blow-up time being the time at which the first singularity appears, $u$ is smooth on $Q_T=\R^3_+\times(0,T)$. Hence, for all $k\in\N$, $u^{(k)}$ is smooth on $\R^3_+\times(-S,0)$. It is clear that $u^{(k)}$ is a local energy solution in the sense of Definition \ref{def.weakles} on $\R^3_+\times(-S,0)$ with initial data $u(\cdot,t_k)$. By invariance of the $L^3(\R^3_+)$ norm under the Navier-Stokes scaling, we get that
\begin{equation}\label{e.unifL3}
\|u^{(k)}(\cdot,-S)\|_{L^3(\R^3_+)}=\|u(\cdot,t_k)\|_{L^3(\R^3_+)}\leq M,
\end{equation}
where $C$ is the constant in \eqref{e.bdu}. Therefore, there exists $u_{-S}\in L^3(\R^3_+)$ such that up to a subsequence (still denoted the same)
\begin{equation}\label{e.weakcvinit}
u^{(k)}(\cdot,-S)\rightharpoonup u_{-S}
\end{equation}
weakly in $L^3(\R^3_+)$. Using Corollary \ref{prop.apriori}, we see that there exists $S>0$ such that $u^{(k)}$ is uniformly bounded in the local energy norm on the time interval $(-S,0)$. Therefore, there exists a constant $0<A<\infty$ such that for all $k\in\N$, for all $s\in(-S,0)$,
\begin{equation*}
\sup_{\eta\in\Z^3_+}\int_{\cu(\eta)}|u^{(k)}(\cdot,s)|^2+\int_{-S}^0\int_{\cu(\eta)}|\nabla u^{(k)}|^2+\left(\int_{-S}^0\int_{\cu(\eta)}|u^{(k)}|^3\right)^\frac23\leq A.
\end{equation*}
It is now a standard procedure (see Section \ref{sec.locexiles}) to see that $u^{(k)}$ converges (up to a subsequence) weakly star in $L^\infty(-S,0;L^2_{loc}(\R^3_+))$, weakly in $L^2(-S,0;H^1_{loc}(\R^3_+))$ and strongly in $L^3_{loc}(\R^3_+\times(-S,0))$ to $u$, a function which satisfies all the axioms in Definition \ref{def.weakles} except one: it is unclear that the strong continuity \eqref{e.cvinit} in $L^2_{loc}(\R^3_+)$ at initial time holds. The convergence \eqref{e.cvinit}, though, is essential to transfer the decay of the initial data in $\mathcal L^2_{uloc}$ to the solution, as explained in Section \ref{sec.decay}. The mere weak continuity to the initial data is not enough for this purpose. Hence the argument has to be modified in the way discovered by Seregin \cite{Ser12}.

Following \cite{Ser12}, we decompose $u^{(k)}$ into $u^{(k)}=v^{(k)}+w^{(k)}$, where $w^{(k)}$ is the solution to the linear Stokes problem
\begin{equation}
\label{e.blup.lin}
\left\{ 
\begin{aligned}
& \partial_tw^{(k)}-\Delta w^{(k)}+\nabla q^{(k)} = 0,  \quad \nabla\cdot w^{(k)}=0 & \mbox{in} & \ (-S,0)\times\R^3_+, \\
& w^{(k)} = 0  & \mbox{on} & \ (-S,0)\times\partial \R^3_+,\\
& w^{(k)}(\cdot,-S) = u(\cdot,t_k)  & \mbox{in} & \ \R^3_+,\\
\end{aligned}
\right.
\end{equation}
and $v^{(k)}$ is the solution to the perturbed Navier-Stokes system driven by $w^{(k)}$ with zero initial data
\begin{equation}
\label{e.blup.nl}
\left\{ 
\begin{aligned}
& \partial_tv^{(k)}+(v^{(k)}+w^{(k)})\cdot\nabla (v^{(k)}+w^{(k)})-\Delta v^{(k)}+\nabla p^{(k)} = 0,  & & \\
& \nabla\cdot v^{(k)}=0 & \mbox{in} & \ (-S,0)\times\R^3_+, \\
& v^{(k)} = 0  & \mbox{on} & \ (-S,0)\times\partial \R^3_+,\\
& v^{(k)}(\cdot,-S) = 0  & \mbox{in} & \ \R^3_+.\\
\end{aligned}
\right.
\end{equation}
Since the initial data for $w^{(k)}$ belongs to $L^3_\sigma(\R^3_+)$, the existence, uniqueness and a priori bounds for \eqref{e.blup.lin} are well-known in the half-space $\R^3_+$. The results of \cite{DHP01} enable to check Assumption (A) in \cite{Giga86} for the Stokes semigroup in $\R^3_+$. Hence, there exists a unique solution $w^{(k)}\in C^0([-S,0];L^3(\R^3_+))\cap L^5(\R^3_+\times (-S,0))$ such that there exists a constant $C(M)<\infty$ ($M$ is the constant in \eqref{e.bdu}), for all $S>0$,
\begin{equation}\label{e.blup.linbds}
\|w^{(k)}\|_{L^\infty(-S,0;L^3(\R^3_+))}+\|w^{(k)}\|_{L^5(\R^3_+\times (-S,0))}\leq C(M).
\end{equation}
The uniformity in $S>0$ of the constant $C(M)$ is due to the fact that the norms in the left hand side of \eqref{e.blup.linbds} are invariant under the Navier-Stokes scaling. Notice moreover that $w^{(k)}$ is smooth, so that it is uniformly bounded in $k$ in the local energy norm, i.e. for all $S>0$, there exists a constant $C(S)>0$ such that for all $k\in N$, 
\begin{equation*}
\sup_{\eta\in\Z^3_+,\, s\in (0,S)}\int_{\cu(\eta)}|w^{(k)}(\cdot,s)|^2+\sup_{\eta\in\Z^3_+} \int_0^{s}\int_{\cu(\eta)}|\nabla w^{(k)}|^2\leq C(S).
\end{equation*}
As for $v^{(k)}$, a small modification of the a priori estimate carried out in the proof of Proposition \ref{prop.vunif} enables to show that (see display \eqref{e.aprioriE(T)}) there exists a constant $C$ such that for all $k\in\N$, for all $s\in(-S,0]$,
\begin{equation}\label{e.eqE(s)}
E_k(s)\leq C(s+S)^\frac1{24}(1+E_k(s)+E_k(s)^3),
\end{equation}
where 
\begin{equation*}
E_k(s):=\sup_{\eta\in\Z^3_+,\, s'\in (0,s)}\int_{\cu(\eta)}|v^{(k)}(\cdot,s')|^2+\sup_{\eta\in\Z^3_+} \int_0^{s}\int_{\cu(\eta)}|\nabla v^{(k)}|^2.
\end{equation*}
From \eqref{e.eqE(s)}, we deduce on the one hand that there exists $S$ (uniform in $k$) such that $E_k(0)$ is uniformly bounded in $k$, and on the other hand $E_k(s)\rightarrow 0$ when $s\rightarrow -S$. We fix now $S$ as above. It follows now that (up to a subsequence) $w^{(k)}$ converges weakly star in $L^\infty(-S,0;L^3(\R^3_+))$, weakly in $L^5(\R^3_+\times(-S,0))$, weakly in $L^2(-S,0;H^1_{loc}(\overline{\R^3_+}))$ and strongly in $L^3_{loc}(\overline{\R^3_+}\times(-S,0))$ to a function $w\in C^0([-S,0];L^3(\R^3_+))\cap L^5(\R^3_+\times (-S,0))$ solving the Stokes system \eqref{e.blup.lin} with $u_{-S}$ defined in \eqref{e.weakcvinit} as initial data. Moreover, $v^{(k)}$ being uniformly bounded in the local energy norm, $v^{(k)}$ converges (up to a subsequence) weakly star in $L^\infty(-S,0;L^2_{loc}(\overline{\R^3_+}))$, weakly in $L^2(-S,0;H^1_{loc}(\overline{\R^3_+}))$, strongly in $L^3_{loc}(\overline{\R^3_+}\times(-S,0))$ and strongly in $C^0([\delta,0];L^\frac98_{loc}(\overline{\R^3_+}))$, for all $\delta>-S$ to a local energy solution $v$ of \eqref{e.blup.nl} with $w^{(k)}$ replaced by $w$. Notice that for all $x_0\in\R^3_+$, passing to the limit on $k$ in 
\begin{equation*}
\int_{\cu(x_0)}|v^{(k)}(\cdot,s')|^2\leq E_k(s)\leq C(s+S)^\frac1{24},
\end{equation*}
with $C$ uniform in $k$, where the last inequality is due to \eqref{e.eqE(s)}, yields
\begin{equation*}
\int_{\cu(x_0)}|v(\cdot,s')|^2\leq\liminf_{k}\int_{\cu(x_0)}|v^{(k)}(\cdot,s')|^2\leq C(s+S)^\frac1{24}.
\end{equation*}
Therefore, $v(\cdot,s)$ converges strongly to $0$ in $L^2_{loc}(\overline{\R^3_+})$ when $s\rightarrow -S$. To put it in a nutshell, $\bar{u}$ defined by $\bar{u}=v+w$ is a local energy weak solution with initial data $u_{-S}\in L^3(\R^3_+)\subset\mathcal L^2(\R^3_+)$ in the sense of Definition \ref{def.weakles}. Theorem \ref{prop.decayinfty} now implies that
\begin{align}\label{e.decayatinftyblup}
\begin{split}
&\sup_{s\in(-S,0)}\sup_{\eta\in\Z^3_+}\int_{\cu(\eta)}|\vartheta_R \bar{u}(\cdot,s)|^2+\int_{-S}^0\int_{\cu(\eta)}|\vartheta_R\nabla \bar{u}|^2\\
&\qquad+\left(\int_{-S}^0\int_{\cu(\eta)}|\vartheta_R \bar{u}|^3\right)^\frac23+\left(\int_\delta^0\int_{\cu(\eta)}|\vartheta_R \bar{p}|^\frac32\right)^\frac23\stackrel{R\rightarrow \infty}{\longrightarrow}0,
\end{split}
\end{align}
for all $\delta>-S$. Estimate \eqref{e.decayatinftyblup} is the key point of the proof and the main contribution of our work for the case of $\R^3_+$.

\smallskip

It remains to prove that $\bar{u}$ vanishes at final time, i.e. $\bar{u}(\cdot,0)=0$. This is standard and could have been done directly without relying on the decomposition of $u^{(k)}$ into $v^{(k)}+w^{(k)}$. We first remark that \eqref{e.bdu} implies that 
\begin{equation}\label{e.bduT}
\|u(\cdot,T)\|_{L^3(\R^3_+)}\leq M,
\end{equation}
where $M$ is the constant in \eqref{e.bdu}. Indeed, up to a subsequence $u(\cdot,t_k)$ converges weakly in $L^3(\R^3_+)$ to $\tilde{u}\in L^3(\R^3_+)$ and 
\begin{equation*}
\left(\int_{\R^3_+}|\tilde u|^3\right)^\frac13\leq \liminf_{k \rightarrow  \infty}\left(\int_{\R^3_+}|u(\cdot,t_k)|^3\right)^\frac13\leq M.
\end{equation*}
We show now that $\tilde{u}=u(\cdot,T)$. Indeed the global in time weak Leray solution $u$ satisfies the following weak continuity property: for all $\varphi\in C^\infty_c(\R^3_+)^3$, 
\begin{equation}\label{e.contweaktime}
\int_{\R^3_+}u(\cdot,t_k)\cdot\varphi\longrightarrow\int_{\R^3_+}u(\cdot,T)\cdot\varphi,\qquad k\rightarrow\infty.
\end{equation}
Therefore, for all $\varphi\in C^\infty_c(\R^3_+)^3$,
\begin{align*}
&\left|\int_{\R^3_+}(u(\cdot,T)-\tilde{u})\cdot\varphi\right|\\
\leq\ &\left|\int_{\R^3_+}(u(\cdot,T)-u(\cdot,t_k))\cdot\varphi\right|+\left|\int_{\R^3_+}(u(\cdot,t_k)-\tilde{u})\cdot\varphi\right|\stackrel{k\rightarrow\infty}{\longrightarrow} 0,
\end{align*}
where the first term goes to zero by the continuity property \eqref{e.contweaktime}, and the second term goes to zero thanks to the weak convergence in $L^3(\R^3_+)$. This concludes the proof of \eqref{e.bduT}. The second observation is that due to the strong convergence of $u^{(k)}$ to $\bar{u}$ in $C^0([\delta,0];L^\frac98_{loc}(\overline{\R^3_+}))$ for all $\delta>-S$, we have in particular, for all $a>0$
\begin{equation}\label{e.cvstrong}
\int_{B(0,a)\cap\R^3_+}|u^{(k)}(y,0)|dy
\stackrel{k\rightarrow\infty}\longrightarrow
\int_{B(0,a)\cap\R^3_+}|\bar{u}(y,0)|dy.
\end{equation}
Let $a>0$. We thus infer from
\begin{equation*}
\frac1{a^2}\int_{B(0,a)\cap\R^3_+}|u^{(k)}(y,0)|dy\leq \left(\int_{B(0,\lambda_ka)\cap\R^3_+}|u(x,T)|^3dx\right)^\frac13
\end{equation*}
that $\frac1{a^2}\int_{B(0,a)\cap\R^3_+}|\bar{u}(y,0)|dy=0$, 
where we used \eqref{e.cvstrong} to pass to the limit in the left hand side, while we used \eqref{e.bduT} to pass to the limit in the right hand side.

\smallskip

\noindent{\emph{Step (ii): Liouville-type result.}} The goal of this section is to show that $\bar u=0$ in $\R^3_+\times(-S,0)$. We just emphasize the main steps of the proof. The arguments have already been written in details for the half-space in \cite[end of Section 5]{BS15}. They are not very different from the arguments in \cite{ESS03,Ser12}.

First, we show that $\bar{u}=0$ in $(\R^3_+\setminus B(0,2R))\times(-S,0)$ for some large $R$. The key for this is the decay estimate \eqref{e.decayatinftyblup}. From this we know that $\bar{u}$ is smooth $\R^3_+\setminus B(0,R)$ thanks to the $\ep$-regularity theorem, Theorem \ref{thm.ep.regularity} above. This gives bounds in $L^\infty((\R^3_+\setminus B(0,R))\times (-S,0))$ on $\bar{u}$ and its first-order spatial derivatives (in fact on derivatives at any order, but this is not needed). This in turn, allows to apply the backward uniqueness theorem \cite[Theorem 5.1]{ESS03} on the vorticity $\omega=\nabla\times \bar{u}$, noticing in addition that 
\begin{equation}\label{e.backuniq}
|\partial_t\omega-\Delta\omega|\leq C(|\nabla\omega|+|\omega|)
\end{equation}
and $\omega(\cdot,0)=0$. Hence, $\omega=0$ on $(\R^3_+\setminus B(0,2R))\times(-S,0)$. Fix now $s\in(-S,0)$. From $\nabla\cdot\bar{u}=0$ and $\omega=0$ outside $B(0,2R)$, we know that $\Delta \bar u(\cdot,s)=0$. Hence, $\bar u(\cdot,s)$ is analytic on $\R^3_+\setminus B(0,2R)$. Moreover, we already know that $\bar{u}(\cdot,s)$ vanishes on $\partial\R^3_+$. As in \cite{BS15}, we easily get that $\partial_3\bar u(\cdot,s)=0$ on $\partial \R^3_+\setminus B(0,2R)$, and for all $k\in\N$, $\nabla^k\bar u(\cdot,s)=0$ on $\partial \R^3_+\setminus B(0,2R)$. So we can extend $\bar u(\cdot,s)$ by symmetry on $\R^3\setminus B(0,2R)$. Since $\bar u(\cdot,s)$ is analytic and vanishes at any order on $\partial \R^3_+\setminus B(0,2R)$, we get by unique continuation that $\bar u(\cdot,s)=0$ on $\R^3\setminus B(0,2R)$ for almost every $s\in(-S,0)$. 

The second step consists in showing that $\bar u$ is zero everywhere on $\R^3_+\times (-S,0)$. This follows from localizing in the ball $B(0,4R)\cap\R^3_+$. 
Let $D$ be a smooth $C^\infty$ domain such that $B(0,3R)\cap\R^3_+\subset D\subset B(0,4R)\cap\R^3_+$. Then $\bar{u}$ solves the Navier-Stokes equations with no-slip boundary condition on $\partial D$ since we know that $\bar u$ vanishes outside $D$. 
For almost all $s_0\in (-S,0)$, $\nabla \bar u(\cdot,s_0)\in L^2(D)$, so that by the classical theory of the Navier-Stokes equations, 
$\partial_t\bar u$, $\nabla^2\bar u$ and $\nabla \bar p$ belong to $L^2(D\times (s_0,s_0+\delta_0))$, for some $\delta_0>0$. Regularity for linear systems then implies bounds on $\nabla^k\bar u$ for $k=0,1$ in $L^\infty(B(0,4R)\cap\R^3_+)\times (s_0+\kappa,s_0+\delta_0-\kappa)$, for some tiny $\kappa>0$. Hence, because of the previous bounds and noticing furthermore that \eqref{e.backuniq} holds and that $\omega=0$ on $(B(0,4R)\setminus B(0,2R))\cap\R^3_+\times (s_0+\kappa,s_0+\delta_0-\kappa)$, we can apply unique continuation across spatial boundaries \cite[Theorem 4.1]{ESS03} to get $\omega=0$ on $B(0,4R)\cap\R^3_+\times (s_0+\kappa,s_0+\delta_0-\kappa)$. This being true for almost every $s_0\in(-S,0)$, we eventually get $\omega=0$.

It remains to conclude that $\bar{u}=0$. For almost every $s\in(-S,0)$, $\omega(\cdot,s)=0$ so that $\Delta\bar u(\cdot,s)=0$ in $D$. Moreover, $\bar u(\cdot,s)=0$ on $\partial D$. Therefore, $\bar u(\cdot,s)=0$ in $D$, hence in $\R^3_+$, which concludes the proof of Step (ii).

\smallskip

\noindent{\emph{Step (iii): end of the proof.}} We claim that there exists $\rho_*>0$ such that 
\begin{align}\label{e.epcrit}
\frac{1}{\rho_*^2}\int_{-\rho_*^2}^{0} \int_{B (0,\rho_*)\cap \R^3_+} \Big ( |u|^3 + |p|^\frac32 \Big ) d x d t <\ep_*,
\end{align}
where $\ep_*>0$ is the constant given by the $\ep$-regularity theorem, Theorem  \ref{thm.ep.regularity}. Our goal in this step is to prove this claim. From Step (ii) we know that $\bar{u}=0$ on $\R^3_+\times(\delta,0]$ for $\delta\in(-S,0)$. This fact combined with the strong convergence of $u^{(k)}$ to $\bar{u}$ in $L^3_{loc}(\overline{\R^3_+}\times(-S,0))$ gives that for $k$ sufficiently large, for all $\rho<\sqrt{S}$,
\begin{equation*}
\frac{1}{(\lambda_k\rho)^2}\int_{-(\lambda_k\rho)^2}^{0} \int_{B (0,\lambda_k\rho)\cap \R^3_+}|u|^3 dxdt=\frac{1}{\rho^2}\int_{-\rho^2}^{0} \int_{B (0,\rho)\cap \R^3_+}|u^{(k)}|^3 dyds<\frac{\ep_*}3.
\end{equation*}
The pressure part is slightly more difficult to handle. Indeed, we do have bounds on spatial derivatives of the pressure, thanks to results of Section \ref{sec.pressureest}. What we are lacking are bounds on time derivatives of the pressure in order to get strong convergence of the pressure in $L^\frac32_{loc}(\overline{\R^3_+}\times(-S,0))$. The point is whether this is true or not, we do not need the pressure to converge strongly. Based on the work of Section \ref{sec.pressureest}, we can decompose the pressure $p^{(k)}$ associated to $u^{(k)}$ into a local part, which will be controlled by the local $L^3$ norm of $u$, and a linear and a nonlocal part, for which the scale invariant quantity in \eqref{e.epcrit} will be small. According to \eqref{e.decomppressure}, we decompose $p^{(k)}$ as follows
\begin{align*}
\begin{split}
p^{(k)} & =p_{li}^{(k)} + p_{loc}^{u^{(k)}\otimes u^{(k)}}+p_{nonloc}^{u^{(k)}\otimes u^{(k)}}. 
\end{split}
\end{align*}
From the semigroup estimates \cite[Proposition 5.3]{MMP17}, we immediately get that
\begin{align*}
\|p_{li}^{(k)}(\cdot,s)\|_{L^\infty(B(0,1)\times (-S,0))}\leq\ & C(S+s)^{-\frac34}\|u^{(k)}(\cdot,-S)\|_{L^2_{uloc}(\R^3_+)}\\
\leq\ & C(S+s)^{-\frac34}\|u^{(k)}(\cdot,-S)\|_{L^3(\R^3_+)},
\end{align*}
so that for all $\delta>-S$, there is a constant $C(\delta,S)<\infty$, uniform in $k$, such that $\|p_{li}^{(k)}(\cdot,s)\|_{L^\infty(\R^3_+\times(-\delta,0))}\leq C(\delta,S)$. From Proposition \ref{prop.nlpressure}, we readily have the follo\-wing $L^\infty$ bound: there is a constant $C(S)<\infty$, uniform in $k$, such that
\begin{equation*}
\|p_{nonloc}^{u^{(k)}\otimes u^{(k)}}\|_{L^\infty(B(0,1)\times (-S,0))}\leq C(S).
\end{equation*}
Therefore, for all $\rho<\min(1,\sqrt{S})$,
\begin{align*}
&\frac{1}{(\lambda_k\rho)^2}\int_{-(\lambda_k\rho)^2}^{0} \int_{B (0,\lambda_k\rho)\cap \R^3_+} |p_{li}(\cdot,s)|^\frac32+ |p_{nonloc}^{u\otimes u}|^\frac32d x d t \\
\leq\ & \frac{1}{\rho^2}\int_{-\rho^2}^{0} \int_{B (0,\rho)\cap \R^3_+} |p_{li}^{(k)}(\cdot,s)|^\frac32+ |p_{nonloc}^{u^{(k)}\otimes u^{(k)}}|^\frac32d y d s \leq C\rho^3\stackrel{\rho\rightarrow 0}\longrightarrow 0.
\end{align*}
It remains to handle the local part of the pressure. For this we proceed slightly differently than in the proof of Proposition \ref{prop.estlocpressure}. Indeed (this idea is taken from \cite[display (5.5)]{BS15}) we estimate $\chi_{\!_{\, 4}}^2 u^{(k)}\cdot \nabla u^{(k)}$ in the slightly energy subcritical space $L^\frac{12}{11}(\R^3_+\times (-S,0))$ instead of $L^\frac32(-S,0;L^\frac98(\R^3_+))$. We have
\begin{align*}
&\int_{-S}^0\left(\int_{\R^3_+}|\chi_{\!_{\, 4}}^2 u^{(k)}\cdot \nabla u^{(k)}|^\frac{12}{11}dy\right)^\frac{11}{12}ds\\
\leq\ &C\|\nabla u^{(k)}\|_{L^2(-S,0;L^2_{uloc}(\R^3_+))}^\frac32\|u^{(k)}\|_{L^\infty(-S,0;L^2_{uloc}(\R^3_+))}^\frac34\|u^{(k)}\|_{L^3(B(0,5)\times (-S,0))}^\frac14.
\end{align*}
Hence, reproducing the estimates of the proof of Proposition \ref{prop.estlocpressure}, we obtain for all $\rho<\min(1,\sqrt{S})$
\begin{equation}\label{e.estplocuL3}
 \frac{1}{\rho^2}\int_{-\rho^2}^{0} \int_{B (0,\rho)\cap \R^3_+} |p_{loc}^{u^{(k)}\otimes u^{(k)}}|^\frac32d y d s\leq C\|u^{(k)}\|_{L^3(B(0,5)\times (-S,0))}^\frac14,
\end{equation}
with $C$ uniform in $k$. The right hand side of \eqref{e.estplocuL3} goes to $0$ when $k\rightarrow\infty$. In the end, we get that for $\rho$ sufficiently small and $k$ sufficiently large, we have 
\begin{align}
\frac{1}{(\lambda_k\rho)^2}\int_{-(\lambda_k\rho)^2}^{0} \int_{B (0,\lambda_k\rho)\cap \R^3_+}|u|^3+ |p_{li}(\cdot,s)|^\frac32+|p_{loc}^{u\otimes u}|^\frac32+ |p_{nonloc}^{u\otimes u}|^\frac32d x d t  <\ep_*,
\end{align}
which proves the claim. By \eqref{e.epcrit} and Theorem \ref{thm.ep.regularity}, we conclude that $u$ is smooth in $B(0,\rho_*)\cap\R^3_+\times (-\rho_*^2,0)$.
\end{proof}

\appendix

\section{Auxiliary tools}

\subsection{The Leray projector in the half-space}
\label{sec.lerayproj}

In this appendix, we consider the case of arbitrary dimension $d\geq 2$. Using the formulas of \cite{MMP17}, we have for all $f,\, g\in C^\infty_c(\R^d_+)$,
\begin{align}\label{e.leraydivtan}
\begin{split}
&\big(\mathbb P\nabla\cdot (f\otimes g)\big)'(z',z_d)=\nabla'\cdot(f'\otimes g)+\partial_d(f_dg)-\nabla'(f_dg_d)\\
&+\frac{\nabla'}{2(-\Delta')^\frac12}\int_0^\infty\left[P(|z_d-y_d|)+P(z_d+y_d)\right]\nabla'\otimes\nabla'\cdot (f'\otimes g')(z',y_d)dy_d\\
&+\frac{(-\Delta')^\frac12\nabla'}{2}\int_0^\infty\left[P(|z_d-y_d|)+P(z_d+y_d)\right]f_dg_d(z',y_d)dy_d\\
&-\frac{\nabla'}2\int_0^{z_d}P(z_d-y_d)\nabla'\cdot(f_dg'+f'g_d)(z',y_d)dy_d\\
&+\frac{\nabla'}2\int_{z_d}^\infty P(y_d-z_d)\nabla'\cdot(f_dg'+f'g_d)(z',y_d)dy_d\\
&+\frac{\nabla'}2\int_{0}^\infty P(z_d+y_d)\nabla'\cdot(f_dg'+f'g_d)(z',y_d)dy_d.
\end{split}
\end{align}
for the tangential component and 
\begin{align}\label{e.leraydivvert}
\begin{split}
&\big(\mathbb P\nabla\cdot (f\otimes g)\big)_d(z',z_d)=-\nabla'\cdot(f_dg')\\
&+\frac{1}{2}\int_0^{z_d}\left[P(z_d-y_d)+P(z_d+y_d)\right]\left(\nabla'\otimes\nabla'\cdot (f'\otimes g')+\Delta'(f_dg_d)\right)(z',y_d)dy_d\\
&-\frac{1}{2}\int_{z_d}^\infty\left[P(y_d-z_d)-P(z_d+y_d)\right]\left(\nabla'\otimes\nabla'\cdot (f'\otimes g')+\Delta'(f_dg_d)\right)(z',y_d)dy_d\\
&+\frac{(-\Delta')^\frac12}{2}\int_0^\infty\left[P(|z_d-y_d|)+P(z_d+y_d)\right]\nabla'\cdot(f_dg'+f'g_d)(z',y_d)dy_d
\end{split}
\end{align}
for the vertical component. Here $P(\cdot)=e^{-(-\Delta')^\frac12\cdot}$ denotes as usual the Poisson semigroup. Hence it appears that the operator $\mathbb P\nabla\cdot $ can be decomposed as the sum of the following two types of terms:
\begin{equation}
\tag{type A}\label{e.termA'}
\partial_\alpha(v\otimes w)
\end{equation}
for some $\alpha\in\{1,\ldots d-1\}$, and
\begin{equation}
\tag{type B}\label{e.termB'}
m_0(D')\nabla'\otimes\nabla'\int_0^\infty\left[P(|z_d-y_d|)+P(z_d+y_d)\right]v\otimes w dy_d,
\end{equation}
where $m_0(D')$ is a (tangential) Fourier multiplier homogeneous of order $0$, which may be a matrix. 
We have used in particular the formula $\frac{-\nabla'\cdot\nabla'}{(-\Delta')^\frac12}=(-\Delta')^{\frac12}$ to see that every term in \eqref{e.leraydivtan} can be put in this form. 

\smallskip

We show the following lemma on the estimate of the nonlocal terms \eqref{e.termB} 
in the Helmholtz-Leray projection. 

\begin{lemma}\label{lem.termsB}
For any $v,\ w\in L^2_{uloc}(\R^d_+)$, we have the following decomposition
\begin{multline*}
\left(m_0(D')\nabla'\otimes\nabla'\int_0^\infty \left[P(|z_d-y_d|)+P(z_d+y_d)\right](1-\chi_{\!_{\, 4}}^2)v\otimes w(\cdot,y_d)dy_d\right)(z')\\
=\mathcal B_1(z',z_d)+\nabla'\otimes\nabla'\mathcal B_2(z',z_d).
\end{multline*}
Moreover, both $\mathcal B_1$ and $\mathcal B_2$ belong to $L^\infty((0,\infty);L^1_{uloc}(\R^{d-1}))$ and we have the following bound:
\begin{align*}
&\|\mathcal B_1(\cdot,z_d)\|_{L^1_{uloc,z'}(\R^{d-1})}+\|\mathcal B_2(\cdot,z_d)\|_{L^1_{uloc,z'}(\R^{d-1})}\\
\leq\ & C\|v\|_{L^2_{uloc}(\R^d_+)}\|w\|_{L^2_{uloc}(\R^d_+)}
\end{align*}
for almost every $z_d\in(0,\infty)$ with a constant $C(d)<\infty$.
\end{lemma}

\begin{remark}
{\rm In the paper \cite[Proposition 6.3]{MMP17} we have introduced another decomposition of the terms \eqref{e.termB} of the Helmholtz-Leray projection. Above, we have suggested another decomposition, which based on a splitting of the integral in the vertical variable (see proof below), rather than on splitting of low and high frequencies as in \cite{MMP17}. Notice that here this rough splitting is enough, since we are only considering the large scales, while in the aforementioned paper, we were considering both small and large scales.}
\end{remark}

\begin{proof}[Proof of Lemma \ref{lem.termsB}]
Below $m_0(D')$ stands for a tangential Fourier multiplier homogeneous of order $0$ which may change from line to line. Let us first concentrate on the part involving $P(z_d+y_d)$. The part involving $P(|z_d-y_d|)$ will be sketched below. So first we aim at estimating 
\begin{align*}
&\left(m_0(D')\nabla'\otimes\nabla'\int_0^\infty P(z_d+y_d)(1-\chi_{\!_{\, 4}}^2)v\otimes w(\cdot,y_d)dy_d\right)(z')\\
=\ &\left(m_0(D')\nabla'\otimes\nabla'\int_1^\infty P(z_d+y_d)(1-\chi_{\!_{\, 4}}^2)v\otimes w(\cdot,y_d)dy_d\right)(z')\\
&+\nabla'\otimes\nabla'\left(\int_0^1 m_0(D')P(z_d+y_d)(1-\chi_{\!_{\, 4}}^2)v\otimes w(\cdot,y_d)dy_d\right)(z')\\
=:\ &B_1(z',z_d)+\nabla'\otimes\nabla'B_2(z',z_d).
\end{align*}
In order to take care of the singularity near $\infty$ and $y_d=0$, we handle differently $B_1$ and $B_2$. Let $\gamma\in(0,1)$. For $B_1$, the idea is to put the derivatives on the Poisson kernel. This gives,
\begin{equation*}
B_1(z',z_d)=\left(\int_1^\infty m_0(D')\nabla'\otimes\nabla'P(z_d+y_d)(1-\chi_{\!_{\, 4}}^2)v\otimes w(\cdot,y_d)dy_d\right)(z'),
\end{equation*}
with 
$\gamma\in(0,1)$ fixed. We remark that for almost all $y_d\in(1,\infty)$, 
\begin{align*}
&\|m_0(D')\nabla'\otimes\nabla'P(z_d+y_d)(1-\chi_{\!_{\, 4}}^2)v\otimes w(\cdot,y_d)\|_{L^1_{uloc,z'}(\R^{d-1})}\\
\leq\ & C\|m_0(D')\nabla'\otimes\nabla'P(z_d+y_d)\|_{L^1_{z'}(\R^{d-1})}\|v\otimes w(\cdot,y_d)\|_{L^1_{uloc,z'}(\R^{d-1})}\\
\leq\ & \frac{C}{(z_d+y_d)^2}\|v(\cdot,y_d)\|_{L^2_{uloc,z'}(\R^{d-1})}\|w(\cdot,y_d)\|_{L^2_{uloc,z'}(\R^{d-1})}.
\end{align*}
We have also used the fact that for any tangential Fourier multiplier $m_\alpha(D')$ homogeneous of order $\alpha>-2$, for all $(y',y_d)\in\R^d_+$, the kernel $P(y',y_d)$ associated to the Poisson semigroup $P(y_d)$ satisfies the bound
\begin{equation*}
|m_\alpha(D')P(y',y_d)|\leq \frac{Cy_d}{(y_d+|y'|)^{d+\alpha}},
\end{equation*}
with a constant $C(d)<\infty$. This bound on fractional derivatives of the Poisson kernel is probably well-known. It can be easily proved by using Lemma 3.1 in \cite{MMP17}. Thus, for almost every $z_d\in(0,\infty)$
\begin{align*}
&\left\|\int_1^\infty m_0(D')\nabla'\otimes\nabla'P(z_d+y_d)(1-\chi_{\!_{\, 4}}^2)v\otimes w(\cdot,y_d)dy_d\right\|_{L^1_{uloc,z'}(\R^{d-1})}\\
\leq\ & \frac{C}{1+z_d}\|v\|_{L^2_{uloc}(\R^d_+)}\|w\|_{L^2_{uloc}(\R^d_+)}.
\end{align*}
As far as $B_2$ is concerned, we have
\begin{equation*}
B_2(z',z_d,s)=\left(\int_0^1m_0(D')P(z_d+y_d)(1-\chi_{\!_{\, 4}}^2)v\otimes w(\cdot,y_d)dy_d\right)(z').
\end{equation*}
We have for almost all $y_d\in(0,1)$, 
\begin{align*}
&\|m_0(D')P(z_d+y_d)(1-\chi_{\!_{\, 4}}^2)v\otimes w(\cdot,y_d)\|_{L^1_{uloc,z'}(\R^{d-1})}\\
\leq\ & C\|m_0(D')P(z_d+y_d)\|_{L^1_{z'}(\R^{d-1})}\|(1-\chi_{\!_{\, 4}}^2)v\otimes w(z',y_d)\|_{L^1_{uloc,z'}(\R^{d-1})}\\
\leq\ & C\|v(\cdot,y_d)\|_{L^2_{uloc,z'}(\R^{d-1})}\|w(\cdot,y_d)\|_{L^2_{uloc,z'}(\R^{d-1})}.
\end{align*}
In the end, for almost every $z_d\in(0,\infty)$,
\begin{align*}
&\left\|\int_{0}^1m_0(D')P(z_d+y_d)(1-\chi_{\!_{\, 4}}^2)v\otimes w(z',y_d)dy_d\right\|_{L^1_{uloc,z'}(\R^{d-1})}\\
\leq\ & C\|v\|_{L^2_{uloc}(\R^d_+)}\|w\|_{L^2_{uloc}(\R^d_+)}.
\end{align*}
We outline now how to deal with the part involving $P(|z_d-y_d|)$. It is handled in a similar way as the part involving $P(z_d+y_d)$. We split the integral in $y_d$ into 
\begin{equation}\label{e.splittingint}
\int_0^{z_d-1}\ldots+\int_{z_d-1}^{z_d}\ldots+\int_{z_d}^{z_d+1}\ldots+\int_{z_d+1}^\infty\ldots
\end{equation} 
and we deal with the first and the fourth integral similarly to $B_1$, while the second and the third are handled as $B_2$. This concludes the proof of the lemma.
\end{proof}

\subsection{A commutator lemma for the Helmholtz-Leray projection}

The following lemma is probably well-known, even in the half-space, but we could not find a reference. We state and prove it for the sake of completeness.

\begin{lemma}\label{lem.comm}
Let $a\in W^{1,\infty}(\R^d_+;\R)$ a Lipschitz function. Then, for all $1<p<\infty$, the commutator $[a,\mathbb P]:=a\mathbb P-\mathbb P(a\cdot)$ is bounded from $L^p(\R^d_+)$ to $L^q(\R^d_+)$ with $\frac1q=\frac1p-\frac1d$. Moreover, for $p$, $q$ and $f$ as above,
\begin{equation*} 
\|[a,\mathbb P]f\|_{L^q(\R^d_+)}\leq C\|\nabla a\|_{L^\infty}\|f\|_{L^p(\R^d_+)}.
\end{equation*}
\end{lemma} 

\begin{proof}
Let $f\in L^p(\R^d_+)$. By writing $f$ and $af$ as 
\begin{equation*}
f=\mathbb Pf+\nabla P,\qquad af=\mathbb P(af)+\nabla Q,
\end{equation*}
we see that $p$ and $q$ solve the following Neumann problems on $\R^d_+$:
\begin{equation*}
\left\{ 
\begin{aligned}
& -\Delta P =  \nabla\cdot f & \mbox{in} &\ \R^d_+, \\
& \partial_d P = f_d  & \mbox{on} &\ \partial \R^d_+. 
\end{aligned}
\right.
\end{equation*}
and
\begin{equation*}
\left\{ 
\begin{aligned}
& -\Delta Q =  \nabla\cdot (af) & \mbox{in} &\ \R^d_+, \\
& \partial_d Q = af_d  & \mbox{on} &\ \partial \R^d_+. 
\end{aligned}
\right.
\end{equation*}
Hence, we can use the Neumann function $N$ for the half-space to represent $P$ and $Q$. This gives, for all $x\in\R^d_+$
\begin{equation*}
P(x)=\int_{\R^d_+}N(x'-z',x_d,z_d)\nabla\cdot f(z',z_d)dz'dz_d,
\end{equation*}
and 
\begin{equation*}
Q(x)=\int_{\R^d_+}N(x'-z',x_d,z_d)\nabla\cdot (af)(z',z_d)dz'dz_d,
\end{equation*}
Hence, 
\begin{equation*}
a\nabla P-\nabla Q=-\int_{\R^d_+}\nabla^2 N(x-z',x_d,z_d)(a(x)-a(z))f(z)dz.
\end{equation*}
This yields the result by classical estimates on singular integral, since $|\nabla^2 N(x-z',x_d,z_d)(a(x)-a(z))|\leq C\|\nabla a\|_{L^\infty}|x-z|^{-d+1}$.
\end{proof}

\subsection{Characterization of $\mathcal{L}^p_{uloc,\sigma} (\R^3_+)$}

\begin{lemma}\label{lem.characterize} Let $1<p<\infty$.  Let $\vartheta_R$ be the cut-off used in Theorem \ref{prop.decayinfty}. Then 
\begin{multline}\label{eq.lem.characterize}
\mathcal{L}^p_{uloc,\sigma} (\R^3_+) \\
= \big \{ f\in L^p_{uloc}(\R^3_+)^3~|~{\rm div}\, f=0 ~\text{in}~\R^3_+\,, \quad  f_3=0~ \text{on}~\partial\R^3_+\,, \quad \lim_{R\rightarrow \infty} \| \vartheta_R f \|_{L^p_{uloc}(\R^3_+)} =0 \big \}\,.
\end{multline}
\end{lemma}

\begin{proof} The inclusion $\subset$ is trivial in \eqref{eq.lem.characterize}. It suffices to show the inverse inclusion. The argument is almost parallel to the whole space case proved in Kikuchi-Seregin \cite{KS07}.
Let $f\in L^p_{uloc} (\R^3_+)^3$ be such that 
$${\rm div}\, f=0 ~\text{in}~\R^3_+\,, \quad  f_3=0~ \text{on}~\partial\R^3_+\,, \quad \lim_{R\rightarrow \infty} \| \vartheta_R f \|_{L^p_{uloc}(\R^3_+)} =0.
$$
Let $\{B_j\}$ be the collection of open cubes in $\R^3_+$ such that $|B_j| =1$, $\R^3_+=\cup_j B_j$, and any $B_{j_0}$ intersects with at most $10$ numbers of the other $B_j$. Let $\{\varphi_j\}$ be the partition of unity subordinate to $\{B_j\}$: $\varphi_j\in C_0^\infty (\overline{\R^3_+})$, ${\rm supp}\, \varphi_j \cap \R^3_+ \subset B_j$, and $\sum_j \varphi_j =1$ in $\R^3_+$. 
Let $\chi_L\in C_0^\infty (\R^3)$, $L\gg 1$, be a cut-off such that $\chi_L=1$ for $|x|\leq L$, $\chi_L=1$ for $|x|\geq 2L$, and $\|\nabla^k \chi_L\|_{L^\infty} \leq C_k L^{-k}$ for each $k$.
Let $v_L^j\in W_0^{1,p} (B_j)^3$ be the solution to the divergence problem
\begin{align}
\nabla \cdot v_L^j = \varphi_j f\cdot \nabla \chi_L  - \int_{B_j} \varphi_j f \cdot \nabla \chi_L  d x \quad \text{in }~B_j ,
\end{align}
satisfying 
\begin{align*}
\| v_L^j \|_{W^{1,p} (B_j)} & \leq C \| \varphi_j f\cdot \nabla \chi_L  - \int_{B_j} \varphi_j f \cdot \nabla \chi_L d x\|_{L^p(B_j)} \\
& \leq \frac{C}{L} \| f\|_{L^p(B_j)}.
\end{align*}
Here $C$ is independent of $L$ and $j$.
Set 
\begin{align*}
v_L=\sum_j v_L^j,
\end{align*}
which satisfies 
\begin{align*}
\nabla \cdot v_L = f \cdot \nabla \chi_L \quad \text{in}~ \R^3_+ \,, \quad v_L|_{\partial\R^3_+} =0, \qquad {\rm supp} \, v_L ~\text{is compact in }~\overline{\R^3_+},
\end{align*}
and 
\begin{align*}
\| v_L\|_{L^p_{uloc}(\R^3_+)} \leq C \sup_j \| v_L^j \|_{L^p (B_j)} \leq \frac{C}{L} \| f \|_{L^p (B_j)}. 
\end{align*}
Finally we set 
\begin{align*}
u_L = \chi_L f- v_L,
\end{align*}
which satisfies $\nabla \cdot u_L=0$ in $\R^3_+$ and $u_{L,3}=0$ on $\partial\R^3_+$ in the sense of generalized trace, and the support of $u_L$ is compact, i.e., $u_L\in L^p_\sigma (\R^3_+)$. 
It is easy to see that 
$$
\lim_{L\rightarrow \infty} \| f- u_L \|_{L^p_{uloc}(\R^3_+)} \leq \lim_{L\rightarrow \infty} \big ( \| (1-\chi_L) f \|_{L^p_{uloc} (\R^3_+)} + \| v_L \|_{L^p_{uloc} (\R^3_+)}\big ) =0.
$$
Finally, since $L^p_\sigma (\R^3_+)$ is the closure of $C_{0,\sigma}^\infty (\R^3_+)$ in $L^p (\R^3_+)^3\hookrightarrow L^p_{uloc} (\R^3_+)^3$, we conclude that the right-hand side of \eqref{eq.lem.characterize} is included in its left-hand side. The proof is complete.
\end{proof}

\section*{Acknowledgement}

The authors thank Tai-Peng Tsai for many helpful comments.
The first author is partially supported by JSPS Program for Advancing Strategic International Networks
to Accelerate the Circulation of Talented Researchers, 'Development of Concentrated Mathematical Center Linking to Wisdom of  the Next Generation', which is organized by Mathematical Institute of Tohoku University.
The second author is partially supported by JSPS grants 25707005 and 17K05312.
The third author acknowledges financial support from the French Agence Nationale de la Recherche under grant ANR-16-CE40-0027-01, as well as from the IDEX of the University of Bordeaux for the BOLIDE project.

\small
\bibliographystyle{abbrv}
\bibliography{lerayuloc}

\end{document}